\documentclass{amsart}
\usepackage{amssymb,amsmath,mathrsfs,amsfonts}
\usepackage{amscd, amssymb, amsmath, amsthm, graphics,graphicx}
\usepackage{amsmath,amsfonts,amsthm,amssymb}
\usepackage{slashbox}
\usepackage{fancyhdr}
\usepackage[all]{xy}
\usepackage[colorlinks=true]{hyperref}
\usepackage[small,nohug,heads=vee]{diagrams}
\usepackage{multirow}
\usepackage{enumitem}
\diagramstyle[labelstyle=\scriptstyle]
\numberwithin{equation}{section}

\newtheorem{theorem}{Theorem}[section]
\newtheorem{lemma}[theorem]{Lemma}
\newtheorem{proposition}[theorem]{Proposition}
\newtheorem{corollary}[theorem]{Corollary}
\newtheorem{construction}[theorem]{Construction}

\newtheorem{data}[theorem]{Data}
\newtheorem{notation}[theorem]{Notation}
\newtheorem{parameter}[theorem]{Parameter}

\theoremstyle{definition}
\newtheorem{definition}[theorem]{Definition}

\newtheorem{question}[theorem]{Question}
\newtheorem{remark}[theorem]{Remark}

\begin{document}
\sloppy

\title[Virtual domination of $3$-manifolds II]{Virtual domination of $3$-manifolds II}

\author{Hongbin Sun}
\address{Department of Mathematics, Rutgers University - New Brunswick, Hill Center, Busch Campus, Piscataway, NJ 08854, USA}
\email{hongbin.sun@rutgers.edu}


\thanks{The author is partially supported by Simons Collaboration Grants 615229.}

\date{\today}
\begin{abstract}
For any closed oriented $3$-manifold $M$ with positive simplicial volume and any closed oriented $3$-manifold $N$, we prove that there exists a finite cover $M'$ of $M$ that admits a degree-$1$ map $f:M'\to N$, i.e. $M$ virtually $1$-dominates $N$. This result generalizes previous virtual domination results with closed hyperbolic domain to more general domains.
\end{abstract}

\maketitle
\vspace{-.5cm}
\section{Introduction}

In this paper, we assume all manifolds are compact, connected and oriented, unless otherwise indicated.

For two closed oriented $n$-manifolds $M, N$ and a non-zero integer $d$, we say that $M$ {\it $d$-dominates} $N$ if there is a map $f:M\to N$ of degree $d$, i.e. $f_*:H_n(M;\mathbb{Z})\to H_n(N;\mathbb{Z})$ satisfies $f_*([M])=d[N]$ for fundamental classes of $M$ and $N$. Moreover, we say that $M$ {\it dominates} $N$ if $M$ $d$-dominates $N$ for some non-zero integer $d$. Roughly speaking, $M$ dominates $N$ implies that $M$ is topologically more complicated than $N$. 

According to \cite{CT}, for two closed oriented $n$-manifolds $M$ and $N$, Gromov defined that $M\geq N$ if $M$ dominates $N$, thus defined a (non-antisymmetric) partial order on the set of closed oriented $n$-manifolds. Moreover, the following question is asked in \cite{CT}:
whether there is an easily described class $\mathscr{C}$ of closed oriented $n$-manifolds, such that for any closed oriented $n$-manifold $N$, there exists $M\in \mathscr{C}$ such that $M$ dominates $N$?

In this paper, we are interested in a more specific version of the above question in \cite{CT}.
\begin{question}\label{CT}
Which closed oriented $n$-manifold $M$ satisfies the following property? For any closed oriented $n$-manifold $N$, there is an  $M'$ in the following family 
$$\mathscr{C}=\{\text{all\ finite\ covers\ of\ }M\},$$ 
such that $M'$ dominates $N$.
\end{question}

At first, by classification theorems in dimension $1$ and $2$, it is easy to obtain a complete answer of Question \ref{CT} for $1$- and $2$-manifolds: $S^1$ and all closed oriented surfaces of genus at least $2$, respectively. 
In \cite{Gai}, Gaifullin prove that for any dimension $n$, there exists a closed oriented $n$-manifold satisfying the property in Question \ref{CT}.

In this paper, we focus on $3$-manifolds and give all $3$-manifolds satisfying the property in Question \ref{CT}.
Since we do not have a complete understanding on manifolds of dimension at least $4$, it seems hopeless to give a complete answer of Question \ref{CT} for dimension at least $4$. 

In the topic of non-zero degree maps, simplicial volume is the most powerful tool: for a map $f:M\to N$ between closed oriented manifolds of same dimension, we have 
$$\|M\|\geq |\text{deg}(f)|\cdot\|N\|.$$ 
Here $\|\cdot \|$ denotes the simplicial volume of a closed orientable manifold. For closed orientable hyperbolic $n$-manifolds, their simplicial volumes are proportional to their hyperbolic volumes, and the multiplicative constant only depends on the dimension.  
In \cite{Soma1}, Soma gave a complete criterion on whether a closed orientable $3$-manifold has positive simplicial volume: a closed orientable $3$-manifold $M$ has positive simplicial volume if and only if one of its prime summand has a hyperbolic JSJ piece.
Moreover, the simplicial volume of $M$ equals the sum of simplicial volumes of hyperbolic JSJ pieces of prime summands of $M$.

We say that a manifold {\it virtually} satisfies a certain property if it has a finite cover satisfying this property.
Note that a closed orientable manifold has positive simplicial volume if and only if it virtually has positive simplicial volume. So if a closed oriented $3$-manifold satisfies the property in Question \ref{CT}, then it must have positive simplicial volume.

In \cite{Sun2} and \cite{LS}, we proved that all closed oriented hyperbolic $3$-manifolds satisfy the property in Question \ref{CT}, and actually the non-zero degree map $f:M'\to N$ can be realized by a degree-$1$ map. 

In this paper, we prove the following result, which generalizes results in \cite{Sun2} and \cite{LS} from closed oriented hyperbolic $3$-manifolds to all closed oriented $3$-manifolds with positive simplicial volumes. This result completely answers Question \ref{CT} in dimension $3$: a closed oriented $3$-manifold satisfies the property in Question \ref{CT} if and only if it has positive simplicial volume.


\begin{theorem}\label{main1}
For any closed oriented $3$-manifold $M$ with positive simplicial volume and any closed oriented $3$-manifold $N$, there exists a finite cover $M'$ of $M$ that admits a degree-$1$ map $f:M'\to N$. In other words, any closed oriented $3$-manifold with positive simplicial volume virtually $1$-dominates all closed oriented $3$-manifolds.
\end{theorem}

Theorem \ref{main1} is actually the strongest possible result in the topic of virtual domination in dimension $3$, and it implies the following result on virtually $\pi_1$-surjective $k$-domination for any non-zero integer $k$.

\begin{corollary}
For any closed oriented $3$-manifold $M$ with positive simplicial volume, any closed oriented $3$-manifold $N$, and any non-zero integer $k$, there exists a finite cover $M'$ of $M$ that admits a $\pi_1$-surjective degree-$k$ map $f:M'\to N$. 
\end{corollary}

\begin{proof}
If $k$ is positive, we take $N'=\#^k N$; if $k$ is negative, we take $N'=\#^{-k}\bar{N}$, where $\bar{N}$ is the orientation reversal of $N$. There is a $\pi_1$-surjective degree-$k$ map $g:N'\to N$, which first maps $N'$ to the one-point union of $k$ copies of $N$ ($k$ is positive) or $-k$ copies of $\bar{N}$ ($k$ is negative), then maps each copy of $N$ or $\bar{N}$ to $N$ by identity. Theorem \ref{main1} provides a finite cover $M'$ of $M$ and a degree-$1$ map $h:M'\to N'$. Since a degree-$1$ map is always $\pi_1$-surjective, $f=g\circ h:M'\to N$ is a $\pi_1$-surjective degree-$k$ map.
\end{proof}

In a forthcoming paper \cite{Sun5}, we will  prove a version of Theorem \ref{main1} for compact oriented $3$-manifolds with nonempty tori boundary.

\bigskip

Theorem \ref{main1} has the following consequence in algebraic topology.

By following \cite{Gai}, we say that a closed oriented $n$-manifold $M$ has the universal realisation of cycles (URC) property if for any topological space $X$ and any $z\in H_n(X;\mathbb{Z})$, there exists a finite cover $M'$ of $M$ and a map $f:M'\to X$ such that $f_*[M']=kz$ for some $k\in \mathbb{Z}\setminus \{0\}$. 
In \cite{Thom}, Thom proved that for any $z\in H_n(X;\mathbb{Z})$, there exists a closed oriented $n$-manifold $N^n$ and a map $g:N^n\to X$, such that $g_*[N^n]=kz$ for some $k\in \mathbb{Z}\setminus \{0\}$. Thom's result and Theorem \ref{main1} together imply the following corollary.

\begin{corollary}
A closed oriented $3$-manifold  $M$ has the URC property if and only if $M$ has positive simplicial volume.
\end{corollary}

Theorem \ref{main1} can also be used to give alternative proofs of some recent results on virtual properties of $3$-manifolds.

In \cite{CG}, Chu and Groves proved that, for any compact irreducible $3$-manifold $M$ with positive simplicial volume and empty or tori boundary, and any finite abelian group $A$, $M$ has a finite cover $M'$ such that $A$ is a direct summand of $H_1(M';\mathbb{Z})$. Theorem \ref{main1} does not fully recover Chu and Groves' result, but it can be used to give an alternative proof for the case of closed $3$-manifolds.

\begin{corollary}
For any closed oriented $3$-manifold $M$ with positive simplicial volume and any finite abelian group $A$, $M$ has a finite cover $M'$ such that $A$ is a direct summand of $H_1(M';\mathbb{Z})$.
\end{corollary}

\begin{proof}
Let $N$ be a connected sum of lens spaces such that $H_1(N;\mathbb{Z})\cong A$. Then Theorem \ref{main1} provides a finite cover $M'$ of $M$ and a degree-$1$ map $f:M'\to N$. By a classical result in algebraic topology, $H_1(N;\mathbb{Z})\cong A$ is a direct summand of $H_1(M';\mathbb{Z})$.
\end{proof}

Let $G$ be either $\text{Isom}_+(\mathbb{H}^3)$ or $\text{Isom}_e(\widetilde{SL_2(\mathbb{R})})$, the identity component of isometry groups of $\mathbb{H}^3$ or $\widetilde{SL_2(\mathbb{R})}$ respectively. For any closed oriented $3$-manifold $M$ and any representation $\rho:\pi_1(M)\to G$, there is an associated representation volume $\text{vol}(M,\rho)\in \mathbb{R}$ defined. If $\rho$ runs over all representations from $\pi_1(M)$ to $G$, $|\text{vol}(M,\rho)|$ only takes finitely many values. The hyperbolic representation volume and Seifert volume of $M$ are defined to be the supremum of $|\text{vol}(M,\rho)|$ for $G=\text{Isom}_+(\mathbb{H}^3)$ and $\text{Isom}_e(\widetilde{SL_2(\mathbb{R})})$ respectively. A fundamental property of representation volume is similar to the property of simplicial volume. For a map $f:M\to N$ between closed oriented $3$-manifolds, we have 
$$V(M)\geq |\text{deg}(f)|\cdot V(N).$$
Here $V(\cdot)$ denotes the hyperbolic representation volume or the Seifert volume. Closed oriented hyperbolic $3$-manifolds and closed oriented Seifert $3$-manifolds with $\widetilde{SL(2,\mathbb{R})}$-geometry are models with positive hyperbolic representation volumes and positive Seifert volumes respectively.

In Theorem 1.6 (1) of \cite{DLW} and Theorem 1.5 of \cite{DLSW}, it is prove that a closed orientable $3$-manifold $M$ with positive simplicial volume virtually has positive hyperbolic representation volume and positive Seifert volume. We can deduce these two results from Theorem \ref{main1} directly.

\begin{corollary}
For any closed oriented $3$-manifold $M$ with positive simplicial volume, it has a finite cover $M'$ with positive hyperbolic representation volume and positive Seifert volume.
\end{corollary}

\begin{proof}
Let $N_1$ be any closed oriented hyperbolic $3$-manifold, and let $N_2$ be any closed oriented Seifert $3$-manifold with $\widetilde{SL(2,\mathbb{R})}$-geometry. By applying Theorem \ref{main1}, there exists a finite cover $M'$ of $M$ that $1$-dominates $N_1\# N_2$. So $M'$ $1$-dominates both $N_1$ and $N_2$. Since $N_1$ has positive hyperbolic representation volume and $N_2$ has positive Seifert volume, $M'$ has both positive hyperbolic representation volume and positive Seifert volume.
\end{proof}

In Theorem 6.4 of \cite{Thu}, Thurston proved that, if both $M$ and $N$ are closed oriented hyperbolic $3$-manifolds and $f:M\to N$ is a non-zero degree map such that 
$$\|M\|=|\text{deg}(f)|\cdot \|N\|,$$
then $f$ is homotopic to a covering map. In \cite{Soma2}, Soma proved that if both $M$ and $N$ have hyperbolic JSJ pieces and the same equation holds for a map $f:M\to N$, then $f$ is homotopic to a map that restricts to a covering map on the disjoint union of hyperbolic JSJ pieces. So in general, it wastes some simplicial volume for $M$ to dominate $N$, and we wonder whether this waste of volume can be small when passing to finite covers.

\begin{question}
Let $M,N$ be two closed oriented $3$-manifolds with positive simplicial volumes. For any $\epsilon>0$, whether there exists a finite cover $M'$ of $M$ and a non-zero degree map $f:M'\to N$ such that 
$$\|M'\|\leq (1+\epsilon) |\text{deg}(f)|\cdot \|N\|?$$
\end{question}

\bigskip

The strategy of the proof of Theorem \ref{main1} is similar to the strategy of the proof in \cite{LS}, which proves the virtual $1$-domination result for closed oriented hyperbolic $3$-manifolds. We sketch the proof of Theorem \ref{main1} in the following.

At first, by using a result in \cite{BW}, we can assume the target manifold $N$ is hyperbolic. Then we take a geometric triangulation of $N$, and we denote the $1$- and $2$-skeletons by $N^{(1)}$ and $N^{(2)}$ respectively. The triangulation of $N$ also induces a handle structure of $N$, and we denote the union of $0$-, $1$- and $2$-handles by $\mathcal{N}^{(2)}$. Since the virtual $1$-domination result of closed oriented hyperbolic $3$-manifolds was proved in \cite{LS}, we assume that $M$ is not hyperbolic and is irreducible. So $M$ contains a hyperbolic JSJ piece $M_0$, which is an oriented cusped hyperbolic $3$-manifold. Then we construct a $\pi_1$-injective map $j:Z\looparrowright M_0$ from a $2$-complex $Z$ to $M_0$. The $2$-complex $Z$ is closely related to $N^{(2)}$: it is obtained by replacing each triangle $\Delta$ in $N^{(2)}$ by a compact orientable surface $S_{\Delta}$ with a single boundary component. 
Since $j_*(\pi_1(Z))<\pi_1(M_0)<\pi_1(M)$ is a separable subgroup of $\pi_1(M)$, $M$ has a finite cover $M'$ such that it contains a codimension-$0$ submanifold $\mathcal{Z}$ with following topological description. Here $\mathcal{Z}$ is homeomorphic to the oriented $3$-manifold obtained by replacing each $2$-handle in $\mathcal{N}^{(2)}$ by $S_{\Delta}\times I$, pasting along the same annulus. Then there is a degree-$1$ proper map $\mathcal{Z}\to \mathcal{N}^{(2)}$ and it extends to a degree-$1$ map $M'\to N$.

The construction of $j:Z\looparrowright M_0$ and the proof of its $\pi_1$-injectivity invoke the machinery of good pants construction initiated by Kahn and Markovic. For the proof of virtual $1$-domination for closed hyperbolic $3$-manifolds in \cite{Sun2} and \cite{LS}, the following results from good pants construction were applied.
\begin{enumerate}
\item In \cite{KM}, Kahn and Markovic constructed $\pi_1$-injective nearly geodesic closed subsurfaces in closed oriented hyperbolic $3$-manifolds.
\item In \cite{LM}, Liu and Markovic computed panted cobordism groups of closed oriented hyperbolic $3$-manifolds.
\item In \cite{Liu}, Liu established the connection principle with homological control for closed oriented hyperbolic $3$-manifolds.
\item In \cite{Sun2}, the author proved the $\pi_1$-injectivity of $j:Z\looparrowright M$ for closed oriented hyperbolic $3$-manifolds.
 \end{enumerate}
Item (1) gives the fundamental construction of nearly geodesic closed subsurfaces in closed hyperbolic $3$-manifolds, and it is the foundation of all further developments of the good pants construction. In the construction of $j:Z\looparrowright M$ in \cite{LS}, item (3) is used to construct $j$-images of edges of $N^{(1)}\subset Z$, so that the $j$-image of the boundary of any triangle in $N^{(2)}$ satisfies an enhanced null-homologous condition in $M$. Then item (2) is used to construct the restriction of $j$ on each $S_{\Delta}$. Finally, item (4) is applied to prove the $\pi_1$-injectivity of $j$.

Unfortunately, all works in items (1)-(4) are not directly applicable for cusped hyperbolic $3$-manifolds, since the work in \cite{KM} requires the crucial condition that the hyperbolic $3$-manifold has positive injectivity radius. Recently in \cite{KW1}, Kahn and Wright successfully generalized the work in \cite{KM} to construct $\pi_1$-injective nearly geodesic closed subsurfaces in cusped hyperbolic $3$-manifolds, thus achieved item (1) for cusped hyperbolic $3$-manifolds. Then in \cite{Sun4}, the author computed the panted cobordism groups (with height control) of cusped hyperbolic $3$-manifolds, thus generalized item (2). In this paper, we need to generalize items (3) and (4) to cusped hyperbolic $3$-manifolds. 

A generalization of the connection principle with homological control (item (3)) will be established in Theorem \ref{enhancedconnectionprinciple}. The proof is similar to the closed hyperbolic $3$-manifold case in \cite{Liu}, by invoking a new connection principle in cusped hyperbolic $3$-manifolds (Theorem \ref{connectionprinciple}) proved in \cite{Sun4}. The proof of the $\pi_1$-injectivity of $j$ (item (4)) is similar to the proof of the $\pi_1$-injectivity in \cite{Sun2}. The main work here is to estimate geometry of nearly geodesic subsurfaces in cusped hyperbolic $3$-manifolds (Proposition \ref{estimateonsurface}). 

In this paper, for some results on cusped hyperbolic $3$-manifolds, the proofs of their corresponding results in closed hyperbolic $3$-manifolds work well for cusped manifolds. In this case, we skip the proofs and refer readers to references.

We summarize the organization of this paper in the following. In Section \ref{pregoodpants}, we review the good pants construction in closed and cusped hyperbolic $3$-manifolds, including constructions of nearly geodesic subsurfaces and computations of panted cobordism groups (\cite{KM}, \cite{LM}, \cite{KW1}, \cite{Sun4}). 
In Section \ref{connection_principle}, we prove the connection principle with homological control for cusped hyperbolic $3$-manifolds (Theorem \ref{enhancedconnectionprinciple}). We also prove Theorem \ref{boundingsurface}, which characterizes which null-homologous good multicurve bounds a nearly geodesic subsurface in a cusped hyperbolic $3$-manifold. In Section \ref{topo1}, we first use the good pants construction to construct a $\pi_1$-injective map $j:Z\looparrowright M$, then we construct the desired finite cover $M'$ of $M$ and the degree-$1$ map $f:M'\to N$. The $\pi_1$-injectivity of $j$ (Theorem \ref{pi1injectivity}) will be proved in Section \ref{pi1inj1}.

\bigskip
\bigskip

\section{Preliminary on the good pants construction}\label{pregoodpants}

In this section, we review the good pants construction on closed and cusped hyperbolic $3$-manifolds.

\subsection{Nearly geodesic subsurfaces in closed hyperbolic $3$-manifolds}\label{prekahnmarkovic}
In \cite{KM}, Kahn and Markovic proved the following celebrated surface subgroup theorem. This result initiates the development of the good pants construction, and it is the first step of Agol's proof of Thurston's virtual Haken and virtual fibering conjectures (\cite{Agol}).
	
	\begin{theorem}[Surface subgroup theorem \cite{KM}]\label{surface}
		For any closed hyperbolic $3$-manifold $M$,
		there exists an immersed closed hyperbolic subsurface $f\colon S\looparrowright M$,
		such that $f_*\colon \pi_1(S)\rightarrow 	\pi_1(M)$ is injective.
	\end{theorem}
	
The immersed subsurface of Kahn and Markovic is nice in geometric sense: it is built by pasting a finite collection of {\it $(R,\epsilon)$-good pants} along {\it $(R,\epsilon)$-good curves} in a nearly geodesic way. More details are given in the following.

We fix a closed oriented hyperbolic $3$-manifold $M$, a small number $\epsilon>0$ and a large number $R>0$. 

\begin{definition}\label{goodcurves}
An {\it $(R,\epsilon)$-good curve} is an oriented closed geodesic in $M$ with complex length satisfying $|{\bf l}(\gamma)-2R|<2\epsilon$. The set consisting of all $(R,\epsilon)$-good curves is denoted by ${\bold \Gamma}_{R,\epsilon}$.
\end{definition}

The complex length of $\gamma$ is defined by ${\bf l}(\gamma)=l+i\theta \in \mathbb{C}/2\pi i\mathbb{Z}$, where $l\in \mathbb{R}_{>0}$ is the length of $\gamma$, and $\theta\in \mathbb{R}/2\pi \mathbb{Z}$ is the rotation angle of the loxodromic isometry of $\mathbb{H}^3$ corresponding to $\gamma$. In this paper, we adopt the convention in \cite{KW1} that $(R,\epsilon)$-good curves have length close to $2R$, instead of the convention in \cite{KM} that $(R,\epsilon)$-good curves have length close to $R$, since estimates in \cite{KW1} will play an important role. Note that ${\bf \Gamma}_{R,\epsilon}$ is a finite set.

\begin{definition}\label{goodpants}
Let $\Sigma_{0,3}$ be the oriented topological pair of pants.
A pair of {\it $(R,\epsilon)$-good pants} is a homotopy class of immersion $\Sigma_{0,3} \looparrowright M$, denoted by $\Pi$, such that all three cuffs of $\Sigma_{0,3}$ are mapped to $(R,\epsilon)$-good curves $\gamma_1,\gamma_2,\gamma_3\in {\bold \Gamma}_{R,\epsilon}$, and the complex half length $\bold{hl}_{\Pi}(\gamma_i)$ of each $\gamma_i$ with respect to $\Pi$ satisfies
		$$\left|\bold{hl}_{\Pi}(\gamma_i)-R\right|<\epsilon.$$
We use ${\bold \Pi}_{R,\epsilon}$ to denote the finite set consisting of all $(R,\epsilon)$-good pants. 
\end{definition}

The complex half length $\bold{hl}_{\Pi}(\gamma_i)$ is defined as following. Let $s_{i-1}$ and $s_{i+1}$ be the common orthogonal segments (also called seams) from $\gamma_i$ to $\gamma_{i-1}$ and $\gamma_{i+1}$ respectively, let $\vec{v}_{i-1}$ and $\vec{v}_{i+1}$ be tangent vectors of $s_{i-1}$ and $s_{i+1}$ at their intersections with $\gamma_i$ respectively. Then the complex half length $\bold{hl}_{\Pi}(\gamma_i)$ is defined to be the complex distance from $\vec{v}_{i-1}$ to $\vec{v}_{i+1}$ along $\gamma_i$. More precisely, we have $\bold{hl}_{\Pi}(\gamma_i)=l+i\theta \in \mathbb{C}/2\pi i\mathbb{Z}$, where $l$ is the length of the oriented geodesic subsegment of $\gamma_i$ from $s_{i-1}$ to $s_{i+1}$, and $\theta$ is the angle from the parallel transport of $\vec{v}_{i-1}$ along this geodesic segment to $\vec{v}_{i+1}$. Note that $\bold{hl}_{\Pi}(\gamma_i)$ does not change if $s_{i-1}$ and $s_{i+1}$ are swapped. For a fixed $\gamma$, the value of $\bold{hl}_{\Pi}(\gamma)$ has two possibilities (depending on $\Pi$). If we know that $\gamma\in {\bf \Gamma}_{R,\epsilon}$ and $\Pi\in {\bf \Pi}_{R,\epsilon}$, $\bold{hl}_{\Pi}(\gamma)$ is uniquely determined by ${\bf l}(\gamma)$, and we denote this value by $\bold{hl}(\gamma)$ when it causes no confusion.

For $\gamma\in {\bf \Gamma}_{R,\epsilon}$, we use ${\bf  \Pi}_{R,\epsilon}(\gamma)$ to denote the set of pairs $(\Pi,c)$ such that $\Pi\in {\bf  \Pi}_{R,\epsilon}$ and $c$ is an oriented boundary component of $\Sigma_{0,3}$ that is mapped to $\gamma$ via the immersion. Each $(\Pi,c)\in {\bf  \Pi}_{R,\epsilon}(\gamma)$ is called a {\it marked $(R,\epsilon)$-good pants} based on $\gamma$.

For any $\gamma\in {\bf \Gamma}_{R,\epsilon}$, we can identify its normal bundle with $N^1(\gamma)=\mathbb{C}/({\bf l}(\gamma)\mathbb{Z}+2\pi i\mathbb{Z})$, then its half-normal bundle is defined to be 
$$N^1(\sqrt{\gamma})=\mathbb{C}/({\bf hl}(\gamma)\mathbb{Z}+2\pi i\mathbb{Z}).$$ 
Given $(\Pi,c)\in {\bf \Pi}_{R,\epsilon}(\gamma)$ and let $\gamma_i=\gamma$, then the pair of normal vectors $\vec{v}_{i-1},\vec{v}_{i+1}$ used to define ${\bf hl}_{\Pi}(\gamma_i)$ gives a unique vector ${\bf foot}_{\gamma}(\Pi,c)\in N^1(\sqrt{\gamma})$, called the {\it formal foot} of $(\Pi,c)$ on $\gamma$.

In \cite{KM}, an important innovation is that $(R,\epsilon)$-good pants are pasted along $(R,\epsilon)$-good curves with nearly $1$-shifts, rather than exactly matching seams along common cuffs. More precisely, in the nearly geodesic subsurface $S\looparrowright M$ in Theorem \ref{surface}, for any two marked $(R,\epsilon)$-good pants $(\Pi_1,c_1)\in {\bf \Pi}_{R,\epsilon}(\gamma), (\Pi_2,c_2)\in {\bf \Pi}_{R,\epsilon}(\bar{\gamma})$ such that $c_1$ is pasted with $c_2$, after identifying $N^1(\sqrt{\gamma})$ with $N^1(\sqrt{\bar{\gamma}})$ naturally, it is required that 
$$|{\bf foot}_{\gamma}(\Pi_1,c_1)-{\bf foot}_{\bar{\gamma}}(\Pi_2,c_2)-(1+\pi i)|<\frac{\epsilon}{R}\ \ \text{in}\ N^1(\sqrt{\gamma}).$$
This nearly $1$-shift is crucial for proving the injectivity of $f_*: \pi_1(S)\rightarrow \pi_1(M)$. Kahn and Markovic showed that, for any $(R,\epsilon)$-good curve $\gamma$, the feet of $(R,\epsilon)$-good pants based on $\gamma$ are nearly evenly distributed on $N^1(\sqrt{\gamma})$. So $M$ contains a large finite collection of $(R,\epsilon)$-good pants and they can be pasted together by nearly $1$-shift along $(R,\epsilon)$-good curves. Therefore, the asserted $\pi_1$-injective nearly geodesic closed subsurface can be constructed. 

\bigskip

\subsection{Nearly geodesic subsurfaces in cusped hyperbolic $3$-manifolds}\label{prekahnwright}

In \cite{KW1}, Kahn and Wright generalized Kahn and Markovic's surface subgroup theorem (Theorem \ref{surface}) to cusped hyperbolic $3$-manifolds.

\begin{theorem}\label{surfaceincusp}(Theorem 1.1 of \cite{KW1})
		Let $\Gamma<PSL_2(\mathbb{C})$ be a Kleinian group and assume that $\mathbb{H}^3/\Gamma$ has finite volume and is not compact. Then for all $K>1$, there exists $K$-quasi-Fuchsian (closed) surface subgroups in $\Gamma$.
	\end{theorem}
	
The main difficulty for proving Theorem \ref{surfaceincusp} is that, for cusped hyperbolic $3$-manifolds, good pants are not evenly distributed along good curves that run into cusps very deeply. 

We first define the height function in cusped hyperbolic $3$-manifolds. At first, by the Margulis lemma, there exists $\epsilon_0>0$ such that the subset of $M$ consisting of points of injectivity radii at most $\epsilon_0$ is a disjoint union of solid tori and cusp neighborhoods of ends of $M$ (simply called cusps). For any point of $M$ in the complement of the union of cusps, we define its height to be $0$. For any point $p$ in a cusp $C\subset M$, we define the height of $p$ to be the distance between $p$ and the boundary of $C$. We call the set of points with zero height the {\it thick part} of $M$, which is different from the usual convention. For a geodesic segment or a closed geodesic in $M$, we define its {\it height} to be the maximal height of points on it. For a pair of $(R,\epsilon)$-good pants in $M$, we define its height to be the maximal height of its three cuffs. 

For any $h>0$, we use ${\bf \Gamma}_{R,\epsilon}^{<h}$ (${\bf \Pi}_{R,\epsilon}^{<h}$) to denote the set of all $(R,\epsilon)$-good curves (the set of all $(R,\epsilon)$-good pants) with height at most $h$. We have similar notations for ${\bf \Gamma}_{R,\epsilon}^{\geq h}$ and ${\bf \Pi}_{R,\epsilon}^{\geq h}$.

To construct nearly geodesic subsurfaces in cusped hyperbolic $3$-manifolds, Kahn and Wright introduced a new geometric object called {\it $(R,\epsilon)$-good hamster wheel}.  For any positive integer $R$, let $Q_R$ be the oriented hyperbolic pants with cuff lengths $2,2$ and $2R$ respectively. The {\it $R$-perfect hamster wheel} $H_R$ is the $R$-sheet cyclic regular cover of $Q_R$ with $R+2$ boundary components, such that each cuff of $H_R$ has length $2R$. The preimage of the length-$2R$ cuff consists of $R$ components, which are called {\it inner cuffs}. The other two cuffs of $H_R$ are called {\it outer cuffs}. The {\it rungs} are the minimal length orthogeodesics between two outer cuffs, and there are $R$ rungs in $H_R$.
An {\it $(R,\epsilon)$-good hamster wheel} (simply an $(R,\epsilon)$-hamster wheel) ${\bf H}$ is a map $f:H_R\to M$ up to homotopy, such that the image of each cuff of $H_R$ lies in ${\bf \Gamma}_{R,\epsilon}$, and $f$ approximates a totally geodesic isometric immersion. See Section 2.9 of \cite{KW1} for the precise definition of $(R,\epsilon)$-hamster wheels. An $(R,\epsilon)$-good component is either an $(R,\epsilon)$-good pants or an $(R,\epsilon)$-good hamster wheel.

An {\it marked $(R,\epsilon)$-hamster wheel} is a pair $({\bf H},c)$ consists of an $(R,\epsilon)$-hamster wheel ${\bf H}$ and an oriented boundary component $c\subset \partial H_R$. Let $\gamma\in {\bf \Gamma}_{R,\epsilon}$ be the image of $c$, then we say that the marked $(R,\epsilon)$-hamster wheel $({\bf H}, c)$ is based at $\gamma\in {\bf \Gamma}_{R,\epsilon}$. 
For any marked $(R,\epsilon)$-hamster wheel $({\bf H}, c)$ based at $\gamma\in {\bf \Gamma}_{R,\epsilon}$, there is a {\it slow and constant turning normal field} ${\bf v}$ on $N^1(\sqrt{\gamma})$ that approximates the (inward) tangent direction of ${\bf H}$. If $c$ is an inner cuff mapped to $\gamma\in {\bf \Gamma}_{R,\epsilon}$, there are two other inner cuffs that are $Ce^{-\frac{R}{2}}$-close to $c$, and they give two vectors in $N^1(\gamma)$ whose distance along $\gamma$ is close to ${\bf hl}(\gamma)$. Kahn and Wright defined the {\it formal feet} ${\bf foot}_{\gamma}({\bf H},c)\in N^1(\sqrt{\gamma})$ to be an average of these two feet, and required the slow and constant turning normal field going through this formal feet. See Section 2.9 of \cite{KW1} for the precise definition of formal feet on inner cuffs. Note that $(R,\epsilon)$-good hamster wheels do not have formal feet on their outer cuffs.




Then Kahn and Wright defined the {\it $(R,\epsilon)$-well-matched condition} for pasting finitely many $(R,\epsilon)$-good components in a nearly geodesic manner. For two marked $(R,\epsilon)$-good components $(\Sigma_1,c_1)$ and $(\Sigma_2,c_2)$ such that $c_1$ and $c_2$ are mapped to $\gamma$ and $\bar{\gamma}$ respectively, Kahn and Wright require:
\begin{enumerate}
\item If both $(\Sigma_1,c_1)$ and $(\Sigma_2,c_2)$ have formal feet on $\gamma$ and $\bar{\gamma}$ respectively, it is required that $|{\bf foot}_{\gamma}(\Sigma_1,c_1)-{\bf foot}_{\bar{\gamma}}(\Sigma_2,c_2)-(1+\pi i)|<\frac{\epsilon}{R}$.
\item Otherwise, at least one of  $(\Sigma_1,c_1)$ and $(\Sigma_2,c_2)$ is an outer cuff of a hamster wheel, it is required that two slow and constant turning fields form a bend of at most $100\epsilon$.
\end{enumerate}
According to \cite{KW1}, for an $(R,\epsilon)$-good hamster wheel ${\bf H}$ and an outer cuff $c$ mapped to $\gamma\in {\bf \Gamma}_{R,\epsilon}$, we can also define ${\bf foot}_{\gamma}({\bf H},c)$ to be any vector in $N^1(\sqrt{\gamma})$ lying in the slow and constant turning vector field on $\gamma$. Unless otherwise indicated, we take ${\bf foot}_{\gamma}({\bf H},c)$ to be $\frac{\epsilon}{R}$-close to the tangent vector of the orthogeodesic from $\gamma$ to a preferred inner cuff of ${\bf H}$. These feet on outer cuffs are not formal feet, and we must use item (2) of the $(R,\epsilon)$-well-matched condition in this case.

A {\it good assembly} in a cusped hyperbolic $3$-manifold is a compact oriented subsurface (possibly with boundary) obtained by pasting finitely many $(R,\epsilon)$-good components according to the $(R,\epsilon)$-well-matched condition. In this paper, we also call a good assembly an {\it $(R,\epsilon)$-nearly geodesic subsurface}. Kahn and Wright proved that a subsurface in a cusped hyperbolic $3$-manifold arose from a good assembly is $\pi_1$-injective.

To construct a good assembly, Kahn and Wright defined another geometric object called {\it umbrella}. An umbrella ${\bf U}$ consists of a compact planar surface $U$ decomposed as a finite union of subsurfaces homeomorphic to $H_R$ and a map $f:U\to M$ to the hyperbolic $3$-manifold, such that the restriction of $f$ on each subsurface of $U$ homeomorphic to $H_R$ gives an $(R,\epsilon)$-good hamster wheel, and these $(R,\epsilon)$-good hamster wheels are $(R,\epsilon)$-well-matched with each other. Umbrellas are used to take care of the problem that feet of good pants are not evenly distributed on $N^1(\sqrt{\gamma})$ when the height of $\gamma$ is high. 

Actually, when constructing an umbrella, we always paste inner cuffs of an $(R,\epsilon)$-good hamster wheel with outer cuffs of other $(R,\epsilon)$-good hamster wheels. In this case, since outer cuffs do not have formal feet, we only require the $(100\epsilon)$-bending condition for slow and constant turning normal fields. Moreover, for two adjacent $(R,\epsilon)$-good hamsters wheel sharing $\gamma\in {\bf \Gamma}_{R,\epsilon}$, we can require that basepoints of their feet on $\gamma$ have distance $\frac{\epsilon}{R}$-close to $1$ (the proof of Theorem 3.8 in \cite{KW1}).

An marked umbrella is a pair $({\bf U},c)$ consisting of an umbrella ${\bf U}$ and an oriented boundary component $c\subset \partial U$. If $c$ is mapped to $\gamma\in {\bf \Gamma}_{R,\epsilon}$, then we say that the marked umbrella $({\bf U}, c)$ is based at $\gamma$. For any boundary component $c$ of ${\bf U}$ mapped to $\gamma\in {\bf \Gamma}_{R,\epsilon}$, it is the boundary component of a unique $(R,\epsilon)$-good hamster wheel ${\bf H}$ contained in ${\bf U}$, and we define ${\bf foot}_{\gamma}({\bf U},c)={\bf foot}_{\gamma}({\bf H},c)\in N^1(\sqrt{\gamma})$.

\bigskip

Now we summarize Kahn and Wright's construction of nearly geodesic subsurfaces in cusped hyperbolic $3$-manifolds. 

Kahn and Wright first took two cut-off heights $h_T$ and $h_c$ such that $h_T\geq 6\log{R}$ and $h_c\geq h_T+44\log{R}$ hold (Theorem 4.15 and equation (5.4.1) of \cite{KW1}). Then they considered the collection of all $(R,\epsilon)$-good pants $\Pi$ with at least one cuff has height at most $h_c$. For any pants $\Pi$ in the above collection with $\Pi\in {\bf \Pi}_{R,\epsilon}^{\geq h_c}$, it gives one or two marked pants $(\Pi,c)$ based at some $\gamma\in {\bf \Gamma}_{R,\epsilon}^{<h_c}$. In Theorem 4.15 of \cite{KW1}, for each such $(\Pi,c)$, Kahn and Wright constructed a $\mathbb{Q}_+$-linear combination of umbrellas $\hat{U}(\Pi,c,\gamma)$ with coefficients sum to $1$ such that the following hold.
\begin{enumerate}
\item As a $\mathbb{Q}_+$-linear combination of umbrellas, the boundary of $\hat{U}(\Pi,c,\gamma)$ contains one copy of $\gamma$, and all other boundary components of $\hat{U}(\Pi,c,\gamma)$ have height at most $h_T$.
\item Each component of $\hat{U}(\Pi,c,\gamma)$ is $(R,\epsilon)$-well-matched with any $(R,\epsilon)$-good pants that is $(R,\epsilon)$-well-matched with $(\Pi,c)$ along $\gamma$.
\end{enumerate} 
Then they used $\hat{U}(\Pi,c,\gamma)$ to replace $\Pi$ in the above collection of good pants. 
 
 After the replacement process, two finite $\mathbb{Q}_+$-linear combinations of $(R,\epsilon)$-good objects are obtained. The first one is the sum of all $(R,\epsilon)$-good pants in ${\bf \Pi}_{R,\epsilon}^{<h_c}$:
 $$A_0=\sum_{\Pi\in {\bf \Pi}_{R,\epsilon}^{<h_c}}\Pi,$$ 
and the second one is the sum of all $\mathbb{Q}_+$-linear combinations of umbrellas constructed above:
 $$A_1=\sum_{\gamma_\in {\bf \Gamma}_{R,\epsilon}^{<h_c}}\sum_{(\Pi,c)\in {\bf \Pi}_{R,\epsilon}^{\geq h_c}(\gamma)}\hat{U}(\Pi,c,\gamma).$$

So $A=A_0+A_1$ is a $\mathbb{Q}_+$-linear combination of $(R,\epsilon)$-good pants and umbrellas $\sum_{i=1}^Na_i{\bf \Sigma}_i$. Here each $a_i\in \mathbb{Q}_+$ and each ${\bf \Sigma}_i$ is an $(R,\epsilon)$-good pants or an umbrella. More precisely, each ${\bf \Sigma}_i$ is the homotopy class of a map $f_i:\Sigma_i\to M$ defined on an oriented planer surface $\Sigma_i$. Then for any $\gamma_0\in {\bf \Gamma}_{R,\epsilon}^{<h_c}$, we define $\partial A_{\gamma_0}$ to be the following measure on $N^1(\sqrt{\gamma_0})$: for any $B\subset N^1(\sqrt{\gamma_0})$, its measure is 
$$\partial A_{\gamma_0}(B)=\sum_{i=1}^Na_i \cdot \#\{c \ |\ c\ \text{ is\ an\ oriented}\ \partial\text{-component\ of\ }\Sigma_i, f_i(c)=\gamma_0, {\bf foot}_{\gamma}({\bf \Sigma}_i,c)\in B\}.$$

Then Kahn and Wright proved that, for any $\gamma_0\in {\bf \Gamma}_{R,\epsilon}^{<h_c}$, $\partial A_{\gamma_0}$ is close to a multiple of the Lebesgue measure on $N^1(\sqrt{\gamma_0})$. After eliminating denominators of coefficients in $A$, they could paste good pants and umbrellas in $A\cup \bar{A}$ ($\bar{A}$ denotes the orientation reversal of $A$) to get the desired $(R,\epsilon)$-nearly geodesic closed subsurface.
 

In this paper, we will use some more precise geometric description of Kahn and Wright's $(R,\epsilon)$-nearly geodesic subsurfaces, so we summarize the above discussion as the following result. The proof of this result is hidden in several proofs in \cite{KW1}, so we only sketch the proof here.

\begin{theorem}\label{cuspedprecisedescription}
For any oriented cusped hyperbolic $3$-manifold $M$, any small enough $\epsilon>0$, any large enough integer $R>0$ depending only on $M$ and $\epsilon>0$,  any constants $h_T\geq 6\log{R}$ and $h_c\geq 44 \log{R}+h_T$, the following statement holds for 
$$A=\sum_{\Pi\in {\bf \Pi}_{R,\epsilon}^{<h_c}}\Pi+\sum_{\gamma_\in {\bf \Gamma}_{R,\epsilon}^{<h_c}}\sum_{(\Pi,c)\in {\bf \Pi}_{R,\epsilon}^{\geq h_c}(\gamma)}\hat{U}(\Pi,c,\gamma).$$
\begin{enumerate}

\item For any $\gamma_0\in {\bf \Gamma}_{R,\epsilon}^{<h_c}$, let $\tau:N^1(\sqrt{\gamma_0})\to N^1(\sqrt{\gamma_0})$ be the translation by $1+\pi i$. Then for any $B\subset N^1(\sqrt{\gamma_0})$, we have $$\partial A_{\gamma_0}(N_{\frac{\epsilon}{R}}(B))\geq \partial A_{\gamma_0}(\tau(B)),$$ and $$\partial A_{\gamma_0}(N_{\frac{\epsilon}{R}}(B))\geq \partial A_{\gamma_0}(\tau(B))+1$$ if $N_{\frac{\epsilon}{R}}(B)\ne N^1(\sqrt{\gamma_0})$. Here $N_{\frac{\epsilon}{R}}(B)$ denotes the $\frac{\epsilon}{R}$-neighborhood of $B$.

\item After eliminating denominators, one can paste all boundary components of good pants and umbrellas in $A\cup\bar{A}$ in the following fashion. If ${\bf \Sigma}_1$ is pasted with ${\bf \Sigma_2}$ along boundary components $c_1$ and $c_2$ such that $c_1$ is mapped to $\gamma_0$ and $c_2$ is mapped to $\bar{\gamma}_0$, then the following inequality holds in $N^1(\sqrt{\gamma_0})$: 
\begin{align}\label{2.1}
|{\bf foot}_{\gamma_0}({\bf \Sigma_1},c_1)-{\bf foot}_{\bar{\gamma}_0}({\bf \Sigma_2},c_2)-(1+i\pi)|<\frac{\epsilon}{R}.
\end{align} 
In particular, we get an $(R,\epsilon)$-nearly geodesic closed subsurface in $M$.

\end{enumerate}
\end{theorem}

\begin{proof}
(1) We fix an $(R,\epsilon)$-good curve $\gamma_0\in  {\bf \Gamma}_{R,\epsilon}^{<h_c}$. For any marked good pants $( \Pi,c)\in {\bf \Pi}_{R,\epsilon}^{\geq h_c}(\gamma_0)$, we can choose the foot of $\hat{U}(\Pi,c,\gamma_0)$ on $\gamma_0$ such that
${\bf foot}_{\gamma_0}(\hat{U}(\Pi,c,\gamma_0),c)={\bf foot}_{\gamma_0}({\bf \Pi},c).$ By following the proof of Theorem 1.1 in \cite{KW1}, we have the following decomposition of the measure $\partial A_{\gamma_0}$: 
\begin{align}\label{2.2}
\partial A_{\gamma_0}=\partial (\sum_{\Pi\in {\bf \Pi}_{R,\epsilon}}\Pi)_{\gamma_0}+\partial_{\text{out}}(\sum_{\gamma \in {\bf \Gamma}_{R,\epsilon}^{< h_c}}\sum_{(\Pi,c)\in {\bf \Pi_{R,\epsilon}^{\geq h_c}}(\gamma)}\hat{U}(\Pi,c,\gamma))_{\gamma_0}.
\end{align}
Here $\partial_{\text{out}}$ only counts the boundary components of $\hat{U}(\Pi,c,\gamma)$ other than $\gamma$. 

Let $A'$ denotes $\sum_{\Pi\in {\bf \Pi}_{R,\epsilon}}\Pi$, then the proof of Theorem 5.1 of \cite{KW1} implies that $\partial A'_{\gamma_0}$ is evenly distributed on $N^1(\sqrt{\gamma_0})$, and in particular 
\begin{align}\label{2.3}
\partial A'_{\gamma_0}(N_{\frac{\epsilon}{2R}}(B))\geq \partial A'_{\gamma_0}(\tau(B))
\end{align}
holds for any $B\subset N^1(\sqrt{\gamma_0})$.

Theorem 5.11 of \cite{KW1} implies that the second term in equation (\ref{2.2}) is bounded above by 
\begin{align}\label{2.4}
C_1(M,\epsilon)R^{19}e^{2R-\frac{1}{2}(h_c-h_T)}\leq C_1(M,\epsilon)R^{-3}e^{2R}.
\end{align} 
The proof of Theorem 5.2 of \cite{KW1} implies that 
\begin{align}\label{2.5}
\partial A'_{\gamma_0}(D_{\frac{\epsilon}{4R}})\geq C_2(M,\epsilon)R^{-2}e^{2R}\geq C_1(M,\epsilon)R^{-3}e^{2R}+1
\end{align} 
for any disk $D_{\frac{\epsilon}{4R}}\subset N^1(\sqrt{\gamma_0})$ of radius $\frac{\epsilon}{4R}$.

If $N_{\frac{\epsilon}{R}}(B)=N^1(\sqrt{\gamma_0})$, the desired inequality in item (1) obviously holds. For any $B\subset N^1(\sqrt{\gamma_0})$ with $N_{\frac{\epsilon}{R}}(B)\ne N^1(\sqrt{\gamma_0})$, the following inequality holds:
\begin{align*}
&\partial A_{\gamma_0}(N_{\frac{\epsilon}{R}}(B))\geq \partial A'_{\gamma_0}(N_{\frac{\epsilon}{R}}(B))\geq \partial A'_{\gamma_0}(N_{\frac{\epsilon}{2R}}(B)) +\partial A'_{\gamma_0}(D_{\frac{\epsilon}{4R}})\\
\geq &\  \partial A'_{\gamma_0}(\tau(B))+C_1(M,\epsilon)R^{-3}e^{2R}+1\geq \partial A_{\gamma_0}(\tau(B))+1.
\end{align*}
Here $D_{\frac{\epsilon}{4R}}$ denotes a disk of radius $\frac{\epsilon}{4R}$ contained in $N_{\frac{\epsilon}{R}}(B)\setminus N_{\frac{\epsilon}{2R}}(B)$, which exists since $N_{\frac{\epsilon}{R}}(B)$ has nonempty boundary.
The first, third, fourth inequalities hold by equations (\ref{2.2}), (\ref{2.3}) and (\ref{2.5}), (\ref{2.4}) respectively.

(2) Note that all boundaries curves of $A$ (after eliminating denominators) lie in ${\bf \Gamma}_{R,\epsilon}^{<h_c}$. For any $\gamma_0\in {\bf \Gamma}_{R,\epsilon}^{<h_c}$, any term ${\bf \Sigma_1}$ in $A$ with $c_1\subset \partial \Sigma_1$ mapped to $\gamma_0$ and some term ${\bf \Sigma_2}$ in $\bar{A}$ with $c_2\subset \partial \Sigma_2$ mapped to $\bar{\gamma}_0$ satisfies $$|{\bf foot}_{\gamma_0}({\bf \Sigma_1},c_1)-{\bf foot}_{\bar{\gamma}_0}({\bf \Sigma_2},c_2)-(1+i\pi)|<\frac{\epsilon}{R}$$
if and only if 
$$\tau({\bf foot}_{\bar{\gamma}_0}({\bf \Sigma_2},c_2))\in N_{\frac{\epsilon}{R}}({\bf foot}_{\gamma_0}({\bf \Sigma_1},c_1))\subset N^1(\sqrt{\gamma_0}).$$ 

Recall that $\partial A_{\gamma_0}(B)$ counts the number of feet arised from $A$ that lie in $B\subset N^1(\sqrt{\gamma_0})$. By the Hall's marriage theorem, we can paste $A\cup \bar{A}$ in a way satisfying equation (\ref{2.1}) (such that terms in $A$ are pasted with terms in $\bar{A}$) exists if and only if for any $\gamma_0\in {\bf \Gamma}_{R,\epsilon}^{<h_c}$ and any $B\subset N^1(\sqrt{\gamma_0})$, the following inequality holds
$$\partial A_{\gamma_0}(\tau(B))\leq \partial A_{\gamma_0}(N_{\frac{\epsilon}{R}}(B)).$$
This is exactly the conclusion of item (1), so such an $(R,\epsilon)$-nearly geodesic closed subsurface is obtained.
\end{proof}

We need the extra $+1$ term in Theorem \ref{cuspedprecisedescription} (1) for further applications in this paper, and actually $1$ can be replaced by any fixed positive integer.

\bigskip

\subsection{The panted cobordism groups of finite volume hyperbolic $3$-manifolds}\label{prepantedcobordism}

In \cite{LM}, Liu and Markovic introduced panted cobordism groups of closed oriented hyperbolic $3$-manifolds and computed these groups. In \cite{Sun4}, the author generalized some results in \cite{LM} to oriented cusped hyperbolic $3$-manifolds. We will review these results in this section.

We first fix a closed oriented hyperbolic $3$-manifold $M$, a small number $\epsilon>0$ and a large number $R>0$. Let $\mathbb{Z}{\bf \Gamma}_{R,\epsilon}$ be the free abelian group generated by ${\bf \Gamma}_{R,\epsilon}$, modulo the relation $\gamma+\bar{\gamma}=0$ for all $\gamma\in {\bf \Gamma}_{R,\epsilon}$. Here $\bar{\gamma}$ denotes the orientation reversal of $\gamma$. Let $\mathbb{Z}{\bf \Pi}_{R,\epsilon}$ be the free abelian group generated by ${\bf \Pi}_{R,\epsilon}$, modulo the relation $\Pi+\bar{\Pi}=0$ for all $\Pi\in {\bf \Pi}_{R,\epsilon}$. By taking the oriented boundary of $(R,\epsilon)$-good pants, we get a homomorphism $\partial:\mathbb{Z}{\bf \Pi}_{R,\epsilon}\to \mathbb{Z}{\bf \Gamma}_{R,\epsilon}$. The panted cobordism group $\Omega_{R,\epsilon}(M)$ is defined as the following in \cite{LM}.

\begin{definition}\label{closedcobordismgroup} 
The {\it panted cobordism group} $\Omega_{R,\epsilon}(M)$ is defined to be the cokernel of the homomorphism $\partial$, i.e. $\Omega_{R,\epsilon}(M)$ fits into the following exact sequence 
$$\mathbb{Z}{\bf \Pi}_{R,\epsilon}\xrightarrow{\partial} \mathbb{Z}{\bf \Gamma}_{R,\epsilon}\to \Omega_{R,\epsilon}(M)\to 0.$$
\end{definition}

To state the result in \cite{LM}, we need the following definition.

\begin{definition}\label{framebundle}
For an oriented hyperbolic $3$-manifold $M$ and a point $p\in M$, a \emph{special orthonormal frame} (simply a frame) of $M$ based at $p$ is a triple of unit tangent vectors $(\vec{t}_p,\vec{n}_p,\vec{t}_p\times \vec{n}_p)$ such that $\vec{t}_p,\vec{n}_p\in T_p^1M$ with $\vec{t}_p\perp \vec{n}_p$, and $\vec{t}_p\times \vec{n}_p\in T_p^1M$ is the cross product with respect to the orientation of $M$. We use $\text{SO}(M)$ to denote the frame bundle over $M$ consisting of all special orthonormal frames of $M$.
\end{definition}

For simplicity, we denote each element in $\text{SO}(M)$ by its basepoint and the first two vectors of the frame, as $(p,\vec{t}_p,\vec{n}_p)$, since the third vector is determined by the first two. We call $\vec{t}_p$ and $\vec{n}_p$ the tangent and normal vectors of this frame, respectively.

In \cite{LM}, Liu and Markovic proved the following theorem.
	
\begin{theorem}\label{homology}(Theorem 5.2 of \cite{LM})
Given a closed oriented hyperbolic $3$-manifold $M$, for small enough $\epsilon>0$ depending on $M$ and large enough $R>0$ depending on $M$ and $\epsilon$, there is a natural isomorphism 
		$$\Phi:\Omega_{R,\epsilon}(M)\rightarrow H_1(\text{SO}(M);\mathbb{Z})$$
from the panted cobordism group of $M$ to the first homology group of $\text{SO}(M)$.
	\end{theorem}
	
In \cite{Sun4}, the author generalized Theorem \ref{homology} to oriented cusped hyperbolic $3$-manifolds, and the result needs some height conditions on involved curves and pants. 

For any $h'>h>0$, $\mathbb{Z}{\bf \Gamma}_{R,\epsilon}^{<h}$ is naturally a subset of $\mathbb{Z}{\bf \Gamma}_{R,\epsilon}^{<h'}$. For the boundary homomorphism $\partial: \mathbb{Z}{\bf \Pi}_{R,\epsilon}^{<h'}\to \mathbb{Z}{\bf \Gamma}_{R,\epsilon}^{<h'}$, we use $\mathbb{Z}{\bf \Pi}_{R,\epsilon}^{h,h'}$ to denote the preimage of $\mathbb{Z}{\bf \Gamma}_{R,\epsilon}^{<h}<\mathbb{Z}{\bf \Gamma}_{R,\epsilon}^{<h'}$ in $\mathbb{Z}{\bf \Pi}_{R,\epsilon}^{<h'}$. Then we have the following definition.

\begin{definition}\label{cuspedcobordismgroup}
For an oriented cusped hyperbolic $3$-manifold $M$ and any $h'>h>0$, we define the $(R,\epsilon)$-panted cobordism group of height $(h,h')$, denoted by $\Omega_{R,\epsilon}^{h,h'}(M)$, to be the cokernel of the homomorphism $\partial|: \mathbb{Z}{\bf \Pi}_{R,\epsilon}^{h,h'}\to  \mathbb{Z}{\bf \Gamma}_{R,\epsilon}^{<h}$. Thus $\Omega_{R,\epsilon}^{h,h'}(M)$ fits into the following exact sequence $$\mathbb{Z}{\bf \Pi}_{R,\epsilon}^{h,h'}\xrightarrow{\partial|}  \mathbb{Z}{\bf \Gamma}_{R,\epsilon}^{<h}\to \Omega_{R,\epsilon}^{h,h'}(M)\to 0 $$
\end{definition}


In \cite{Sun4}, the author proved the following analogy of Theorem \ref{homology}.

\begin{theorem}\label{cuspedhomology}
For any oriented cusped hyperbolic 3-manifold M, any numbers
$\beta>\alpha\geq 4$ with $\beta-\alpha\geq 3$ and any small enough $\epsilon>0$, there exists $R_0> 0$,
such that for any $R > R_0$, we have an isomorphism
$$\Phi:\Omega_{R,\epsilon}^{\alpha\log{R},\beta\log{R}}(M)\to H_1(\text{SO}(M); \mathbb{Z}).$$
\end{theorem}

For both Theorem \ref{homology} and Theorem \ref{cuspedhomology}, the isomorphism $\Phi$ is defined by taking {\it canonical lifts} of $(R,\epsilon)$-good curves. For any $(R,\epsilon)$-good curve $\gamma$, we fix a point $p\in \gamma$ and a frame ${\bf p}=(p,\vec{v}_p,\vec{n}_p)\in \text{SO}(M)_p$ based at $p$ such that $\vec{v}_p$ is tangent to $\gamma$. Then a canonical lift of $\gamma$ is a map $\hat{\gamma}: S^1\to \text{SO}(M)$ defined by the concatenation of following paths:
\begin{itemize}
\item First parallel transport ${\bf p}$ along $\gamma$ to another frame ${\bf p'}$ based at $p$.
\item Do $2\pi$-rotation of ${\bf p'}$ along its normal vector and return to ${\bf p'}$.
\item Connect ${\bf p'}$ back to ${\bf p}$ along a $(2\epsilon)$-short path in $\text{SO}_p(M)$.
\end{itemize}
In \cite{LM}, it is proved that the homology class of $\hat{\gamma}$ is independent of the choice of $p$ and ${\bf p}$.
So if we compose the isomorphism $\Phi:\Omega_{R,\epsilon}^{\alpha\log{R},\beta\log{R}}(M)\to H_1(\text{SO}(M); \mathbb{Z})$ with the homomorphism $\pi_*:H_1(\text{SO}(M);\mathbb{Z})\to H_1(M;\mathbb{Z})$ induced by bundle projection, $\pi_*\circ \Phi$ maps each $(R,\epsilon)$-good multicurve to its homology class in $H_1(M;\mathbb{Z})$.

For any small $\epsilon>0$ and large $R>0$, an {\it $(R,\epsilon)$-panted subsurface} in a hyperbolic $3$-manifold $M$ is a (possibly disconnected) compact oriented surface $F$ with a pants decomposition and an immersion $i: F\looparrowright M$, such that the restriction of $i$ to each pair of pants in the pants decomposition of $F$ gives a pair of $(R,\epsilon)$-good pants.

The $(R,\epsilon)$-panted subsurface was originally defined by Liu and Markovic in \cite{LM}, and it does not require any feet-matching condition when two $(R,\epsilon)$-good pants are pasted along an $(R,\epsilon)$-good curve.
Geometrically, an $(R,\epsilon)$-multicurve $L\in \mathbb{Z}{\bf \Gamma}_{R,\epsilon}^{<h}$ represents the trivial element in $\Omega_{R,\epsilon}^{h,h'}(M)$ if and only if it bounds an $(R,\epsilon)$-panted subsurface of height at most $h'$.

\bigskip
\bigskip

\section{A connection principle with homological control}\label{connection_principle}

In general, the connection principle is a machinery in hyperbolic geometry that constructs geodesic segments and $\partial$-framed segments (see definition below) in hyperbolic $3$-manifolds. For closed hyperbolic $3$-manifolds, the idea of connection principle was initiated in Proposition 4.4 of \cite{KM}, which relies on the exponential mixing property of frame flow  (\cite{Mor} and \cite{Pol}).
The first formally stated connection principle was given in Lemma 4.15 of \cite{LM}, and the first homological version of connection principle was given in Theorem 3.1 of \cite{Liu}. For cusped hyperbolic $3$-manifolds, 
key idea of the connection principle was given in \cite{KW1} and \cite{KW2}, and the author proved such a connection principle in Theorem 4.1 of \cite{Sun4}. 

In this section, we first define the main geometric object: $\partial$-framed segments, give some geometric estimates and review the connection principle of cusped hyperbolic $3$-manifolds in \cite{Sun4} (Theorem \ref{connectionprinciple}). Then we prove the homological version of connection principle of cusped hyperbolic $3$-manifolds (Theorem \ref{enhancedconnectionprinciple}). The proof of Theorem \ref{enhancedconnectionprinciple} is similar to the corresponding result of closed hyperbolic $3$-manifolds in \cite{Liu}. In the proof of Theorem \ref{enhancedconnectionprinciple}, we provide details of all geometric constructions, since they invoke the new connection principle (Theorem \ref{connectionprinciple}). On the other hand, when checking homological conditions, the readers are referred to proofs in \cite{Liu}, since those proofs work well for cusped hyperbolic $3$-manifolds.
We also use the connection principle (Theorem \ref{connectionprinciple}) and Theorem \ref{cuspedprecisedescription} to prove Theorem \ref{boundingsurface}, which implies that certain null-homologous $(R,\epsilon)$-multicurve in a cusped hyperbolic $3$-manifold bounds a nearly geodesic subsurface.

\subsection{Geometric estimates on $\partial$-framed segments and the connection principle}\label{partialframedsegment}

At first, we need to define oriented $\partial$-framed segments and their associated objects, since they are geometric objects that will be constructed by the connection principle.
\begin{definition}\label{segment} 
An {\it oriented $\partial$-framed segment} in $M$ is a triple 
$$\mathfrak{s}=(s,\vec{n}_{\text{ini}},\vec{n}_{\text{ter}}),$$ 
such that $s$ is an immersed oriented compact geodesic segment (simply called a geodesic segment), $\vec{n}_{\text{ini}}$ and $\vec{n}_{\text{ter}}$ are unit normal vectors of $s$ at its initial and terminal points respectively. 
\end{definition}

We have the following objects associated to an oriented $\partial$-framed segment $\mathfrak{s}$:
\begin{itemize}
\item The {\it carrier segment} of $\mathfrak{s}$ is the oriented geodesic segment $s$, and the {\it height} of $\mathfrak{s}$ is the height of $s$.	

\item The {\it initial endpoint} $p_{\text{ini}}(\mathfrak{s})$ and the {\it terminal point} $p_{\text{ter}}(\mathfrak{s})$ are the initial and terminal points of $s$ respectively.

\item The {\it initial framing} $\vec{n}_{\text{ini}}(\mathfrak{s})$ and the {\it terminal framing} $\vec{n}_{\text{ter}}(\mathfrak{s})$ are the unit normal vectors $\vec{n}_{\text{ini}}$ and $\vec{n}_{\text{ter}}$ respectively.

\item The {\it initial direction} $\vec{t}_{\text{ini}}(\mathfrak{s})$ and the {\it terminal direction} $\vec{t}_{\text{ter}}(\mathfrak{s})$ are the unit tangent vectors of $s$ at $p_{\text{ini}}(\mathfrak{s})$ and $p_{\text{ter}}(\mathfrak{s})$ respectively.

\item The {\it initial frame} and the {\it terminal frame} of $\mathfrak{s}$ are $(p_{\text{ini}}(\mathfrak{s}),\vec{t}_{\text{ini}}(\mathfrak{s}),\vec{n}_{\text{ini}}(\mathfrak{s}))$ and $(p_{\text{ter}}(\mathfrak{s}),\vec{t}_{\text{ter}}(\mathfrak{s}),\vec{n}_{\text{ter}}(\mathfrak{s}))$ respectively.

\item The {\it length} $l(\mathfrak{s})\in(0,\infty)$ of $\mathfrak{s}$ is the length of the geodesic segment $s$, the {\it phase} $\varphi(\mathfrak{s})\in \mathbb{R}/2\pi \mathbb{Z}$ of $\mathfrak{s}$ is the angle from the parallel transport of $\vec{n}_{\text{ini}}$ along $s$ to $\vec{n}_{\text{ter}}$, with respect to $\vec{t}_{\text{ter}}(\mathfrak{s})$.

\item The {\it orientation reversal} of $\mathfrak{s}=(s,\vec{n}_{\text{ini}},\vec{n}_{\text{ter}})$ is defined to be 
$$\bar{\mathfrak{s}}=(\bar{s},\vec{n}_{\text{ter}},\vec{n}_{\text{ini}}).$$ 

\item For any angle $\phi \in \mathbb{R}/2\pi\mathbb{Z}$, the {\it frame rotation} of $\mathfrak{s}$ by $\phi$ is defined to be
$$\qquad \qquad \mathfrak{s}(\phi)=(s,\cos{\phi}\cdot \vec{n}_{\text{ini}}+\sin{\phi} \cdot (\vec{t}_{\text{ini}}\times \vec{n}_{\text{ini}}),\cos{\phi} \cdot \vec{n}_{\text{ter}}+\sin{\phi} \cdot (\vec{t}_{\text{ter}}\times \vec{n}_{\text{ter}})).$$

\end{itemize}

we also need a few geometric definitions from Section 4 of \cite{LM}. 

\begin{definition}\label{chainsandcycles}
Let $0<\delta<\frac{\pi}{3}$, $L>0$ and $0<\theta<\pi$ be three constants.
\begin{enumerate}

\item Two oriented $\partial$-framed segments $\mathfrak{s}$ and $\mathfrak{s}'$ are {\it $\delta$-consecutive} if the terminal point of $\mathfrak{s}$ is the initial point of $\mathfrak{s}'$, and the terminal framing of $\mathfrak{s}$ is $\delta$-close to the initial framing of $\mathfrak{s}'$. The {\it bending angle} between $\mathfrak{s}$ and $\mathfrak{s}'$ is the angle between the terminal direction of $\mathfrak{s}$ and the initial direction of $\mathfrak{s}'$.

\item A {\it $\delta$-consecutive chain} of oriented $\partial$-framed segments is a finite sequence $\mathfrak{s}_1,\cdots,\mathfrak{s}_m$ such that each $\mathfrak{s}_i$ is $\delta$-consecutive to $\mathfrak{s}_{i+1}$ for $i=1,\cdots,m-1$. It is a {\it $\delta$-consecutive cycle} if furthermore $\mathfrak{s}_m$ is $\delta$-consecutive to $\mathfrak{s}_1$. A $\delta$-consecutive chain or cycle is {\it $(L,\theta)$-tame} if each $\mathfrak{s}_i$ has length at least $2L$ and each bending angle is at most $\theta$.

\item For an $(L,\theta)$-tame $\delta$-consecutive chain $\mathfrak{s}_1,\cdots,\mathfrak{s}_m$, the {\it reduced concatenation}, denoted by $\mathfrak{s}_1\cdots \mathfrak{s}_m$, is the oriented $\partial$-framed segment defined as the following. The carrier segment of $\mathfrak{s}_1\cdots \mathfrak{s}_m$ is homotopic to the concatenation of carrier segments of $\mathfrak{s}_1,\cdots,\mathfrak{s}_m$, with respect to endpoints. The initial and terminal framings of $\mathfrak{s}_1\cdots\mathfrak{s}_m$ are the closest unit normal vectors to the initial framing of $\mathfrak{s}_1$ and the terminal framing of $\mathfrak{s}_m$ respectively.

\item For an $(L,\theta)$-tame $\delta$-consecutive cycle $\mathfrak{s}_1,\cdots,\mathfrak{s}_m$, the {\it reduced cyclic concatenation}, denoted by $[\mathfrak{s}_1\cdots \mathfrak{s}_m]$, is the oriented closed geodesic freely homotopic to the cyclic concatenation of carrier segments of $\mathfrak{s}_1,\cdots,\mathfrak{s}_m$, assuming it is not null-homotopic.
\end{enumerate}
\end{definition}

The following lemma from \cite{LM} is very useful for estimating the length and phase of a concatenation or a cycle of oriented $\partial$-framed segments. Here the function $I(\cdot)$ is defined by $I(\theta)=2\log{(\sec{\frac{\theta}{2}})}$.

\begin{lemma}\label{lengthphase} (Lemma 4.8 of \cite{LM})
Given any positive constants $\delta,\theta,L$ with $0<\theta<\pi$ and $L\geq I(\theta)+10\log{2}$, the following statements hold in any oriented hyperbolic $3$-manifold.
\begin{enumerate}

\item If $\mathfrak{s}_1,\cdots,\mathfrak{s}_m$ is an $(L,\theta)$-tame $\delta$-consecutive chain of oriented $\partial$-framed segments, denoting the bending angle between $\mathfrak{s}_i$ and $\mathfrak{s}_{i+1}$ as $\theta_i\in[0,\theta)$, then 
$$\Big|l(\mathfrak{s}_1\cdots \mathfrak{s}_m)-\sum_{i=1}^ml(\mathfrak{s}_i)+\sum_{i=1}^{m-1}I(\theta_i)\Big|<\frac{(m-1)e^{(-L+10\log{2})/2}\sin{(\theta/2)}}{L-\log{2}}$$
and 
$$\Big|\varphi(\mathfrak{s}_1\cdots\mathfrak{s}_m)-\sum_{i=1}^m\varphi(\mathfrak{s}_i)\Big|<(m-1)(\delta+e^{(-L+10\log{2})/2}\sin{(\theta/2)}),$$
where $|\cdot|$ on $\mathbb{R}/2\pi\mathbb{Z}$ is understood as the distance from zero valued in $[0,\pi]$.

\item If $\mathfrak{s}_1,\cdots,\mathfrak{s}_m$ is an $(L,\theta)$-tame $\delta$-consecutive cycle of oriented $\partial$-framed segments, denoting the bending angle between $\mathfrak{s}_i$ and $\mathfrak{s}_{i+1}$ as $\theta_i\in[0,\theta)$ with $\mathfrak{s}_{m+1}$ equals $\mathfrak{s}_1$ by convention, then 
$$\Big|l([\mathfrak{s}_1\cdots \mathfrak{s}_m])-\sum_{i=1}^ml(\mathfrak{s}_i)+\sum_{i=1}^{m}I(\theta_i)\Big|<\frac{me^{(-L+10\log{2})/2}\sin{(\theta/2)}}{L-\log{2}}$$
and 
$$\Big|\varphi([\mathfrak{s}_1\cdots\mathfrak{s}_m])-\sum_{i=1}^m\varphi(\mathfrak{s}_i)\Big|<m(\delta+e^{(-L+10\log{2})/2}\sin{(\theta/2)}).$$
\end{enumerate}
\end{lemma}

For an $(L,\theta)$-tame $\delta$-consecutive chain of $\partial$-framed segments $\mathfrak{s}_1,\cdots,\mathfrak{s}_m$, we need the following lemma to bound the difference between initial frames of $\mathfrak{s}_1$ and $\mathfrak{s}_1\cdots \mathfrak{s}_m$.

\begin{lemma}\label{directiondifference}
Let $\delta,\theta,L$ be positive constants with $0<\theta<\pi$ and $L\geq I(\theta)+10\log{2}$. If $\mathfrak{s}_1,\cdots,\mathfrak{s}_m$ is an $(L,\theta)$-tame $\delta$-consecutive chain of oriented $\partial$-framed segments, then the distance between initial frames of $\mathfrak{s}_1$ and $\mathfrak{s}_1\cdots \mathfrak{s}_m$ in $\text{SO}(M)_{p_{\text{ini}}(\mathfrak{s}_1)}$ is at most $8e^{-L}$.
\end{lemma}

\begin{proof}
We first bound the angle between initial directions of $\mathfrak{s}_1\cdots \mathfrak{s}_k$ and $\mathfrak{s}_1\cdots \mathfrak{s}_{k+1}$.

By Lemma \ref{lengthphase} (1), then length of $\mathfrak{s}_1\cdots \mathfrak{s}_{k+1}$ is at least 
\begin{align*}
& \sum_{i=1}^{k+1}l(\mathfrak{s}_i)-\sum_{i=1}^k I(\theta_i)-\frac{ke^{(-L+10\log{2})/2}\sin(\theta/2)}{L-\log 2}\\
\geq \ & l(\mathfrak{s}_{k+1})+k(2L-I(\theta)-\frac{e^{(-L+10\log 2)/2}\sin{(\theta/2)}}{L-\log{2}}) \\
\geq \ & l(\mathfrak{s}_{k+1})+kL.
\end{align*}
By the hyperbolic sine law, the angle between initial directions of $\mathfrak{s}_1\cdots \mathfrak{s}_k$ and $\mathfrak{s}_1\cdots \mathfrak{s}_{k+1}$ is at most $$\arcsin{\frac{\sinh{(l(\mathfrak{s}_{k+1}))}}{\sinh{(l(\mathfrak{s}_{k+1})+kL)}}}\leq \arcsin{e^{-kL}}\leq 2e^{-kL}.$$

So the angle between initial directions of $\mathfrak{s}_1$ and $\mathfrak{s}_1\cdots \mathfrak{s}_m$ is bounded above by $$\sum_{k=1}^{m-1}2e^{-kL}\leq 2\frac{e^{-L}}{1-e^{-L}}\leq 4 e^{-L}.$$

By definition, the angle between initial framings of $\mathfrak{s}_1$ and $\mathfrak{s}_1\cdots \mathfrak{s}_m$ is also bounded above by $4 e^{-L}$. So the initial frames have distance at most $8e^{-L}$ in $\text{SO}(M)|_{p_{\text{ini}}(\mathfrak{s}_1)}$.
\end{proof}

The following lemma (Lemma 3.7 of \cite{Sun4}) bounds the distance between a $\delta$-consecutive cycle of geodesic segments and the corresponding closed geodesic, which is useful for bounding heights of closed geodesics arised from geometric constructions. In the following, an $(L,\theta)$-tame cycle of geodesic segments is a cycle of geodesic segments that satisfies the definition of a $\delta$-consecutive $(L,\theta)$-tame cycle of $\partial$-framed segments (in Definition \ref{chainsandcycles} (2)), except the $\delta$-closeness condition on framings.

\begin{lemma}\label{distance}
For any positive constants $\theta,L$ where $0<\theta<\pi$ and $L\geq 4(I(\theta)+10\log{2})$, the following statement holds in any oriented hyperbolic $3$-manifold. Suppose that $s_1,\cdots,s_m$ is an $(L,\theta)$-tame cycle of geodesic segments with $m\leq L$. 
Then the closed geodesic $[s_1\cdots s_m]$ lies in the $1$-neighborhood of the union $\cup_{i=1}^m s_i$. 
\end{lemma}

Finally, we introduce a notation here. Let $\mathfrak{s}$ be a $\partial$-framed segment from $p$ to $q$ with phase close (say $10^{-2}$-close) to $0$, let ${\bf p}$ and ${\bf q}$ be two frames based at $p$ and $q$ that are very close (say $10^{-2}$-close) to initial and terminal frames of $\mathfrak{s}$ respectively. Then we use 
$$[{\bf p}\xrightarrow{\mathfrak{s}}{\bf q}]$$
 to denote the relative homology class in $H_1(\text{SO}(M),\{{\bf p},{\bf q}\};\mathbb{Z})$ represented by the following path. First travel from ${\bf p}$ to the initial frame of $\mathfrak{s}$ along a $10^{-2}$-short path in $\text{SO}(M)_p$, then parallel transport the initial frame of $\mathfrak{s}$ along the carrier of $\mathfrak{s}$ (to a frame in $\text{SO}(M)_q$ that is $10^{-2}$-close to the terminal frame of $\mathfrak{s}$), then travel to ${\bf q}$ along a $2\times 10^{-2}$-short path in $\text{SO}(M)_q$. 
Moreover, if ${\bf p}={\bf q}$, then $[{\bf p}\xrightarrow{\mathfrak{s}}{\bf q}]$ also represents a homology class in $H_1(\text{SO}(M);\mathbb{Z})$.

Similarly, let $s$ be an oriented geodesic segment from $p$ to $q$, and let ${\bf p}$ and ${\bf q}$ be two frames based at $p$ and $q$ respectively. Suppose that the parallel transport of ${\bf p}$ along $s$ is close (say $10^{-2}$-close) to ${\bf q}$, then we define a relative homology class 
$$[{\bf p}\xrightarrow{s}{\bf q}]\in H_1(\text{SO}(M),\{{\bf p}, {\bf q}\};\mathbb{Z})$$ 
represented by the following path. First parallel transport ${\bf p}$ along $s$ to a frame ${\bf q}'$, then travel from ${\bf q}'$ to ${\bf q}$ along a $10^{-2}$-short path in $\text{SO}(M)_q$.

\bigskip


In \cite{Sun4}, by using results in \cite{KW1} and \cite{KW2}, the author proved a connection principle in oriented cusped hyperbolic $3$-manifolds. 

We first need a notation here. For a point $p\in M$ and a unit tangent vector $\vec{v}_p \in T_pM$, we define a number
$\alpha_{\vec{t}_p}\in[0,\frac{\pi}{2}]$ as the following.
\begin{enumerate}
\item If $p$ has zero height or $p$ has positive height and $\vec{t}_p$ points down
the cusp, we define $\alpha_{\vec{t}_p}= 0$. Here $\vec{t}_p$ ``points down the cusp" means that,
for the cusp whose horotorus boundary going through $p$, $\vec{t}_p$ points to the outward of
this cusp.
\item If $p$ has positive height and $\vec{v}_p$ points up the cusp, we define $\alpha_{\vec{t}_p}\in [0,\frac{\pi}{2}]$
to be the angle between $\vec{v}_p$ and the horotorus going through $p$.
\end{enumerate}
The author proved the following connection principle of oriented cusped hyperbolic $3$-manifolds in Theorem 4.1 of \cite{Sun4}.

\begin{theorem}\label{connectionprinciple}
For any oriented cusped hyperbolic 3-manifold $M$, there exist constants $T>0, C>2$ and $0 <\kappa< 1$ only depend on $M$, such that for any $\delta\in(0,10^{-2})$ and any $t\geq \max{\{T,C(\log{\frac{1}{\delta}}+1)\}}$, the following holds.

Let $p,q\in M$ be two points in $M$ with heights $h_p,h_q< \kappa t$. Let ${\bf p}=(p,\vec{t}_p, \vec{n}_p)\in \text{SO}(M)_p$ and
${\bf q}=(q,\vec{t}_q, \vec{n}_q)\in \text{SO}(M)_q$ be a pair of frames based at $p$ and $q$ respectively. Then for
any $\theta\in \mathbb{R}/2\pi\mathbb{Z}$, there exists a $\partial$-framed segment $\mathfrak{s}$ from $p$ to $q$ such that the following hold.
\begin{enumerate}
\item The length and phase of $\mathfrak{s}$ are $\delta$-close to $t$ and $\theta$ respectively.
\item The initial and terminal frames of $\mathfrak{s}$ are $\delta$-close to ${\bf p}$ and ${\bf q}$ respectively.
\item The height of $\mathfrak{s}$ is at most
$$\max{\Big\{h_p+\min{\{\log{(\sec{\alpha_{\vec{t}_p}})},\log{\frac{1}{\delta}}+C\}}, h_q+\min{\{\log{(\sec{\alpha_{-\vec{t}_q}})},\log{\frac{1}{\delta}}+C\}},\frac{1}{2}\log{t}+C\log{\frac{1}{\delta}}+C\Big\}}.$$
\end{enumerate}
\end{theorem}

\bigskip

\subsection{A connection principle with homological control}\label{connection_principal_homology}

In this section, we prove the following homological version of connection principle in oriented cusped hyperbolic $3$-manifolds.

\begin{theorem}\label{enhancedconnectionprinciple}
Let $M$ be an oriented cusped hyperbolic 3-manifold, and let ${\bf p}=(p,\vec{t}_p,\vec{n}_p),{\bf q}=(q,\vec{t}_q,\vec{n}_q)\in \text{SO}(M)$ be two frames based at $p,q\in M$ respectively. Let $\Xi \in H_1(\text{SO}(M),\{{\bf p}, {\bf q}\};\mathbb{Z})$ be a relative homology class with boundary $\partial\Xi=[{\bf q}]-[{\bf p}]$.

Then for any $\delta\in (0,10^{-2})$, there exists $T=T(M,\Xi,\delta)$ depending on $M,\Xi$ and $\delta$, such that for any $t>T$, there is a $\partial$-framed segment $\mathfrak{s}$ from $p$ to $q$ such that the following hold.
\begin{enumerate}
\item The heights of $p,q$ are at most $\log{t}$, and the height of $\mathfrak{s}$ is at most $2\log{t}$.
\item The length and phase of $\mathfrak{s}$ are $\delta$-close to $t$ and $0$ respectively. The initial and terminal frames of $\mathfrak{s}$ are $\delta$-close to ${\bf p}$ and ${\bf q}$ respectively.
\item We have $$[{\bf p}\xrightarrow{\mathfrak{s}}{\bf q}]=\Xi \in H_1(\text{SO}(M),\{{\bf p}, {\bf q}\};\mathbb{Z}).$$
\end{enumerate}
\end{theorem}


The proof of Theorem \ref{enhancedconnectionprinciple} follows a similar approach as in \cite{Liu}, with the connection principle of closed hyperbolic $3$-manifolds (Lemma 4.15 of \cite{LM}) replaced by Theorem \ref{connectionprinciple}. In the proof of Theorem \ref{enhancedconnectionprinciple} and its lemmas, we will focus on how to use the connection principle of cusped hyperbolic $3$-manifolds (Theorem \ref{connectionprinciple}) to accomplish geometric constructions and how to deduce height control. Once the geometric construction is done, the proof of homological conditions follows from the corresponding proof in \cite{Liu}. Those proofs in \cite{Liu} only involve constructions of chains and homotopy, so they work well for cusped hyperbolic $3$-manifolds.


\begin{lemma}\label{commonnormal}
Let $M$ be an oriented cusped hyperbolic $3$-manifold, let $p$ be a point in $M$, and let ${\bf p},{\bf p}'\in \text{SO}(M)_p$ be two frames based at $p$ sharing the same normal vector. Then for any $\delta\in (0,10^{-2})$, there exists $T_1=T_1(M,p,\delta)$ depending on $M$, $p$ and $\delta$, such that for any $t>T_1$, there is a $\partial$-framed segment $\mathfrak{s}$ from $p$ to itself such that the following hold.
\begin{enumerate}
\item The height of $p$ is at most $\log{t}$, and the height of $\mathfrak{s}$ is at most $2\log{t}$.
\item The length and phase of $\mathfrak{s}$ are $\delta$-close to $t$ and $0$ respectively. The initial and terminal frames of $\mathfrak{s}$ are $\delta$-close to ${\bf p}$ and ${\bf p}'$ respectively.
\item The carrier of $\mathfrak{s}$ gives a null-homologous closed curve in $M$.
\end{enumerate}
Moreover, if ${\bf p}={\bf p'}$, then $[{\bf p}\xrightarrow{\mathfrak{s}} {\bf p}]=0\in H_1(\text{SO}(M);\mathbb{Z})$.
\end{lemma}

\begin{proof}

We take $T_1=T_1(M,p,\delta)$ such that the following hold: $$T_1\geq 4\max\{100,T, C(\log{\frac{10}{\delta}}+1), \kappa^{-1}h_p,e^{h_p},\frac{10}{\delta}e^{C+1},\frac{10^Ce^{C+1}}{\delta^C},10\log{\frac{1000}{\delta}}\}.$$
Here $T,C,\kappa$ are constants in Theorem \ref{connectionprinciple}, and $h_p$ denotes the height of $p$. Then we take any $t>T_1$.

Since ${\bf p}$ and ${\bf p}'$ share the same normal vector, we denote ${\bf p}=(p,\vec{v}_1,\vec{n})$ and ${\bf p}'=(p,\vec{v}_2,\vec{n})$. We take two unit vectors $\vec{v}_3,\vec{v}_4$ based at $p$ and orthogonal with $\vec{n}$, such that all angles $\theta_{23}=\angle(\vec{v}_2,-\vec{v}_3),\theta_{34}=\angle(-\vec{v}_3,\vec{v}_4),\theta_{41}=\angle(\vec{v}_4,\vec{v}_1)$ are at most $\frac{\pi}{3}$. Here $\vec{v}_3$ and $\vec{v}_4$ exist since $\vec{v}_1$ and $\vec{v}_2$ divide the normal plane of $\vec{n}$ to two pieces and one of them has angle at most $\pi$.

We take two frames ${\bf q}=(p,\vec{v}_3,\vec{n}), {\bf q}'=(p,\vec{v}_4,\vec{n})\in \text{SO}(M)_p$ based at $p$. Let $I=I(\theta_{23})+I(\theta_{34})+I(\theta_{41})$. By our choice of $T_1$, we can apply Theorem \ref{connectionprinciple} (with respect to $\frac{\delta}{10}$) to obtain two $\partial$-framed segments $\mathfrak{a},\mathfrak{b}$ from $p$ to itself such that the following hold.
\begin{itemize}
\item[(a)] The lengths and phases of $\mathfrak{a}$ and $\mathfrak{b}$ are $\frac{\delta}{10}$-close to $\frac{t+I}{4}$ and $0$ respectively. The heights of $\mathfrak{a}$ and $\mathfrak{b}$ are at most $2\log{t}-1$.
\item[(b)] The initial and terminal frames of $\mathfrak{a}$ are $\frac{\delta}{10}$-close to ${\bf p}$ and ${\bf q}$ respectively.
\item[(c)] The initial and terminal frames of $\mathfrak{b}$ are $\frac{\delta}{10}$-close to $(p,-\vec{v}_2,\vec{n})$ and ${\bf q'}$ respectively.
\end{itemize}

Then we take $\mathfrak{s}$ to be the reduced concatenation $\mathfrak{a}\mathfrak{b}\bar{\mathfrak{a}}\bar{\mathfrak{b}}$, and note that $\theta_{23},\theta_{34},\theta_{41}$ are bending angles of this concatenation. 
Now we check that $\mathfrak{s}$ satisfies desired conditions. For condition (1), the height bound of $p$ follows from the choice of $T_1$, and the height bound of $\mathfrak{s}$ follows from item (a) and Lemma \ref{distance}. For condition (2), the length and phase bounds of $\mathfrak{s}$ follow form items (a) - (c) and Lemma \ref{lengthphase} (1), while the conditions on initial and terminal frames follow from items (b) (c) and Lemma \ref{directiondifference}. Condition (3) clearly holds since the carrier of $\mathfrak{s}$ is a commutator in $\pi_1(M,p)$.

For the moreover part, since ${\bf p}={\bf p'}$, we have $\vec{v}_1=\vec{v}_2$. Then we take $\vec{v}_3=-\vec{v}_1$ and $\vec{v}_4=\vec{v}_1$ in the above construction. The initial direction of $\mathfrak{a}$ and terminal direction of $\mathfrak{b}$ are both $\frac{\delta}{10}$-close to $\vec{v}_1$, while the terminal direction of $\mathfrak{a}$ and initial direction of $\mathfrak{b}$ are both $\frac{\delta}{10}$-close to $-\vec{v}_1$.
The homological condition $[{\bf p}\xrightarrow{\mathfrak{s}} {\bf p}]=0\in H_1(\text{SO}(M);\mathbb{Z})$ follows from Lemma 3.4 of \cite{Liu}. 

\end{proof}

For a frame ${\bf p}=(p,\vec{v},\vec{n})\in \text{SO}(M)_p$, we have a path of frames $\omega:\mathbb{R}\to \text{SO}(M)_p$ defined by $\omega (t)=(p,\vec{v},\cos{t}\cdot \vec{n}+\sin{t}\cdot \vec{v}\times \vec{n})$ given by frame rotation. For any $\phi\in \mathbb{R}$, we define a new frame ${\bf p}(\phi)=(p,\vec{v},\cos{\phi}\cdot \vec{n}+\sin{\phi}\cdot \vec{v}\times \vec{n})$. The following lemma is an analogy of Proposition 3.5 of \cite{Liu}.

\begin{lemma}\label{commontangent}
Let $M$ be an oriented cusped hyperbolic $3$-manifold, let $p\in M$ be a point, let ${\bf p}\in \text{SO}(M)_p$ be a frame based at $p$, and let $\phi\in [-2\pi,2\pi]$ be a real number. Then for any $\delta \in (0,10^{-2})$, there exists $T_2=T_2(M,p,\delta)$ depending on $M$, $p$ and $\delta$, such that for any $t>T_2$, there is a $\partial$-framed segment $\mathfrak{s}$ from $p$ to itself such that the following hold.
\begin{enumerate}
\item The height of $p$ is at most $\log{t}$, and the height of $\mathfrak{s}$ is at most $2\log{t}$.
\item The length and phase of $\mathfrak{s}$ are $\delta$-close to $t$ and $0$ respectively. The initial and terminal frames of $\mathfrak{s}$ are $\delta$-close to ${\bf p}$ and ${\bf p}(\phi)$ respectively.
\item The carrier of $\mathfrak{s}$ is null-homologous in $M$.
\item The following two relative homology classes are equal: $$[{\bf p}\xrightarrow{\mathfrak{s}}{\bf p}(\phi)]=[\omega|_{[0,\phi]}:[0,\phi]\to \text{SO}(M)] \in H_1(\text{SO}(M),\{{\bf p}, {\bf p}(\phi)\};\mathbb{Z}).$$
\end{enumerate}
\end{lemma}

\begin{proof}
Let $d>0$ be a positive number such that $10^4d$ is smaller than the injectivity radius of $\text{SO}(3)$. Take a positive integer $N$ such that $N>\frac{2\pi}{d}$. We take $T_2=T_2(M,p,\delta)$ to be: 
$$\max\{1000N^2,4NT, 4NC(\log{\frac{100N}{\delta}}+1), \frac{4Nh_p}{\kappa},e^{h_p},\frac{100N}{\delta}e^{C+2},\frac{(100N)^Ce^{C+2}}{\delta^C},10N\log{\frac{10^4N}{\delta}}\}.$$
Here $T,C,\kappa$ are constants in Theorem \ref{connectionprinciple}, and $h_p$ denotes the height of $p$. We assume that $t>T_2$ holds.

At first, we assume that $|\phi|<d$. Let ${\bf p}=(p,\vec{v},\vec{n})$ and let ${\bf q}=(p,-\vec{v},\vec{n})$. Then we apply Theorem \ref{connectionprinciple} (with respect to $\frac{\delta}{100N}$) to obtain two $\partial$-framed segments  $\mathfrak{a},\mathfrak{b}$ from $p$ to itself such that the following hold.
\begin{itemize}
\item The lengths and phases of $\mathfrak{a}$ and $\mathfrak{b}$ are $\frac{\delta}{100N}$-close to $\frac{t}{4N}$ and $0$ respectively. The heights of $\mathfrak{a}$ and $\mathfrak{b}$ are at most $2\log{t}-2$.
\item The initial and terminal frames of $\mathfrak{a}$ are $\frac{\delta}{100N}$-close to ${\bf p}$ and ${\bf q}$ respectively.
\item The initial and terminal frames of $\mathfrak{b}$ are $\frac{\delta}{100N}$-close to ${\bf q}$ and ${\bf p}(\phi/2)$ respectively.
\end{itemize} 
We take $\mathfrak{s}$ to be the reduced concatenation $\mathfrak{a}\cdot\mathfrak{b}\cdot\bar{\mathfrak{a}}(\phi/2)\cdot \bar{\mathfrak{b}}(\phi)$ (arised from a $\frac{2\delta}{100N}$-consecutive chain). Then $\mathfrak{s}$ satisfies the desired properties of this lemma, such that its height is at most $2\log{t}-1$, its length is close to $\frac{t}{N}$, with all $\delta$ replaced by $\frac{\delta}{10N}$. The height bound, length and phase bound, and frame closeness follow from Lemmas \ref{distance}, \ref{lengthphase} (1) and \ref{directiondifference} respectively. Condition (4) follows from the argument in Section 3.2.1 of \cite{Liu}.

Now $\phi$ can be any angle in $[-2\pi,2\pi]$. For any $k=0,1,\cdots N$, let $\phi_k=\frac{k\phi }{N}$, then $|\phi_{k+1}-\phi_k|<d$ holds. For each $k=1,2,\cdots,N$, we apply the previous case to construct a $\partial$-framed segment $\mathfrak{s}_k$ with respect to ${\bf p}(\phi_{k-1})$ and ${\bf p}(\phi_k)$ (since $({\bf p}(\phi_{k-1}))(\phi/N)={\bf p}(\phi_k)$) such that the following hold.
\begin{itemize}

\item[(a)] The height of $\mathfrak{s}_k$ is at most $2\log{t}-1$. The length and phase of $\mathfrak{s}_k$ are $\frac{\delta}{10N}$-close to $\frac{t}{N}$ and $0$ respectively.

\item[(b)] The initial and terminal frames of $\mathfrak{s}_k$ are $\frac{\delta}{10N}$-close to ${\bf p}(\phi_{k-1})$ and ${\bf p}(\phi_k)$ respectively.

\item[(c)] The carrier of $\mathfrak{s}_k$ gives a null-homologous closed curve in $M$.

\item[(d)] The following closed curve is null-homologous in $\text{SO}(M)$. First take the path from ${\bf p}(\phi_{k-1})$ to ${\bf p}(\phi_k)$ in the definition of $[{\bf p}(\phi_{k-1})\xrightarrow{\mathfrak{s}_k}{\bf p}(\phi_k)]$, then travels back to ${\bf p}(\phi_{k-1})$ along the $\frac{\phi}{N}$-long rotation path in $\text{SO}(M)_p$.
\end{itemize}

Then we take $\mathfrak{s}$ to be the reduced concatenation $\mathfrak{s}_1\cdots \mathfrak{s}_N$. The height bound of $\mathfrak{s}$ in (1) follows from item (a) and Lemma \ref{distance}. The length and phase estimates of $\mathfrak{s}$ in (2) follow from items (a) (b) and Lemma \ref{lengthphase} (1). The condition of initial and terminal frames of $\mathfrak{s}$ in (2) follows from item (b) and Lemma \ref{directiondifference}. The homological condition in (3) follows directly from item (c). The relative homological condition in (4) follows from arguments in Section 3.2.2 and Lemma 3.2 of \cite{Liu}. 

\end{proof}

The following lemma is a direct consequence of Lemmas \ref{commonnormal} and \ref{commontangent}.
\begin{lemma}\label{nullhomologous}
Let $M$ be an oriented cusped hyperbolic $3$-manifold, let $p\in M$ be a point, and let ${\bf p},{\bf p}'\in \text{SO}(M)_p$ be two frames based at $p$. Then for any $\delta \in (0,10^{-2})$, there exists $T_3=T_3(M,p,\delta)$ depending on $M$, $p$ and $\delta$, such that for any $t>T_3$, there is a $\partial$-framed segment $\mathfrak{s}$ from $p$ to itself such that the following hold.
\begin{enumerate}
\item The height of $p$ is at most $\log{t}$, and the height of $\mathfrak{s}$ is at most $2\log{t}$.
\item The length and phase of $\mathfrak{s}$ are $\delta$-close to $t$ and $0$ respectively. The initial and terminal frames of $\mathfrak{s}$ are $\delta$-close to ${\bf p}$ and ${\bf p}'$ respectively.
\item The carrier of $\mathfrak{s}$ is null-homologous in $M$.
\end{enumerate}
\end{lemma}

\begin{proof} We take $$T_3=T_3(M,p,\delta)=3\Big(T_1(M,p,\frac{\delta}{10})+T_2(M,p,\frac{\delta}{10})\Big),$$ and take any $t>T_3$.

Let ${\bf p}=(p,\vec{v}_1,\vec{n}_1)$ and ${\bf p}'=(p,\vec{v}_2,\vec{n}_2)$. We take $\vec{n}\in T_pM$ to be a unit vector orthogonal to both $\vec{v}_1$ and $\vec{v}_2$, and take two frames ${\bf p}_1=(p,\vec{v}_1,\vec{n})$ and ${\bf p}_2=(p,\vec{v}_2,\vec{n})$ based at $p$. Since both $\vec{n}_1$ and $\vec{n}$ are normal to $\vec{v}_1$, we have ${\bf p}_1={\bf p}(\phi_1)$ for some $\phi_1\in [0,2\pi]$. Similarly, ${\bf p}'={\bf p}_2(\phi_2)$ for some $\phi_2\in [0,2\pi]$. 

By applying Lemmas \ref{commontangent}, \ref{commonnormal}, \ref{commontangent} respectively (with $\delta$ replaced by $\frac{\delta}{10}$), we construct $\partial$-framed segments $\mathfrak{s}_1$, $\mathfrak{s}_2$, $\mathfrak{s}_3$ such that the following hold.
\begin{itemize}

\item[(a)] The height of $p$ is at most $\log{t}$. The heights of $\mathfrak{s}_1$, $\mathfrak{s}_2$ and $\mathfrak{s}_3$ are all at most $2\log{\frac{t}{3}}<2\log{t}-1$.

\item[(b)] The lengths and phases of $\mathfrak{s}_1$, $\mathfrak{s}_2$ and $\mathfrak{s}_3$ are all $\frac{\delta}{10}$-close to $\frac{t}{3}$ and $0$ respectively. 

\item[(c)] The initial and terminal frames of $\mathfrak{s}_1$ are $\frac{\delta}{10}$-close to ${\bf p}$ and ${\bf p}_1$ respectively. The initial and terminal frames of $\mathfrak{s}_2$ are $\frac{\delta}{10}$-close to ${\bf p}_1$ and ${\bf p}_2$ respectively. The initial and terminal frames of $\mathfrak{s}_3$ are $\frac{\delta}{10}$-close to ${\bf p}_2$ and ${\bf p}'$ respectively. 

\item[(d)] The carriers of $\mathfrak{s}_1$, $\mathfrak{s}_2$ and $\mathfrak{s}_3$ are all null-homologous closed curves in $M$.
\end{itemize}

Then we take $\mathfrak{s}$ to be the reduced concatenation $\mathfrak{s}_1\mathfrak{s}_2\mathfrak{s}_3$. The height bound of $\mathfrak{s}$ in (1) follows from item (a) and Lemma \ref{distance}. The length and phase estimate of $\mathfrak{s}$ in (2) follows from items (b) (c) and Lemma \ref{lengthphase} (1). The conditions on initial and terminal frames of $\mathfrak{s}$ in (2) follow from item (c) and Lemma \ref{directiondifference}. The null-homologous condition on $\mathfrak{s}$ in (3) follows from item (d).

\end{proof}

Now we are ready to prove Theorem \ref{enhancedconnectionprinciple}.
\begin{proof}[Proof of Theorem \ref{enhancedconnectionprinciple}]

Let $\pi:\text{SO}(M)\to M$ be the projection of frame bundle, and let $\xi=\pi_*(\Xi)\in H_1(M,\{p,q\};\mathbb{Z})$. Then we take a geodesic segment $s_0$ from $p$ to $q$ representing $\xi$, and we can make sure that $s_0$ has length at least $10\log{\frac{10000}{\delta}}$ (first take any $s_0$ representing $\xi$, then concatenate it with a large power of a null-homologous closed curve based at $q$).

Then we equip $s_0$ with arbitrary initial and terminal frames ${\bf p}'$ and ${\bf q}'$ respectively, so that we get a $\partial$-framed segment $\mathfrak{s}_0$ with phase $0$.  Then we take 
$$T(M,\Xi,\delta)=\max\Big\{3\big(l(\mathfrak{s}_0)+T_2(M,q,\frac{\delta}{100})+T_3(M,p,\frac{\delta}{100})+T_3(M,q,\frac{\delta}{100})\big),e^{\frac{h(\mathfrak{s}_0)+2}{2}}\Big\}$$ 
and take any $t>T(M,\Xi,\delta)$.
Here $T_2,T_3$ are constants given in Lemmas \ref{commontangent} and \ref{nullhomologous} respectively, $l(\mathfrak{s}_0)$ and $h(\mathfrak{s}_0)$ denote the length and height of $\mathfrak{s}_0$ respectively.

By applying Lemma \ref{nullhomologous} (with $\delta$ replaced by $\frac{\delta}{100}$), we construct two $\partial$-framed segments $\mathfrak{s}_1,\mathfrak{s}_2$  such that the following hold.
\begin{itemize}
\item The heights of $p$ and $q$ are at most $\log{t}$. The heights of $\mathfrak{s}_1$ and $\mathfrak{s}_2$ are at most $2\log{\frac{t}{3}}<2\log{t}-2$.
\item The lengths and phases of $\mathfrak{s}_1$ and $\mathfrak{s}_2$ are $\frac{\delta}{100}$-close to $\frac{t-l(\mathfrak{s}_0)}{3}$ and $0$ respectively.
\item The initial and terminal points of $\mathfrak{s}_1$ are both $p$, while the initial and terminal points of $\mathfrak{s}_2$ are both $q$.
\item The initial and terminal frames of $\mathfrak{s}_1$ are $\frac{\delta}{100}$-close to ${\bf p}$ and ${\bf p}'$ respectively. The initial and terminal frames of $\mathfrak{s}_2$ are $\frac{\delta}{100}$-close to ${\bf q}'$ and ${\bf q}$ respectively.
\item The closed curves given by $\mathfrak{s}_1$ and $\mathfrak{s}_2$ are both null-homologous in $M$.
\end{itemize}

Then the reduced concatenation $\mathfrak{s}'=\mathfrak{s}_1\mathfrak{s}_0\mathfrak{s}_2$ satisfies the following conditions.
\begin{itemize}
\item The height of $\mathfrak{s}'$ is at most $2\log{t}-1$.
\item The length and phase of $\mathfrak{s}'$ are $\frac{\delta}{10}$-close to $\frac{2t+l(\mathfrak{s}_0)}{3}$ and $0$ respectively.
\item The initial and terminal frames of $\mathfrak{s}'$ are $\frac{\delta}{10}$-close to ${\bf p}$ and ${\bf q}$ respectively.
\item The relative homology class of the carrier of $\mathfrak{s}'$ equals $\xi=\pi_*(\Xi)\in H_1(M,\{p,q\};\mathbb{Z})$.
\end{itemize}
The conditions on height, length and phase, initial and terminal frames of $\mathfrak{s'}$ follow from corresponding conditions of $\mathfrak{s}_0,\mathfrak{s}_1,\mathfrak{s}_2$ and Lemmas \ref{distance}, \ref{lengthphase} (1) and \ref{directiondifference} respectively. The condition on the relative homology class follows from homological conditions on $\mathfrak{s}_0, \mathfrak{s}_1, \mathfrak{s}_2$.

For the relative homology class $\Xi'=[{\bf p}\xrightarrow{\mathfrak{s}'} {\bf q}]\in H_1(\text{SO}(M),\{{\bf p},{\bf q}\};\mathbb{Z})$, since $\pi_*(\Xi')=\xi=\pi_*(\Xi)\in H_1(M,\{p,q\};\mathbb{Z})$, we have $\Xi'\in (\pi_*)^{-1}(\xi) \subset H_1(\text{SO}(M),\{{\bf p},{\bf q}\};\mathbb{Z})$.

Since $(\pi_*)^{-1}(\xi)$ consists of only two elements, either $\Xi'=\Xi$ or $\Xi'=\Xi+\Lambda$ holds. Here $\Lambda$ denotes the nontrivial element in $\mathbb{Z}/2\mathbb{Z}\cong H_1(\text{SO}(3);\mathbb{Z})\subset H_1(\text{SO}(M);\mathbb{Z})$. Then we apply Lemma \ref{commonnormal} to construct a $\partial$-framed segment $\mathfrak{s}_3$, with $\phi=0$ when $\Xi'=\Xi$ and $\phi=2\pi$ when $\Xi'=\Xi+\Lambda$, such that the following hold.
\begin{itemize}
\item The height of $\mathfrak{s}_3$ is at most $2\log{t}-1$.
\item The length and phase of $\mathfrak{s}_3$ are $\frac{\delta}{10}$-close to $\frac{t-l(\mathfrak{s}_0)}{3}$ and $0$ respectively.
\item The initial and terminal frames of $\mathfrak{s}_3$ are both $\frac{\delta}{10}$-close to ${\bf q}$.
\item The carrier of $\mathfrak{s}_3$ gives a null-homologous closed curve in $M$.
\item If $\Xi'=\Xi$, we have
$[{\bf q}\xrightarrow{\mathfrak{s}_3}{\bf q}]=0\in H_1(\text{SO}(M);\mathbb{Z})$. If $\Xi'=\Xi+\Lambda$, we have
$[{\bf q}\xrightarrow{\mathfrak{s}_3}{\bf q}]=\Lambda \in H_1(\text{SO}(M);\mathbb{Z})$.
\end{itemize}

Then we take the desired $\partial$-framed segment $\mathfrak{s}$ to be the reduced concatenation $\mathfrak{s}'\mathfrak{s}_3$. Again, the conditions on height, length and phase, initial and terminal frames of $\mathfrak{s}$ follow from corresponding conditions of $\mathfrak{s}',\mathfrak{s}_3$ and Lemmas \ref{distance}, \ref{lengthphase} (1), \ref{directiondifference} respectively. By homological conditions on carriers of $\mathfrak{s}'$ and $\mathfrak{s}_3$, the carrier of $\mathfrak{s}$ represents $\xi\in H_1(M,\{p,q\};\mathbb{Z})$. Since our construction of $\mathfrak{s}_3$ depends on the relative homological condition of $\mathfrak{s}'$, by Lemma 3.2 of \cite{Liu}, $[{\bf p}\xrightarrow{\mathfrak{s}} {\bf q}]$ equals $\Xi\in H_1(\text{SO}(M),\{{\bf p}, {\bf q}\};\mathbb{Z})$ in both cases.

\end{proof}

\subsection{$(R,\epsilon)$-good multicurves that bound nearly geodesic subsurfaces}

In this section, we prove the following result that characterizes certain null-homologous $(R,\epsilon)$-good multicurve in a cusped hyperbolic $3$-manifold bounds a nearly geodesic subsurface. This result generalizes Corollary 2.11 of \cite{Sun2}. 

Note that \cite{KW1} requires that $R$ must be an integer, while we may start with a real number $R$ in this paper, so the statement of this result is a bit complicated.

\begin{theorem}\label{boundingsurface}
Let $M$ be an oriented cusped hyperbolic $3$-manifold. Then for any constant $\alpha\geq 4$, any small enough $\epsilon> 0$ depending on $M$ and any large enough real number $R > 0$ depending on $\epsilon$ and $M$, the following statement holds. For any null-homologous oriented $(R,\epsilon)$-multicurve $L \in \mathbb{Z}{\bf \Gamma}_{R,\epsilon}^{\alpha \log{R}}$, there is a nontrivial invariant $\sigma(L)\in \mathbb{Z}_2$ such that $\sigma(L_1\cup L_2)=\sigma(L_1)+\sigma(L_2)$ and the following hold.

If $\sigma(L)=0$, for any integer $R'\geq R$, $L$ is the oriented boundary of an (possibly disconnected) immersed subsurface $f:S\looparrowright M$ satisfying the following conditions.
\begin{enumerate}
\item If we write $L$ as a union of its components $L=L_1\cup \cdots\cup L_k$, then $S$ decomposes as oriented subsurfaces $S=(\cup_{i=1}^k \Pi_i)\cup S'$ with disjoint interior, such that $\Pi_i\cap \partial S$ is a single curve $c_i$ that is mapped to $L_i$.
\item The restriction $f:\Pi_i\looparrowright M$ gives an immersed pair of pants such that $|{\bf hl}_{\Pi_i}(L_i)-R|<\epsilon$, and $|{\bf hl}_{\Pi_i}(s_i)-R'|<\epsilon$ holds for the $f$-image of any other cuff $s_i$ of $\Pi_i$.
\item If we fix a normal vector $\vec{v}_i\in N^1(\sqrt{L_i})$ for each component $L_i$ of $L$, then we can construct $f:S\looparrowright M$ such that $|{\bf foot}_{L_i}(\Pi_i,c_i)-\vec{v}_i|<\epsilon$ holds.
\item The restriction $S'\looparrowright M$ is an oriented $(R',\epsilon)$-nearly geodesic subsurface.
\item For any component $s_i\subset S'\cap \Pi_i$ mapped to $\gamma_i\in {\bf \Gamma}_{R',\epsilon}$, we take its orientation induced from $\Pi_i$, then we have $$|{\bf foot}_{\gamma_i}(\Pi_i,s_i)-{\bf foot}_{\bar{\gamma}_i}(S',\bar{s}_i)-(1+\pi i)|<\frac{\epsilon}{R}.$$
\end{enumerate}
\end{theorem}

We call each immersed pants in item (2) a pair of {\it $(R,R',\epsilon)$-good pants}, and we call the immersed subsurface $f:S\looparrowright M$ in Theorem \ref{boundingsurface} {\it an $(R',\epsilon)$-nearly geodesic subsurface with $(R,\epsilon)$-good boundary}.

 \begin{proof}
 We take $\beta=\alpha+4$, $\epsilon>0$ and $R>0$ such that Theorem \ref{cuspedhomology} holds. 
 For any $(R,\epsilon)$-multicurve $L$ of height at most $\alpha \log{R}$, we consider $L$ as an element in $\Omega_{R,\epsilon}^{\alpha \log{R},\beta \log{R}}(M)$. Then we apply the isomorphism $\Phi$ in Theorem \ref{cuspedhomology} to get $$\Phi(L)\in H_1(\text{SO}(M);\mathbb{Z})\cong H_1(M;\mathbb{Z})\times H_1(\text{SO}(3);\mathbb{Z})\cong H_1(M;\mathbb{Z})\times \mathbb{Z}_2.$$
 Since $L$ is null-homologous, $\Phi(L)$ is trivial in the first component. So we define $\sigma(L)\in \mathbb{Z}_2$ to be the second component of $\Phi(L)$, and $\sigma$ obviously satisfies $\sigma(L_1)+\sigma(L_2)=\sigma(L_1\cup L_2)$.
 
 Now we fix a null-homologous $(R,\epsilon)$-multicurve $L$ of height at most $\alpha \log{R}$ such that $\sigma(L)=0$ holds. For each $i$, let $\hat{l}_i$ be one canonical lift of $L_i$, then it is an oriented closed curve in $\text{SO}(M)$. Since $\Phi(L)=0 \in H_1(\text{SO}(M);\mathbb{Z})$, $\hat{l}_1\cup \cdots \cup \hat{l}_k$ is null-homologous in $\text{SO}(M)$.

{\bf Step I.} In this step, we construct the pairs of pants $\Pi_i$.

 For each component $L_i$ of $L$, we take two antipodal points $p_i,q_i$ on $L_i$, such that their distances to the two basepoints of $\vec{v}_i$ are all $\frac{l(L_i)}{4}$. Then $p_i$ and $q_i$ divide $L_i$ to two oriented geodesic segments $\alpha_{i,1}$ (from $p_i$ to $q_i$) and $\alpha_{i,2}$ (from $q_i$ to $p_i$). We take two frames ${\bf p_i}=(p_i,\vec{v}_{p_i},\vec{n}_{p_i}),{\bf q_i}=(q_i,\vec{v}_{q_i},\vec{n}_{q_i})$ based at $p_i$ and $q_i$ respectively such that the following hold.
\begin{itemize}
\item The tangent vectors of ${\bf p_i},{\bf q_i}$ are both tangent to $L_i$.
\item The third vectors (cross products) of ${\bf p_i},{\bf q_i}$ both lie in the slow and constant turning vector field going through $\vec{v}_i$. So the complex distance between ${\bf p_i}$ and ${\bf q_i}$ along $L_i$ is ${\bf hl}(L_i)$.
 \end{itemize}
 
 Then we apply Theorem \ref{connectionprinciple} to construct a $\partial$-framed segment $\mathfrak{s}_i$ from $p_i$ to $q_i$ such that the following hold.
 \begin{itemize}
 \item The length and phase of $\mathfrak{s}_i$ are $\frac{\epsilon}{10}$-close to $2R'-\frac{l(L_i)}{2}+2\log{2}$ and $0$ respectively
 \item The initial and terminal frames of $\mathfrak{s}_i$ are $\frac{\epsilon}{10}$-close to $(p_i,\vec{v}_{p_i}\times \vec{n}_{p_i},\vec{n}_{p_i})$ and $(q_i,-\vec{v}_{q_i}\times \vec{n}_{q_i},\vec{n}_{q_i})$ respectively
 \item The height of $\mathfrak{s}_i$ is at most $(\alpha+1)\log{R'}-1$.
 \end{itemize}

Let $s_i$ be the carrier of $\mathfrak{s}_i$, and let $L_{i,1}, L_{i,2}$ be the closed geodesic homotopic to concatenations $\overline{\alpha_{i,1}}s_i$ and $\overline{s_i}\ \overline{\alpha_{i,2}}$ respectively. By Lemma \ref{lengthphase} (2), both $L_{i,1}$ and $L_{i,2}$ are $(R',\epsilon)$-good curves. By Lemma \ref{distance}, both $L_{i,1}$ and $L_{i,2}$ have height at most $(\alpha+1)\log{R'}$. We take canonical lifts $\hat{l}_{i,1}$ and $\hat{l}_{i,2}$ of $L_{i,1}$ and $L_{i,2}$ respectively, then the proof of Lemma 5.12 of \cite{LM} implies that $\hat{l}_{i,1}\cup \hat{l}_{i,2}\cup \hat{l}_i$ is null-homologous in $\text{SO}(M)$. Although Lemma 5.12 of \cite{LM} is only stated for $(R,\epsilon)$-good pants, the proof works well in our context.

By our construction of $\Pi_i$, its cuffs satisfy condition (2) and its foot on $L_i$ satisfies condition (3).

\bigskip
 
{\bf Step II.} Now we construct the $(R',\epsilon)$-nearly geodesic subsurface $S'\looparrowright M$.

Since $\hat{l}_1\cup \cdots \cup \hat{l}_k$ is null-homologous in $\text{SO}(M)$, $\hat{l}_{1,1}\cup \hat{l}_{1,2}\cup \cdots \cup \hat{l}_{k,1}\cup \hat{l}_{k,2}$ is also null-homologous in $\text{SO}(M)$. Let $L'=\bar{L}_{1,1}\cup \bar{L}_{1,2}\cup \cdots \cup \bar{L}_{k,1} \cup \bar{L}_{k,2}$, then $L'$ is a null-homologous $(R',\epsilon)$-multicurve in $M$. 

Let $\alpha'=\alpha+1$ and $\beta'=\beta+1$, then we take $h_c$ and $h_T$ with respect to $\alpha', \beta'$ and $R'$. We consider $L'$ as an element in $\Omega_{R,\epsilon}^{\alpha'\log{R'},\beta'\log{R'}}(M)$. Let $\Phi:\Omega_{R,\epsilon}^{\alpha'\log{R'},\beta'\log{R'}}(M)\to H_1(\text{SO}(M);\mathbb{Z})$ be the isomorphism given by Theorem \ref{cuspedhomology}. Since $\hat{l}_{1,1}\cup \hat{l}_{1,2}\cup \cdots \cup \hat{l}_{k,1}\cup \hat{l}_{k,2}$ is null-homologous in $\text{SO}(M)$ and is a component-wise canonical lift of $\bar{L}'$, we have $\Phi(L')=\Phi(\bar{L}')=0\in H_1(\text{SO}(M);\mathbb{Z})$. So $L'$ is trivial in $\Omega_{R,\epsilon}^{\alpha'\log{R'},\beta'\log{R'}}(M)$.

By definition of the panted cobordism group, there is an $(R',\epsilon)$-panted subsurface $\Sigma \to M$ of height at most $\beta' \log{R'}$, such that $L'$ is the oriented boundary of $\Sigma$. This $(R',\epsilon)$-panted subsurface $\Sigma$ may not be $(R',\epsilon)$-nearly geodesic, so we do the following process to make it to be $(R',\epsilon)$-nearly geodesic and satisfies condition (5). The construction is essentially same with the construction in Corollary 2.11 of \cite{Sun2}.

Let $A'$ be the $\mathbb{Z}$-linear combination of $(R',\epsilon)$-good pants given by $\Sigma$, then $\partial A'=L'\in \mathbb{Z}{\bf \Gamma}_{R,\epsilon}^{\alpha' \log{R'}}$ holds. For any component $\bar{L}_{i,s}$ of $L'$ with $s=1,2$, let $\vec{w}_{i,s}={\bf foot}_{L_{i,s}}(\Pi_i)\in N^1(\sqrt{L_{i,s}})$. We take a positive integer $N$ such that $N\geq 3|A'|+3$.

We take $h_T\geq \beta'\log{R'},h_c\geq h_T+44\log{R'}$ and apply Theorem \ref{cuspedprecisedescription}, to get a $\mathbb{Z}_+$-linear combination $A$ of $(R',\epsilon)$-good pants and umbrellas (after eliminating denominators).
Then we take two measures on $N^1(\sqrt{L_{i,s}})$:
$$M_{L_{i,s}}^1=\partial (NA+N\bar{A}+A')_{L_{i,s}} +\vec{w}_{i,s},\ M_{L_{i,s}}^2=\partial (NA+N\bar{A}+\bar{A'})_{L_{i,s}}$$ 
For any other $\gamma\in {\bf \Gamma}_{R',\epsilon}^{<h_c}$, we simply take the following two measures on $N^1(\sqrt{\gamma})$: $$M_{\gamma}^1=\partial (NA+N\bar{A}+A')_{\gamma},\ M_{\gamma}^2=\partial (NA+N\bar{A}+\bar{A'})_{\gamma}.$$

Since $A'$ corresponds to an $(R',\epsilon)$-panted subsurface with oriented boundary $L'$, the above pairs of measures ($M_{L_{i,s}}^1$ versus $M_{L_{i,s}}^2$ and $M_{\gamma}^1$ versus $M_{\gamma}^2$) have same total measure on $N^1(\sqrt{L_{i,s}})$ and $N^1(\sqrt{\gamma})$ respectively. For any $\gamma\in {\bf \Gamma}_{R',\epsilon}^{<h_c}$, the total measure of $\partial A'_{\gamma}$ is at most $3|A'|$. By Theorem \ref{cuspedprecisedescription} (1) and our choice of $N$, for any $\gamma \in {\bf \Gamma}_{R',\epsilon}^{<h_c}$ (possibly $\gamma=L_{i,s}$) and any $B\subset N^1(\sqrt{\gamma})$, we have 
$$M_{\gamma}^1(N_{\frac{\epsilon}{R'}}(B))\geq M_{\gamma}^2(\tau(B))\ \text{and}\ M_{\gamma}^2(N_{\frac{\epsilon}{R'}}(B))\geq M_{\gamma}^1(\tau(B)).$$

As in Theorem \ref{cuspedprecisedescription} (2), by the Hall's marriage theorem, we can paste pants and umbrellas in $NA+N\bar{A}+A'$ to obtain an $(R',\epsilon)$-nearly geodesic subsurface $S'\looparrowright M$ with oriented boundary $L'$. Moreover, for any component $\bar{L}_{i,s}$ of $L'$, the foot of $S'$ on $\bar{L}_{i,s}$ is $(R',\epsilon)$-well matched with $\vec{w}_{i,s}$. So the $(R',\epsilon)$-nearly geodesic subsurface $S'\looparrowright M$ satisfies conditions (4) and (5).

Step I and step II together give the desired $(R',\epsilon)$-nearly geodesic subsurface with $(R,\epsilon)$-good boundary $f:S\looparrowright M$.

 \end{proof}
 
 \begin{remark}
 In Corollary 2.11 of \cite{Sun2} for closed hyperbolic $3$-manifolds, it states that $\sigma(L)=0$ if and only if $L$ is the oriented boundary of an $(R,\epsilon)$-nearly geodesic subsurface. The author believes the if and only if condition also holds in Proposition \ref{boundingsurface}. However, in Proposition \ref{boundingsurface}, we allow $(R,R',\epsilon)$-good pants and hamster wheels as building blocks of the subsurface. So it requires extra work to prove it, and we skip it in this paper.
 \end{remark}
 
 \begin{remark}\label{combinatorialdistance}
For the $(R',\epsilon)$-nearly geodesic subsurface with $(R,\epsilon)$-good boundary $S\looparrowright M$ constructed in Corollary \ref{boundingsurface}, the set of decomposition curves $\mathcal{C}\subset S$ gives a graph of space structure on $S$ with dual graph $\Gamma$. Each component of $S\setminus \mathcal{C}$ gives a vertex of $\Gamma$, and each component of $\mathcal{C}$ gives an edge of $\Gamma$. Let $v_i$ be the vertex of $\Gamma$ corresponding to $\Pi_i\subset S$. Then we can further modify the nearly geodesic subsurface $S\looparrowright M$ as in Section 3.1 Step IV of \cite{Sun1}, such that the combinatorial distance from $v_i$ to any $v_j\ne v_i$ and the combinatorial length of any essential closed curve in $\Gamma$ going through $v_i$ are at least $R'e^{R'}$.
\end{remark}

\bigskip
\bigskip

\section{The topological construction of virtual $1$-domination}\label{topo1}

In this section, we prove our main result Theorem \ref{main1}, modulo a $\pi_1$-injectivity result (Theorem \ref{pi1injectivity}). The construction is similar to the construction in \cite{LS} for closed hyperbolic $3$-manifolds. For some lemmas that work for both closed and cusped hyperbolic $3$-manifolds, we only state the results and refer readers to \cite{LS} for their proofs.

\subsection{A topological reduction}

At first, we follow the strategy in \cite{LS} to reduce the target manifold in Theorem \ref{main1} to a closed oriented hyperbolic $3$-manifold. The reduction in \cite{LS} is not very complete, since the cited result in \cite{BW} only concerns irreducible $3$-manifolds, while it only requires a bit of work to fix it.

We first prove a lemma about compact oriented $3$-manifolds (with or without boundary).

\begin{lemma}\label{reductionlemma}
Let $N$ be a compact oriented $3$-manifold with empty or tori boundary, then there exits a compact oriented irreducible $3$-manifold $M$ with empty or tori boundary that $1$-dominates $N$ and $|\partial M|=|\partial N|$.
\end{lemma} 

\begin{proof}
By Corollary 6.3 of \cite{Mye}, $N$ contains a hyperbolic knot $K$, i.e. $N\setminus \mathcal{N}(K)$ admits a complete finite volume hyperbolic structure. Here $\mathcal{N}(K)$ denotes a tubular neighborhood of $K$ homeomorphic to $D^2\times S^1$. 

Let $\Sigma_{1,1}$ be the one-punctured torus, and let $h:\partial(\Sigma_{1,1}\times S^1)=\partial \Sigma_{1,1} \times S^1\to \partial(\mathcal{N}(K))=\partial D^2 \times S^1$ be a homeomorphism defined by the product of two homeomorphisms on two components. Then we take $M=(N\setminus \mathcal{N}(K))\cup_{h}(\Sigma_{1,1}\times S^1)$, and $|\partial M|=|\partial N|$ clearly holds.

Since both $\partial (\mathcal{N}(K))\subset N\setminus \mathcal{N}(K)$ and $\partial (\Sigma_{1,1}\times S^1)\subset \Sigma_{1,1}\times S^1$ are  incompressible, while both $N\setminus \mathcal{N}(K)$ and $\Sigma_{1,1}\times S^1$ are irreducible, $M$ is irreducible. We take the desired degree-$1$ map $f:M=(N\setminus \mathcal{N}(K))\cup_{h}(\Sigma_{1,1}\times S^1)\to N$ to be identity on $N\setminus \mathcal{N}(K)$, and to be a pinching map cross the identity on $\Sigma_{1,1}\times S^1\to D^2\times S^1=\mathcal{N}(K)$.
\end{proof}

The following result is essentially proved by Boileau and Wang in Proposition 3.3 of \cite{BW}, while they assumed irreducibility of the target manifold.

\begin{proposition}\label{reductionclosed}
For any closed oriented $3$-manifold $N$, there exists a closed oriented hyperbolic $3$-manifold $M$, such that $M$ $1$-dominates $N$.
\end{proposition}

\begin{proof}
By Lemma \ref{reductionlemma}, $N$ is $1$-dominated by a closed oriented irreducible $3$-manifold $N'$. Since $N'$ is closed oriented and irreducible, Proposition 3.3 of \cite{BW} implies that $N'$ is $1$-dominated by a closed oriented hyperbolic $3$-manifold $M$. So the proof is done.
\end{proof}

By Proposition \ref{reductionclosed}, we can assume the target manifold $N$ in Theorem \ref{main1} is a closed oriented hyperbolic $3$-manifold.

\subsection{The construction of an immersed $2$-complex $j:Z\looparrowright M$}\label{constructZ1}

If $M$ is a closed oriented hyperbolic $3$-manifold, then Theorem \ref{main1} was already proved in \cite{LS}. So we can assume $M$ is not hyperbolic, and we also assume irreducibility of $M$ in this section. Since $M$ has positive simplicial volume, by \cite{Soma1}, $M$ has a hyperbolic JSJ piece $M_0$ with tori boundary, thus its interior admits a complete hyperbolic metric of finite volume. We also use $M_0$ to denote the corresponding non-compact oriented cusped hyperbolic $3$-manifold.

In this section, we will construct a $\pi_1$-injective map $j:Z\looparrowright M_0$ from a $2$-complex $Z$ to $M_0$, where $Z$ is modeled by a $2$-skeleton of $N$. The construction is similar to the construction in \cite{LS}, except that we use new tools adapted to cusped hyperbolic $3$-manifolds: Theorem \ref{enhancedconnectionprinciple} and Proposition \ref{boundingsurface}. 

We start with setting up the following initial data. 
\begin{data}\label{data}
\begin{enumerate}
\item A geometric triangulation of $N$. It means that, each edge of the triangulation is a geodesic segment in $N$, and each triangle is a totally geodesic triangle in $N$. By doing barycentric subdivision, we can assume that each edge has two distinct vertices and any two vertices are connected by at most one edge. The $1$- and $2$-skeletons of $N$ are denoted by $N^{(1)}$ and $N^{(2)}$ respectively.
\item A homological homomorphism and a homological lift: 
$$i_*:H_1(N;\mathbb{Z})\to H_1(M_0;\mathbb{Z})\ \text{and}\ i_1:N^{(1)}\to M_0.$$ 
Here $i_*$ can be any homomorphism between these two homology groups, and $i_1$ is an embedding such that $(i_1)_*:H_1(N^{(1)};\mathbb{Z})\to H_1(M_0;\mathbb{Z})$ factors through $i_*$.  The map $i_1$ can be constructed by first defining it on a spanning tree of $N^{(1)}$.
\item A pair of trivializations of $\text{SO}(3)$-principal bundles: $$t_N:\text{SO}(N)\to N\times \text{SO}(3)\ \text{and}\ t_{M_0}:\text{SO}(M_0)\to M_0\times \text{SO}(3).$$ Such trivializations exist since tangent bundles of orientable $3$-manifolds are trivial.
\end{enumerate}
\end{data}

Then we take the following geometric notations. 
\begin{notation}\label{notation}
\begin{enumerate}
\item We denote the set of vertices of the triangulation of $N$ by $V_N=\{n_1,n_2,\cdots,n_l\}.$ If there is an edge in $N^{(1)}$ between $n_i$ and $n_j$, we denote the oriented edge from $n_i$ to $n_j$ by $e_{ij}$, and $e_{ji}$ is its orientation reversal. If there is a triangle in $N^{(2)}$ spanned by (distinct) vertices $n_i,n_j,n_k$, we have a marked $2$-simplex $\Delta_{ijk}$.

\item For any oriented edge $e_{ij}$, let $\vec{v}_{ij}$ be the unit tangent vector of $e_{ij}$ based at $n_i$. For any marked $2$-simplex $\Delta_{ijk}$, we denote 
$$\vec{n}_{ijk}=\frac{\vec{v}_{ij}\times \vec{v}_{ik}}{|\vec{v}_{ij}\times \vec{v}_{ik}|},$$ 
which is a normal vector of $\Delta_{ijk}$ at $n_i$. Then we have a frame 
$${\bf F}_{ijk}=(n_i,\vec{v}_{ij},\vec{n}_{ijk})\in \text{SO}(N)_{n_i}$$ 
and we denote  
$$-{\bf F}_{ijk}=(n_i,-\vec{v}_{ij},-\vec{n}_{ijk}).$$ 
Now we have a finite collection of frames in $N$:
$$F_N=\{\pm {\bf F}_{ijk}\ |\ \Delta_{ijk}\ \text{is\ a\ marked\ 2-simplex\ in\ }N^{(2)}\}.$$

\item Let ${\bf i}_1:\text{SO}(N)|_{N^{(1)}}\to \text{SO}(M_0)$ be the morphism of $\text{SO}(3)$-principal bundles defined by 
$$\text{SO}(N)|_{N^{(1)}}\xrightarrow{t_N|}N^{(1)}\times \text{SO}(3)\xrightarrow{i_1\times \text{id}}M_0\times \text{SO}(3)\xrightarrow{t_{M_0}^{-1}}\text{SO}(M_0).$$ 
For the finite collection of points $V_N$ and frames $F_N$ in $N$, we define the corresponding collections in $M_0$ as following. We define $m_k=i_1(n_k)\in M_0$, and define ${\bf F}^{M_0}_{ijk}=(m_i,\vec{v}^{M_0}_{ij}, \vec{n}^{M_0}_{ijk})={\bf i}_1({\bf F}_{ijk})\in \text{SO}(M_0)$. Then we have finite collections 
$$V_{M_0}=i_1(V_N)\subset M_0,\ F_{M_0}={\bf i}_1(F_N)\subset \text{SO}(M_0).$$
Note that ${\bf i}_1(-{\bf F}_{ijk})=-{\bf i}_1({\bf F}_{ijk})=-{\bf F}_{ijk}^{M_0}$ holds since ${\bf i}_1$ is a morphism of $\text{SO}(3)$-principal bundles.

\item Since there are only finitely many marked triangles $\Delta_{ijk}$ and each $\Delta_{ijk}$ is a geodesic triangle, there exists $\phi_0\in(0,\frac{\pi}{2})$, such that all inner-angles of all $\Delta_{ijk}$ lie in $[\phi_0,\pi-\phi_0]$, and any two triangles in $N$ that share a vertex or an edge have dihedral angle in $[\phi_0,\pi]$.

\item For any vertex $n_i$ of $N$, the geometric triangulation of $N$ gives a subset $S_i\subset T_{n_i}^1N=S^2$, consisting of all unit vectors based at $n_i$ and tangent to some triangle $\Delta_{ijk}$. Then $S_i$ is a union of finitely many geodesic arcs in $T_{n_i}^1N=S^2$ and has an induced path metric.
Since each $S_i$ is compact and there are finitely many of them, there exists $\theta_0>0$, such that for any vertex $n_i$ of $N$ and any two vectors $\vec{v}_1,\vec{v}_2\in S_i$, if the angle between them in $T_{n_i}^1N$ is at most $\theta_0$, then their distance under the path metric of $S_i$ is at most $\phi_0$.
\end{enumerate}
\end{notation}

The construction of the immersed $2$-complex $j:Z\looparrowright M_0$ is instructed by a group homomorphism called {\it homological instruction}, which is constructed in the following proposition.

\begin{proposition}\label{homologicalinstruction}
There exists a group homomorphism $${\bf j}_*:H_1(\text{SO}(N),F_N;\mathbb{Z})\to H_1(\text{SO}(M_0),F_{M_0};\mathbb{Z})$$ such that the following diagram commutes:
\begin{diagram}
H_1(\text{SO}(N)|_{V_N},F_N;\mathbb{Z}) & \rTo & H_1(\text{SO}(N)|_{N^{(1)}},F_N;\mathbb{Z}) & \rTo& H_1(\text{SO}(N),F_N;\mathbb{Z}) & \rTo^{\partial} & H_0(F_N;\mathbb{Z}) \\
\dTo^{({\bf i}_1)_*}                                  &         & \dTo^{({\bf i}_1)_*}                &       & \dTo^{{\bf j}_*}                           &                             &  \dTo^{({\bf i}_1)_*}\\
H_1(\text{SO}(M_0)|_{V_{M_0}},F_{M_0};\mathbb{Z}) & \rTo & H_1(\text{SO}(M_0),F_{M_0};\mathbb{Z}) &\rTo^{id_*} & H_1(\text{SO}(M_0),F_{M_0};\mathbb{Z}) & \rTo^{\partial} & H_1(F_{M_0};\mathbb{Z}).
\end{diagram}
Here horizontal homomorphisms $\partial$ are boundary homomorphisms in homological long exact sequences of pairs, horizontal homomorphisms that have no names are induced by inclusions. Vertical maps $({\bf i}_1)_*$ are induced by restrictions of ${\bf i}_1$ (defined in Notation \ref{notation} (3)).
\end{proposition}

The proof of Proposition \ref{homologicalinstruction} is given in Proposition 3.1 of \cite{LS}, which works well for cusped hyperbolic $3$-manifolds, so we skip the proof. Here ${\bf j}_*$ is constructed by proving that $({\bf i}_1)_*:H_1(\text{SO}(N)|_{N^{(1)}},F_N;\mathbb{Z})\to H_1(\text{SO}(M_0),F_{M_0};\mathbb{Z})$ factors through $H_1(\text{SO}(N),F_N;\mathbb{Z})$, by using Data \ref{data} (2) and Notation \ref{notation} (3).

Given the homological instruction 
$${\bf j}_*:H_1(\text{SO}(N),F_N;\mathbb{Z})\to H_1(\text{SO}(M_0),F_{M_0};\mathbb{Z})$$
in Proposition \ref{homologicalinstruction}, we are ready to construct a $2$-complex $Z$ and a map $j:Z\looparrowright M$.

The $2$-complex $Z$ consists of vertices, edges and $2$-dim pieces that are compact orientable surfaces with connected boundaries. These building blocks have a one-to-one correspondence with $0$-, $1$- and $2$-cells of $N^{(2)}$. We will build complexes $Z^{(0)}$, $Z^{(1)}$, $Z$ and maps $j^0:Z^{(0)}\to M_0$, $j^1:Z^{(1)}\to M_0$, $j:Z\to M_0$ inductively.

We construct the vertex set $Z^{(0)}$ of $Z$ and $j^0:Z^{(0)}\to M_0$ as following.
\begin{construction}\label{constructionZ0}
We define $Z^{(0)}$ to be a finite set $\{v_1,\cdots,v_l\}$ that has the same cardinality as $N^{(0)}$. We define the map $j^0:Z^{(0)}\to M_0$ by sending each $v_k$ to $m_k=i_1(n_k)\in V_{M_0}$.
\end{construction}

The indices of vertices induce a total order on $N^{(0)}$ and a total order on $Z^{(0)}$. Given this order, for any (unoriented) edge in $N^{(1)}$ between $n_i$ and $n_j$ with $i<j$, it has a preferred orientation from $n_i$ to $n_j$ and this oriented edge is denoted by $e_{ij}$.

Since there are only finitely many edges $e_{ij}$ in $N$, for any small $\epsilon>0$, we take $R>3I(\pi-\phi_0)$ and large enough so that we can apply Theorem \ref{enhancedconnectionprinciple} to do the following constructions.

\begin{construction}\label{constructionZ1}
\begin{enumerate}
\item For indices $i<j$, if $n_i$ and $n_j$ give an oriented edge $e_{ij}$ in $N^{(1)}$, $Z$ has an oriented edge $e_{ij}^Z$ from $v_i$ to $v_j$. So we obtain the $1$-skeleton $Z^{(1)}$.

\item To construct the map $j^1:Z^{(1)}\to M_0$, as an extension of $j^0:Z^{(0)}\to M_0$, we define a map on each oriented segment $e_{ij}^Z$ with $i<j$. We apply Theorem \ref{enhancedconnectionprinciple} to construct a $\partial$-framed segment $\mathfrak{s}_{ij}$ in $M_0$ from $m_i=j^0(v_i)$ to $m_j=j^0(v_j)$ such that the following conditions hold, and we map $e_{ij}^Z$ to the carrier of $\mathfrak{s}_{ij}$ via an orientation preserving linear map.
\begin{enumerate}
\item Then length and phase of $\mathfrak{s}_{ij}$ are $\frac{\epsilon}{10}$-close to $R$ and $0$ respectively, and the height of $\mathfrak{s}_{ij}$ is at most $2\log{R}-1$.

\item Let $k$ be the smallest index so that $n_i,n_j,n_k$ span a triangle in $N^{(2)}$, then the initial and terminal frames of $\mathfrak{s}_{ij}$ are $\frac{\epsilon}{10}$-close to ${\bf F}_{ijk}^{M_0}\in \text{SO}(M)_{m_i}$ and $-{\bf F}_{jik}^{M_0}\in \text{SO}(M)_{m_j}$ respectively.

\item  Let $s_{ij}$ be the carrier of $\mathfrak{s}_{ij}$, then we have
$$[{\bf F}_{ijk}^{M_0}\xrightarrow{s_{ij}}(-{\bf F}_{jik}^{M_0})]={\bf j}_*[{\bf F}_{ijk}\xrightarrow{e_{ij}}(-{\bf F}_{jik})]\in H_1(\text{SO}(M_0),\{{\bf F}_{ijk}^{M_0},-{\bf F}_{jik}^{M_0}\};\mathbb{Z}).$$
\end{enumerate}
\end{enumerate}
\end{construction}

Note that although $Z^{(1)}$ is isomorphic to $N^{(1)}$, $i_1:N^{(1)}\to M_0$ and $j^1:Z^{(1)}\to M_0$ are two different maps, while they induce the same homomorphism on $H_1$ (by Proposition \ref{homologicalinstruction}).

The construction of the geodesic segment $s_{ij}$ seems depend on the choice of the smallest $k$ in Construction \ref{constructionZ1} (2) (b). The following lemma proves that, once $s_{ij}$ is constructed with respect to the smallest $k$, it actually works for any other $k'$.

\begin{lemma}\label{welldefine}
For any $k'\ne k$ such that $n_i,n_j,n_{k'}$ span a triangle in $N$, $s_{ij}$ satisfies 
$$[{\bf F}_{ijk'}^{M_0}\xrightarrow{s_{ij}}(-{\bf F}_{jik'}^{M_0})]={\bf j}_*[{\bf F}_{ijk'}\xrightarrow{e_{ij}}(-{\bf F}_{jik'})]\in H_1(\text{SO}(M_0),\{{\bf F}_{ijk'}^{M_0},-{\bf F}_{jik'}^{M_0}\};\mathbb{Z}).$$ 
\end{lemma}

The proof of this statement is given in Lemma 3.1 of \cite{LS}, which works well for oriented cusped hyperbolic $3$-manifolds, so we skip the proof. The idea of proof works as following. Since the two geodesic triangles $\Delta_{ijk}$ and $\Delta_{ijk'}$ in $N$ share a geodesic segment $e_{ij}$, there is a unique $A\in \text{SO}(3)$ such that ${\bf F}_{ijk}={\bf F}_{ijk'}\cdot A$ and $(-{\bf F}_{jik})=(-{\bf F}_{jik'})\cdot A$ hold. Since ${\bf i}_1$ is a morphism of $\text{SO}(3)$-principal bundles, we have ${\bf F}_{ijk}^{M_0}={\bf F}_{ijk'}^{M_0}\cdot A$ and $(-{\bf F}_{jik}^{M_0})=(-{\bf F}_{jik'}^{M_0})\cdot A$.

For any $n_i,n_j,n_k$ in $N^{(0)}$ that span a geodesic triangle $\Delta_{ijk}$ in $N^{(2)}$, we have constructed a triple of oriented geodesic segments $s_{ij}, s_{jk}, s_{ki}$ in $M_0$ ($s_{ij}$ denotes the orientation reversal of $s_{ji}$ if $i>j$). The following lemma shows that the cyclic concatenation $s_{ij}s_{jk}s_{ki}$ is homtopic to a null-homologous good curve in $M_0$ that bounds a nearly geodesic subsurface.

\begin{lemma}\label{triplenullhomologous}
For any  $n_i,n_j,n_k\in N^{(0)}$ that span a geodesic triangle $\Delta_{ijk}$ in $N^{(2)}$, the cyclic concatenation of geodesic segments $s_{ij} s_{jk} s_{ki}$ in $M_0$ is homotopic to an oriented closed geodesic $\gamma_{ijk}$ such that the following hold.
\begin{enumerate}
\item There exists an immersed annulus $A_{ijk}$ in $M_0$ bounded by the cyclic concatenation $s_{ij}s_{jk}s_{ki}$ and $\gamma_{ijk}$, and it is nearly geodesic in the following sense.
\begin{enumerate}
\item The annulus $A_{ijk}$ has a triangulation with six vertices. Three of them are $m_i,m_j,m_k$ and three of them lie on $\gamma_{ijk}$.

\item The trianguation of $A_{ijk}$ consists of six totally geodesic triangles in $M_0$, and any two that share an edge have bending angle at most $\epsilon$.

\item At vertex $m_i$, the normal vector of ${\bf F}_{ijk}^{M_0}$ is $\epsilon$-close to the normal vector of any totally geodesic triangle with vertex $m_i$. The same holds for $m_j$ and $m_k$.

\end{enumerate}
\item The oriented closed geodesic $\gamma_{ijk}$ is good and null-homologous in the following sense.
\begin{enumerate}

\item The closed geodesic $\gamma_{ijk}$ is an $(R_{ijk},\epsilon)$-good curve of height at most $2\log{R_{ijk}}$. Here 
$R_{ijk}=\frac{3}{2}R-\frac{1}{2}(I(\pi-\theta_{ijk})+I(\pi-\theta_{jki})+I(\pi-\theta_{kij}))\in [\frac{3}{2}R-\frac{3}{2}I(\pi-\phi_0),\frac{3}{2}R],$
and $R_{ijk}>R$ holds. Here $\theta_{ijk}$ is the inner angle of $\Delta_{ijk}$ at vertex $n_j$.

\item The closed geodesic $\gamma_{ijk}$ is null-homologous in $M$, and $\sigma(\gamma_{ijk})=0\ \text{mod}\ 2$ for the invariant $\sigma$ in Proposition \ref{boundingsurface}.
\end{enumerate}
\end{enumerate}
\end{lemma}

Estimates of the immersed annulus in item (1) follows from Lemma 3.5 of \cite{Sun2}. Estimates of the good curve $\gamma_{ijk}$ in item (2) follow from Construction \ref{constructionZ1} (2) (a) (b), Lemmas \ref{lengthphase} (2) and \ref{distance}.  The homological properties of $\gamma_{ijk}$ in item (2) follow from Construction \ref{constructionZ1} (2) (c) and the proof of Lemma 3.2 of \cite{LS}, which works well for oriented cusped hyperbolic $3$-manifolds.

Now we are ready to construct the map $j:Z\looparrowright M$. Here we take a large integer $R'$ greater than all the $R_{ijk}$.

\begin{construction}\label{constructionZ}
Let $n_i,n_j,n_k\in N^{(0)}$ be a triple of vertices that span a geodesic triangle $\Delta_{ijk}$ in $N^{(2)}$. By Lemma \ref{triplenullhomologous} (1), the cyclic concatenation $s_{ij}s_{jk}s_{ki}$ and the $(R_{ijk},\epsilon)$-good curve $\gamma_{ijk}$ cobound a nearly geodesic annulus $A_{ijk}$ in $M_0$. We take the shortest geodesic segment in $A_{ijk}$ from $\gamma_{ijk}$ to a preferred vertex $m_i$, let its tangent vector at the intersection with $\gamma_{ijk}$ be $\vec{w}_{ijk}$, and let $\vec{v}_{ijk}=\vec{w}_{ijk}+(1+\pi i)\in N^1(\sqrt{\gamma_{ijk}})$. 

By Lemma \ref{triplenullhomologous} (2), for the $(R_{ijk},\epsilon)$-good curve $\gamma_{ijk}$ in $M_0$ with height at most $2\log{R_{ijk}}$, we have $\sigma(\gamma_{ijk})=0\ \text{mod}\ 2$. Then we apply Proposition \ref{boundingsurface} and Remark \ref{combinatorialdistance} to construct an $(R',\epsilon)$-nearly geodesic subsurface with $(R_{ijk},\epsilon)$-good boundary $S_{ijk}\looparrowright M$, such that the following hold. 
\begin{enumerate}
\item The oriented boundary of $S_{ijk}$ is $\gamma_{ijk}$, and its foot on $\gamma_{ijk}$ is $\frac{\epsilon}{R}$-close to $\vec{v}_{ijk}$.
\item Any essential path in $S_{ijk}$ from $\gamma_{ijk}$ to itself has combinatorial length (with respect to the decomposition of $S_{ijk}$ to pants and hamster wheels) at least $R'e^{R'}$.
\end{enumerate}
For the corresponding triple of vertices $v_i,v_j,v_k\in Z^{(0)}$, we paste the annulus $A_{ijk}$ to $Z^{(1)}$ along $e_{ij}^Z\cup e_{jk}^Z \cup e_{ki}^Z$ and paste the surface $S_{ijk}$ with $A_{ijk}$ along $\gamma_{ijk}$. After doing this process for all such triple of vertices in $N$, we obtain the desired $2$-complex $Z$. The map $j:Z\looparrowright M_0$ is given by $j^1:Z^{(1)}\looparrowright M_0$ and the maps from $A_{ijk}$ and $S_{ijk}$ to $M_0$.
\end{construction}

\bigskip

\subsection{Construction of the virtual $1$-domination}\label{contructdomination1}

In this section, we construct the desired finite cover $M'$ of $M$ and the degree-$1$ map $f:M'\to N$, modulo a $\pi_1$-injectivity result (Theorem \ref{pi1injectivity}).

To construct the finite cover $M'$ of $M$, we need to prove that $j:Z\looparrowright M_0$ is $\pi_1$-injective and $j_*(\pi_1(Z))<\pi_1(M_0)$ is a convex cocompact subgroup of $\text{Isom}_+(\mathbb{H}^3)$. For a torsion-free discrete subgroup $\Gamma<\text{Isom}_+(\mathbb{H}^3)$ , it has a limit set $\Lambda_{\Gamma}\subset \partial \mathbb{H}^3$ (the set of limit points of a $\Gamma$-orbit) and a convex hull $\text{Hull}_{\Gamma}\subset \mathbb{H}^3$ (the intersection between $\mathbb{H}^3$ and the minimal convex subset of $\mathbb{H}^3\cup \partial \mathbb{H}^3$ containing $\Lambda_{\Gamma}$). Recall that $\Gamma<\text{Isom}_+(\mathbb{H}^3)$ is a {\it convex cocompact} subgroup if $C_{\Gamma}=\text{Hull}_{\Gamma}/\Gamma$ is compact.

To state the $\pi_1$-injectivity result, we also need to define a compact oriented $3$-manifold $\mathcal{Z}$. The geodesic triangulation of the closed hyperbolic $3$-manifold $N$ also gives a handle decomposition of $N$. Let $\mathcal{N}^{(1)}$ be the union of $0$- and $1$-handles of $N$, which is naturally oriented. Each triangle $\Delta_{ijk}$ in $N^{(2)}$ gives a $2$-handle, and it is pasted to $\mathcal{N}^{(1)}$ along an annulus $L_{ijk}\subset \partial \mathcal{N}^{(1)}$, while these annuli are disjoint from each other. For each $\Delta_{ijk}$, we take a homeomorphism $\partial S_{ijk}\times I\to L_{ijk}$ ($S_{ijk}$ is defined in Construction \ref{constructionZ}). By using these homeomophisms to paste $\mathcal{N}^{(1)}$ and $\{S_{ijk}\times I\}$ together, we obtain a compact oriented $3$-manifold $\mathcal{Z}$.

Our technical result on the $\pi_1$-injectivity of $j:Z\looparrowright M_0$ is stated as following.

\begin{theorem}\label{pi1injectivity}
For any small $\epsilon>0$, there exists $R_0>0$ depend on $M_0$ and $\epsilon$, such that the following hold for any $R>R_0$. If we construct $j:Z\looparrowright M_0$ as in Constructions \ref{constructionZ1} and \ref{constructionZ} with respect to $R$ and $\epsilon$, then $j:Z\looparrowright M_0$ is $\pi_1$-injective and $j_*(\pi_1(Z))<\pi_1(M_0)$ is a convex cocompact subgroup. Moreover, the convex core of the covering space of $M_0$  corresponding to $j_*(\pi_1(Z))<\pi_1(M_0)$ is homeomorphic to the compact $3$-manifold $\mathcal{Z}$, as oriented manifolds.
\end{theorem}

The proof of Theorem \ref{pi1injectivity} is in different flavor from other material in this section, so we defer the proof to Section \ref{pi1inj1}.

\begin{lemma}\label{properdomination}
There exists a degree-$1$ proper map $(\mathcal{Z},\partial \mathcal{Z})\to (\mathcal{N}^{(2)}, \partial \mathcal{N}^{(2)})$.
\end{lemma}

The construction of such a degree-$1$ proper map is described in the proof of Lemma 3.4 of \cite{LS}. This map maps $\mathcal{N}^{(1)}\subset \mathcal{Z}$ to $\mathcal{N}^{(1)}\subset \mathcal{N}^{(2)}$ by the identity map. Then it maps each $S_{ijk}\times I\subset \mathcal{Z}$ to the $2$-handle in $\mathcal{N}^{(2)}$ corresponding to $\Delta_{ijk}$ (homeomorphic to $\Delta_{ijk}\times I$) by $p_{ijk}\times \text{id}_I$, where $p_{ijk}:S_{ijk}\to \Delta_{ijk}$ is a degree-$1$ pinching map.

Now we are ready to construct the virtual degree-$1$ map, thus finish the proof of Theorem \ref{main1}.

\begin{proof}[Proof of Theorem \ref{main1}]
At first, by Lemma \ref{reductionclosed}, we can assume that $N$ is a closed oriented hyperbolic $3$-manifold. 

We can assume that $M$ is a prime $3$-manifold by the following argument. If $M$ is not prime, then $M=M_1\#M_2$ for a closed oriented prime $3$-manifold $M_1$ with positive simplicial volume. If we know that $M_1$ has a degree-$d$ finite cover $M_1'$ that admits a degree-$1$ map $f':M_1'\to N$, then we take 
$$M'=M_1'\# (\#_{i=1}^d M_2).$$ 
We can assume that the $d$ decomposition spheres in $M_1'$ are mapped to the same $2$-sphere in $M_1$ that bounds a $3$-ball. Then $M'$ is a degree-$d$ cover of $M$, by mapping $M_1'$ to $M_1$ via the degree-$d$ covering map and mapping each $M_2$ prime summand of $M'$ to the $M_2$ prime summand of $M$ by identity. Moreover, $M'$ admits a degree-$1$ map to $N$ defined by $M'\to M_1'\xrightarrow{f'} N$, where the first map pinches each $M_2$ prime summand of $M'$ to a $3$-ball in $M_1'$.

So we can assume that $M$ is a prime $3$-manifold. If $M$ is a closed oriented hyperbolic $3$-manifold, the result was already proved in \cite{LS}, so we can assume that $M$ is not hyperbolic. Since $M$ has positive simplicial volume, it has nontrivial JSJ decomposition and has a hyperbolic JSJ piece $M_0$. We also use $M_0$ to denote the non-compact finite volume hyperbolic $3$-manifold homeomorphic to the interior of $M_0$.

We first fix a geometric triangulation of $N$ as in the beginning of Section \ref{constructZ1}. Then we take a small constant $\epsilon>0$, and take a large enough $R>0$ such that Constructions \ref{constructionZ0}, \ref{constructionZ1} and \ref{constructionZ} and Theorem \ref{pi1injectivity} work.  Based on the triangulation of $N$, by Constructions \ref{constructionZ0}, \ref{constructionZ1} and \ref{constructionZ}, we construct a map $j:Z\looparrowright M_0$. By Theorem \ref{pi1injectivity}, $j:Z\looparrowright M_0$ is $\pi_1$-injective with convex cocompact image $j_*(\pi_1(Z))<\pi_1(M_0)<\text{Isom}_+(\mathbb{H}^3)$. Moreover, let $\tilde{M}_0$ be the covering space of $M_0$ corresponding to $j_*(\pi_1(Z))<\pi_1(M_0)$, then its convex core $C$ is homeomorphic to the compact $3$-manifold $\mathcal{Z}$ in Lemma \ref{properdomination}, as oriented manifolds.

Since the convex core $C$ of $\tilde{M}_0$ is compact, we delete some cusp ends of $M_0$ to get a compact submanifold $M_0^c\subset M_0$, so that the covering map $\tilde{M}_0\to M_0$ maps $C$ into $M_0^c$. We identify the compact manifold $M_0^c$ with the JSJ piece of $M$ corresponding to $M_0$.

Now we consider the subgroup $j_*(\pi_1(Z))<\pi_1(M_0^c)<\pi_1(M)$. Let $\tilde{M}$ be the covering space of $M$ corresponding to $j_*(\pi_1(Z))<\pi_1(M)$, then it contains a submanifold homeomorphic to $\tilde{M}_0^c$, the preimage of $M_0^c$ in $\tilde{M}_0$, and we have a compact submanifold $C\subset \tilde{M}_0^c\subset \tilde{M}$. Since $j_*(\pi_1(Z))$ is contained in a vertex subgroup of $\pi_1(M)$ (under the JSJ decomposition), by the characterization of separable subgroups in \cite{Sun3} (which heavily relies on the LERFness of hyperbolic $3$-manifold groups in \cite{Agol}), $j_*(\pi_1(Z))$ is separable in $\pi_1(M)$.

By Scott's topological characterization of separable subgroups (\cite{Sco}), there is an intermediate finite cover $M'\to M$ of $\tilde{M}\to M$, such that the compact submanifold $C\subset \tilde{M}$ is mapped into $M'$ via embedding, thus $M'$ contains a submanifold $C$ homeomorphic to $\mathcal{Z}$, as oriented manifolds. By Lemma \ref{properdomination}, there is a degree-$1$ proper map $f_0:C\to \mathcal{N}^{(2)}$. We take a neighborhood $\mathcal{N}(C)$ of $C$ in $M'$, then $\mathcal{N}(C)\setminus C$ is homeomorphic to $\partial C\times I$. Since each component of $N\setminus \mathcal{N}^{(2)}$ is a $3$-handle, there is a map $f_1:\mathcal{N}(C)\to N$ extending $f_0$, that maps each component of $\mathcal{N}(C)\setminus C$ to a $3$-handle of $N$, and maps each component of $\partial \mathcal{N}(C)$ to the center of the corresponding $3$-handle. 

Then we construct $f:M'\to N$ as an extension of $f_1$, by sending $M'\setminus \mathcal{N}(C)$ to a graph in $N$ that connects centers of $3$-handles of $N$. Moreover, we can assume that this graph in $N$ misses $\mathcal{N}^{(1)}$. Then for any point $n\in \mathcal{N}^{(1)}$, $f^{-1}(n)$ consists of a unique point in the submanifold of $C$ homeomorphic to $\mathcal{N}^{(1)}$, and $f$ is locally orientation preserving at $f^{-1}(n)$. So the degree of $f$ is $1$ and the proof is done.

\end{proof}

\bigskip
\bigskip

\section{Proof of Theorem \ref{pi1injectivity}}\label{pi1inj1}

In this section, we prove that the map $j:Z\looparrowright M_0$ constructed in Section \ref{constructZ1} is $\pi_1$-injective (Theorem \ref{pi1injectivity}), thus finish the proof of Theorem \ref{main1}.

For a $2$-complex $Z$ and a map $j:Z\looparrowright M$ to a closed hyperbolic $3$-manifold  arised from the good pants construction, many previous works (in \cite{Sun1}, \cite{Sun2}, \cite{Liu}, \cite{LS}) on the $\pi_1$-injectivity of $j$ assumed a finer estimate: 
$$|{\bf hl}(\gamma)-R|<\frac{\epsilon}{R} \ \text{and}\ |{\bf foot}_{\gamma}(\Pi_1,c_1)-{\bf foot}_{\bar{\gamma}}(\Pi_2,c_2)-(1+\pi i)|<\frac{\epsilon}{R^2},$$ instead of the condition
$$|{\bf hl}(\gamma)-R|<\epsilon \ \text{and}\ |{\bf foot}_{\gamma}(\Pi_1,c_1)-{\bf foot}_{\bar{\gamma}}(\Pi_2,c_2)-(1+\pi i)|<\frac{\epsilon}{R}$$ given in Section \ref{prekahnmarkovic}.

This finer estimate was essential for proving the $\pi_1$-injectivity in aforementioned works, but we do not have it in cusped hyperbolic $3$-manifolds anymore. At first, $(R,\epsilon)$-nearly geodesic subsurfaces in cusped hyperbolic $3$-manifolds also contain $(R,\epsilon)$-hamster wheels, so previous works in closed hyperbolic $3$-manifolds are not applicable. Moreover, in \cite{KW1}, when pasting an $(R,\epsilon)$-hamster wheel along its outer boundary, the $(R,\epsilon)$-well-matched condition only requires that two slow and constant turning normal fields are matched up to a $(100\epsilon)$-error. This $(100\epsilon)$-error can not be replaced by $\frac{100\epsilon}{R}$ easily, since it was essential for proving finiteness of  umbrellas in \cite{KW1}.

The proof of Theorem \ref{pi1injectivity} consists of three parts. In Section \ref{estimatesurface}, we give an estimate of the geometry of $(R,\epsilon)$-nearly geodesic subsurfaces, based on the work in \cite{KW1}. In Section \ref{estimate1}, we define an ideal model of $\tilde{j}:\tilde{Z}\looparrowright \tilde{M}_0=\mathbb{H}^3$ on universal covers, consisting of a representation $\rho_0:\pi_1(Z)\to \text{Isom}_+(\mathbb{H}^3)$ and a $\rho_0$-equivariant map $\tilde{j}_0:\tilde{Z}\to \mathbb{H}^3$, such that $\tilde{j}_0$ maps each component of $\tilde{Z}\setminus \tilde{Z}^{(1)}$ to a totally geodesic subsurface in $\mathbb{H}^3$. Then we prove that $\tilde{j}_0$ is both a quasi-isometric embedding and an embedding, and figure out the topology of the convex core of $\rho_0(\pi_1(Z))$. In Section \ref{qi1}, we finish the prove of Theorem \ref{pi1injectivity}, by proving that $\tilde{j}:\tilde{Z}\to \mathbb{H}^3$ is close to $\tilde{j}_0:\tilde{Z}\to \mathbb{H}^3$.

We give one more technical comment on the proof of Theorem \ref{pi1injectivity}. In our proof, we also need an interpolation between $\tilde{j}_0:\tilde{Z}\to \mathbb{H}^3$ and $\tilde{j}:\tilde{Z}\to \mathbb{H}^3$, since we need to figure our the topology of the convex core of $j_*(\pi_1(Z))$. The proof of the $\pi_1$-injectivity result in \cite{KW1} does not need such an interpolation, since the topology of the convex core is simple for surface groups. In \cite{CG}, Chu and Groves used a local-global argument to prove the $\pi_1$-injectivity of mapped-in $2$-complexes in cusped hyperbolic $3$-manifolds, but they do not need to figure out the topology of the convex core, so they did not use an interpolation either.

\subsection{An estimate on $(R,\epsilon)$-nearly geodesic subsurfaces}\label{estimatesurface}

In this section, we prove an estimate on $(R,\epsilon)$-near geodesic subsurfaces. 

At first, we review some works in \cite{KW1}.
By considering the left action of $\text{Isom}_+(\mathbb{H}^3)$ on $\text{SO}(\mathbb{H}^3)$, once we fix a base frame $v_0\in \text{SO}(\mathbb{H}^3)$, we have a bijection between  $\text{Isom}_+(\mathbb{H}^3)$ and
$\text{SO}(\mathbb{H}^3)$. The canonical Riemannian metric on  $\text{SO}(\mathbb{H}^3)$ induces a metric on $\text{Isom}_+(\mathbb{H}^3)$. For two frames $u,v\in \text{SO}(\mathbb{H}^3)$, by identifying them with elements in $\text{Isom}_+(\mathbb{H}^3)$, there is a unique element $(u\to v)\in \text{Isom}_+(\mathbb{H}^3)$ such that $u\cdot (u\to v)=v$ holds. Note that for any $g\in \text{Isom}_+(\mathbb{H}^3)$, we have $(u\to v)=(gu\to gv)\in \text{Isom}_+(\mathbb{H}^3)$. The following definition (in Section A.3 of \cite{KW1}) describes a map that locally behaves like an isometry.

\begin{definition}\label{boundeddistortion}
For a subset $X\subset \text{SO}(\mathbb{H}^3)$, a map $\tilde{e}:X\to \text{SO}(\mathbb{H}^3)$ is {\it $\epsilon$-distortion to distance $D$} if for any $u,v\in X$ with $d(u,v)<D$, we have $$d\big((u\to v),(\tilde{e}(u)\to \tilde{e}(v))\big)<\epsilon.$$
\end{definition}

For an $(R,\epsilon)$-nearly geodesic subsurface (possibly with boundary) $i:S\looparrowright M_0$ in a cusped hyperbolic $3$-manifold, it is given by an $(R,\epsilon)$-assembly $\mathcal{A}$. Let $\mathcal{C}$ be the family of unoriented curves that decomposes $S$ into pants and hamster wheels. In Section A.5 of \cite{KW1}, Kahn and Wright constructed a hyperbolic structure on $S$ such that each component of $S\setminus \mathcal{C}$ is isometric to the perfect pants or the perfect hamster wheel. 
This construction gives a perfect assembly $\hat{\mathcal{A}}$ corresponding to $\mathcal{A}$. 

Now we equip $S$ with this hyperbolic structure, and use $\hat{\mathcal{C}}$ to denote the the union of curves in $\mathcal{C}$ and $\partial S$. We define the following two subsets of $\text{SO}(S)$ and $\text{SO}(M_0)$:
$$\partial^{\mathcal{F}} \hat{\mathcal{A}}=\{(p,\vec{v},\vec{n})\in \text{SO}(S) \ |\ p\in \hat{\mathcal{C}},\ \vec{v}\ \text{is\ tangent\ to\ }\hat{\mathcal{C}}, \vec{n} \ \text{is\ inward\ pointing\ if\ }p\in \partial S \},$$
$$\partial^{\mathcal{F}} \mathcal{A}=\{(p,\vec{v},\vec{n})\in \text{SO}(M_0) \ |\ p\in i(\hat{\mathcal{C}}),\ \vec{v}\ \text{is\ tangent\ to\ }i(\hat{\mathcal{C}})\}.$$
Note that $\partial^{\mathcal{F}} \hat{\mathcal{A}}$ is a compact $1$-dim manifold and each component is homeomorphic to $S^1$, while $\partial^{\mathcal{F}} \mathcal{A}$ is a compact $2$-dim manifold and each component is homeomorphic to $T^2$. Since each frame in $\partial^{\mathcal{F}}\hat{\mathcal{A}}$ is determined by a
normal vector of $\hat{\mathcal{C}}$, we simply use normal vectors to represent frames in $\partial^{\mathcal{F}}\hat{\mathcal{A}}$.

Now we define a map $e:\partial^{\mathcal{F}}\hat{\mathcal{A}}\to  \partial^{\mathcal{F}}\mathcal{A}$ as in \cite{KW1}. For each component $c$ of $\mathcal{C}$, it is adjacent to two pieces $\Sigma_1,\Sigma_2$ of $S\setminus \mathcal{C}$, and we give $c$ an orientation so that $\Sigma_1$ lies to its left. There are two $S^1$-components $c_1,c_2$ of $\partial^{\mathcal{F}}\hat{\mathcal{A}}$ corresponding to $c$, and we assume that normal vectors of $c_1$ point into $\Sigma_1$ and normal vectors of $c_2$ point into $\Sigma_2$. Suppose that $c$ is mapped to $\gamma\in {\bf \Gamma}_{R,\epsilon}$, then $e$ maps $c_1$ and $c_2$ to slow and constant turning normal fields in $N^1(\gamma), N^1(\bar{\gamma})\subset \partial^{\mathcal{F}}\mathcal{A}$, such that their tangent vectors give the orientation of $\gamma$ and $\bar{\gamma}$, and it maps ${\bf foot}_c(\Sigma_1,c)\in c_1$ and ${\bf foot}_c(\bar{\Sigma}_2,c)\in c_2$ to ${\bf foot}_{\gamma}(\Sigma_1,c)$ and ${\bf foot}_{\gamma}(\bar{\Sigma}_2,c)$ respectively. For $\vec{n}_1\in c_1$ and $\vec{n}_2\in c_2$ based at the same point in $c$, note that $e(\vec{n}_1)$ and $e(\vec{n}_2)$ may not be based at the same point of $\gamma$, but these two basepoints have distance at most $\frac{2\epsilon}{R}$. Similarly, for any component $c\subset \partial S$, it gives a unique component of $\partial^{\mathcal{F}}\hat{\mathcal{A}}$, and we define $e$ on this component similarly.

Let $\tilde{i}:\tilde{S}\subset \mathbb{H}^2\to \tilde{M}_0=\mathbb{H}^3$ be the lifting of $i:S \looparrowright M_0$ to the universal cover. Let $\partial^{\mathcal{F}}\tilde{\hat{\mathcal{A}}}$ and $\partial^{\mathcal{F}}\tilde{\mathcal{A}}$ be the preimages of  $\partial^{\mathcal{F}}\hat{\mathcal{A}}$ and $\partial^{\mathcal{F}}\mathcal{A}$ in $\text{SO}(\mathbb{H}^2)$ and $\text{SO}(\mathbb{H}^3)$ respectively.  On the universal cover, the map $e:\partial^{\mathcal{F}}\hat{\mathcal{A}}\to  \partial^{\mathcal{F}}\mathcal{A}$ and the lifting  $\tilde{i}:\tilde{S}\to \tilde{M}_0$ induce a map 
$$\tilde{e}:\partial^{\mathcal{F}}\tilde{\hat{\mathcal{A}}}\subset \text{SO}(\mathbb{H}^2)\to  \partial^{\mathcal{F}}\tilde{\mathcal{A}}\subset \text{SO}(\mathbb{H}^3).$$

We say that $e:\partial^{\mathcal{F}}\hat{\mathcal{A}}\to  \partial^{\mathcal{F}}\mathcal{A}\subset \text{SO}(\mathbb{H}^3)$ is {\it $\epsilon$-distorted to distance $D$} if $\tilde{e}:\partial^{\mathcal{F}}\tilde{\hat{\mathcal{A}}}\to  \partial^{\mathcal{F}}\tilde{\mathcal{A}}$ is $\epsilon$-distorted to distance $D$ as in Definition \ref{boundeddistortion}. Then Kahn and Wright proved the following result in \cite{KW1}.

\begin{theorem}\label{localestimateKW}[Theorem A.18 of \cite{KW1}]
For any $D>0$, there are constants $C,R_0>1$, such that for any $\epsilon>0$ and any $R>R_0$, the following statement holds. Let $\mathcal{A}$ be an $(R,\epsilon)$-good assembly, then there is a perfect assembly $\hat{\mathcal{A}}$ and a map $e:\partial^{\mathcal{F}}\hat{\mathcal{A}}\to  \partial^{\mathcal{F}}\mathcal{A}$ that is $C\epsilon$-distorted at distance $D$.
\end{theorem}

Now we state our estimate on $(R,\epsilon)$-nearly geodesic subsurfaces, which describes their global geometry. At first, we need to set up some notations.

Recall that $\hat{\mathcal{C}}$ denotes the union of curves in $\mathcal{C}$ and $\partial S$. Let $C$ be the $i$-image of $\hat{\mathcal{C}}$ in $M_0$, consisting of finitely many $(R,\epsilon)$-good curves. The restriction $i|_{\hat{\mathcal{C}}}:\hat{\mathcal{C}}\to C$ is not precisely defined from the definition of $i$, since good pants/hamster wheels are only defined up to homotopy. We will define a map $E:\hat{\mathcal{C}}\to C$ as following. 

For any component $c$ of $\hat{\mathcal{C}}$, let $i(c)=\gamma \subset C$. If $c$ is a component of $\mathcal{C}$, the map $e$ does not descend to a map from $c$ to $\gamma$, since $c$ corresponds to two components of $\partial^{\mathcal{F}}\hat{\mathcal{A}}$. We have seen this phenomenon in the definition of $e$, and we describe it in an alternative way here. By considering basepoints of vectors, $e$ induces two maps $e_1,e_2:c\to \gamma$ and $d(e_1(x),e_2(x))<2\frac{\epsilon}{R}$ holds. We define the map $E$ on $c$ to be the average of $e_1$ and $e_2$, and $d(e_1(x),E(x)),d(e_2(x),E(x))<\frac{\epsilon}{R}$ hold. If $c$ is a component of $\partial S$, then  $e$ descends to a (unique) map from $c$ to $\gamma$, which is the restriction of $E$ on $c$.


Let $\tilde{\hat{\mathcal{C}}}$ and $\tilde{C}$ be preimages of $\hat{\mathcal{C}}$ and $C$ in universal covers $\tilde{S}\subset \mathbb{H}^2$ and $\tilde{M_0}=\mathbb{H}^3$ respectively. Then $i:S\looparrowright M_0$ and $E:\hat{\mathcal{C}}\to C$ induce a map $\tilde{E}:\tilde{\hat{\mathcal{C}}}\to \tilde{C}$. 
For each curve $c\subset \hat{\mathcal{C}}$, we give it an orientation, such that $S$ lies to the left of $c$ if $c\subset \partial S$. This orientation induces an orientation on $\tilde{\hat{\mathcal{C}}}$ and an orientation on $\tilde{C}$, so that $\tilde{E}$ is orientation preserving. 

For any $x\in \tilde{\hat{\mathcal{C}}}$, let $\tilde{c}$ be the component of $\tilde{\hat{\mathcal{C}}}$ containing $x$. We take the frame $(x,\vec{v},\vec{n})\in \partial^{\mathcal{F}}\tilde{\hat{\mathcal{A}}}$ based at $x$ such that $\vec{v}$ gives the orientation of $\tilde{c}$. A frame $(\tilde{E}(x),\vec{v}',\vec{n}')\in \partial^{\mathcal{F}}\tilde{\mathcal{A}}$ is obtained by parallel transporting $\tilde{e}(x,\vec{v},\vec{n})$ to $\tilde{E}(x)$, along an $\frac{\epsilon}{R}$-short geodesic segment in $\tilde{E}(\tilde{c})$. For any $x,y\in \tilde{\hat{\mathcal{C}}}$, we use $\overline{xy}$ to denote the oriented geodesic segment from $x$ to $y$. Under the coordinate of $T_x(\mathbb{H}^2)\subset T_x(\mathbb{H}^3)$ given by $(\vec{v},\vec{n},\vec{v}\times \vec{n})$, the tangent vector of $\overline{xy}$ based at $x$ gives a point $\Theta(x,y,\tilde{c})\in S^2=T_x^1(\mathbb{H}^3)$. Note that $\Theta(x,y,\tilde{c})\in S^2$ has zero third-coordinate, since $\tilde{\hat{\mathcal{C}}}$ lies in $\mathbb{H}^2$. Similarly, the tangent vector of $\overline{\tilde{E}(x)\tilde{E}(y)}$ based at $\tilde{E}(x)$ gives a point 
$\Theta(\tilde{E}(x),\tilde{E}(y),\tilde{E}(\tilde{c}))\in S^2=T^1_{\tilde{E}(x)}(\mathbb{H}^3)$, with respect to the coordinate  of $T_{\tilde{E}(x)}(\mathbb{H}^3)$ given by $(\vec{v}',\vec{n}',\vec{v}'\times \vec{n}')$. 

The following result gives our geometric estimate on $(R,\epsilon)$-nearly geodesic subsurfaces. 

\begin{theorem}\label{estimateonsurface}
For any $\delta\in (0,10^{-6})$, there are $\epsilon_0>0$ and $R_0>0$, such that for any $\epsilon\in (0,\epsilon_0)$, any $R>R_0$ and any $(R,\epsilon)$-good assembly $\mathcal{A}$, the following hold for $\tilde{E}:\tilde{\hat{\mathcal{C}}}\to \tilde{C}\subset \mathbb{H}^3$.
\begin{enumerate}
\item The map $\tilde{E}:\tilde{\hat{\mathcal{C}}}\to \tilde{C}$ is a $(1+1000\delta,\delta)$-quasi-isometric embedding. 
\item For any $x,y\in \tilde{\hat{\mathcal{C}}}$ such that $d(x,y)>10$, with $x\in \tilde{c}\subset \tilde{\hat{\mathcal{C}}}$, we have $$d_{S^2}\big(\Theta(x,y,\tilde{c}),\Theta(\tilde{E}(x),\tilde{E}(y),\tilde{E}(\tilde{c}))\big)<\delta.$$
\end{enumerate}
Moreover, for any $\tilde{E}':\tilde{\hat{\mathcal{C}}}\to \tilde{C}$ such that $d(\tilde{E}(x),\tilde{E}'(x))<\epsilon$ for any $x\in \tilde{\hat{\mathcal{C}}}$, the above properties hold with $\delta$ replaced by $2\delta$.
\end{theorem}

\begin{proof}
We first take $D=\delta^{-1}$ and apply Theorem \ref{localestimateKW} to obtain $C$ and $R_0$, then we take $\epsilon_0=\frac{\delta}{10C}$. For any $\epsilon\in (0,\epsilon_0)$, any $R>R_0$ and any $(R,\epsilon)$-good assembly $\mathcal{A}$, $e:\partial^{\mathcal{F}}\mathcal{A}\to  \partial^{\mathcal{F}}\hat{\mathcal{A}}$ is $\frac{\delta}{10}$-distorted at distance $\delta^{-1}$. 

We first prove the result for $\tilde{E}$.
For each geodesic $\tilde{c} \in \tilde{\hat{\mathcal{C}}}$, recall that it may correspond to two components of $\partial^{\mathcal{F}}\hat{\mathcal{A}}$. As for the map $e$, $\tilde{e}:\partial^{\mathcal{F}}\tilde{\hat{\mathcal{A}}}\to  \partial^{\mathcal{F}}\tilde{\mathcal{A}}$ gives two maps $\tilde{e}_1, \tilde{e}_2:\tilde{c} \to \tilde{E}(\tilde{c})$,
and we have 
\begin{align}\label{5.0}
d(\tilde{e}_1(x),\tilde{E}(x)),d(\tilde{e}_2(x),\tilde{E}(x))<\frac{\epsilon}{R}.
\end{align}

For any $x,y\in \hat{\mathcal{A}}$, we need to prove that 
\begin{align}\label{5.1}
(1+1000\delta)^{-1}d(x,y)-\delta<d(\tilde{E}(x),\tilde{E}(y))<(1+1000\delta)d(x,y)+\delta
\end{align} 
and estimate the angle difference in condition (2).

{\bf Case I}. Suppose that $d(x,y)<\delta^{-1}/2$ holds. Then we have 
\begin{align*}
& |d(\tilde{E}(x),\tilde{E}(y))-d(x,y)|\\
\leq \ & |d(\tilde{E}(x),\tilde{E}(y))-d(\tilde{e}_1(x),\tilde{e}_1(y))|+|d(\tilde{e}_1(x),\tilde{e}_1(y))-d(x,y)|\\
\leq \ & 2\frac{\epsilon}{R}+\frac{\delta}{10}<\frac{\delta}{2}.
\end{align*}
Here the second inequality follows from equation (\ref{5.0}) and Theorem \ref{localestimateKW}, the third inequality follows from the choice of $\epsilon_0$. So equation (\ref{5.1}) holds in this case.

Now we suppose that $d(x,y)>10$. We take a frame ${\bf p}=(x,\vec{v},\vec{n})\in \partial^{\mathcal{F}}\tilde{\hat{\mathcal{A}}}$ based at $x$ as above, and take a frame ${\bf q}\in \partial^{\mathcal{F}}\tilde{\hat{\mathcal{A}}}$ based at $y$. Since $d(x,y)<\delta^{-1}/2$, we have $d({\bf p},{\bf q})<\delta^{-1}$. By Theorem \ref{localestimateKW}, we have 
\begin{align}\label{5.2}
d\big(({\bf p}\to {\bf q}),(\tilde{e}({\bf p})\to \tilde{e}({\bf q}))\big)<\frac{\delta}{10}.
\end{align}

Let $x',y'$ be the base points of $\tilde{e}({\bf p})$ and $\tilde{e}({\bf q})$ respectively, and let $\Theta(x',y',\tilde{E}(\tilde{c}))\in S^2$ be defined with respect to the frame $\tilde{e}({\bf p})$. By equation (\ref{5.2}) and the fact that $d(x,y)>10$, a computation in hyperbolic geometry gives
\begin{align}\label{5.3}
d_{S^2}\big(\Theta(x,y,\tilde{c}),\Theta(x',y',\tilde{E}(\tilde{c}))\big)<\frac{\delta}{100}.
\end{align}

By equation (\ref{5.0}), we have $d(x',\tilde{E}(x)),d(y',\tilde{E}(y))<\frac{\epsilon}{R}$. By equation (\ref{5.2}), we have $d(x',y')>9$. Then computations in hyperbolic geometry gives
\begin{align}\label{5.4}
d_{S^2}\big(\Theta(x',y',\tilde{E}(\tilde{c})),\Theta(x',\tilde{E}(y),\tilde{E}(\tilde{c}))\big)<\frac{2\epsilon}{R}, \ d_{S^2}\big(\Theta(x',\tilde{E}(y),\tilde{E}(\tilde{c})),\Theta(\tilde{E}(x),\tilde{E}(y),\tilde{E}(\tilde{c}))\big)<\frac{4\epsilon}{R}.
\end{align}

Then equations (\ref{5.3}) and (\ref{5.4}) imply that $$d_{S^2}\big(\Theta(x,y, \tilde{c}),\Theta(\tilde{E}(x),\tilde{E}(y),\tilde{E}(\tilde{c}))\big)<\frac{\delta}{100}+6\frac{\epsilon}{R}<\frac{\delta}{10}.$$

{\bf Case II}. Now we suppose that $d(x,y)\geq \delta^{-1}/2$ holds. We take points $x_0=x,x_1,\cdots,x_n=y\in \mathbb{H}^3$ on $\overline{xy}$ following the orientation of $\overline{xy}$, such that $d(x_i,x_{i+1})\in [\frac{\delta^{-1}}{12},\frac{\delta^{-1}}{6}]$. Since $\tilde{\hat{C}}$ is $2$-dense in $\tilde{S}$, for each $x_i$, we take $y_i\in \tilde{\hat{C}}$ such that $d(y_i,x_i)<2$. We can and will take $y_0=x_0=x$ and $y_n=x_n=y$. Then we have 
$$d(y_i,y_{i+1})\in[\frac{\delta^{-1}}{12}-4,\frac{\delta^{-1}}{6}+4]\subset (\frac{\delta^{-1}}{16},\frac{\delta^{-1}}{4}).$$

Let $z_i$ be the point on $\overline{x_{i-1}x_{i+1}}$ closest to $y_i$, then we have $d(y_i,z_i)<2$ and $d(z_i,x_{i-1}),d(z_i,x_{i+1})>\frac{\delta^{-1}}{12}-2$. By hyperbolic sine law, we have
\begin{align*}
\angle z_iy_ix_{i-1} >\ &\sin \angle z_iy_ix_{i-1}=\frac{\sinh{d(x_{i-1},z_i)}}{\sqrt{\cosh^2{d(x_{i-1},z_i)}\cosh^2(d(y_i,z_i))-1}}\\
>\ & \frac{1}{2\cosh{d(y_i,z_i)}}> \frac{1}{2\cosh{2}}>0.132.
\end{align*}
Similarly, we have $\angle z_iy_ix_{i+1} >0.132$. So we have 
$$\angle{x_{i-1}y_ix_{i+1}}=\angle{x_{i-1}y_iz_i}+\angle{z_iy_ix_{i+1}}>0.264.$$ Moreover, since $d(x_{i-1},y_{i-1})<2$ and $d(x_{i-1},y_i)>d(x_{i-1},x_i)-2>\frac{\delta^{-1}}{12}-2$, we have 
$$ \angle{x_{i-1}y_iy_{i-1}}<2\sin{\angle{x_{i-1}y_iy_{i-1}}}\leq 2\frac{\sinh{d(x_{i-1},y_{i-1})}}{\sinh{d(x_{i-1},y_i)}}<2\frac{\sinh{2}}{\sinh{(\frac{\delta^{-1}}{12}-2)}}<0.002.$$ Similarly, we have $\angle{x_{i-1}y_iy_{i-1}}<0.002$. So we have $$\angle{y_{i-1}y_iy_{i+1}}>0.264-0.002\times 2=0.26.$$

For simplicity, we use $y_i'$ to denote $\tilde{E}(y_i)$.
Since $d(y_i,y_{i+1}), d(y_{i-1},y_{i+1})<\frac{\delta^{-1}}{2}$, by Case I, we have
\begin{align}\label{5.5}
|d(y_i',y_{i+1}')-d(y_i,y_{i+1})|<\frac{\delta}{2},\ |d(y_{i-1}',y_{i+1}')-d(y_{i-1},y_{i+1})|<\frac{\delta}{2}.
\end{align}
So we have 
\begin{align*}
& \cos{\angle{y_{i-1}'y_i'y_{i+1}'}}= \frac{\cosh{d(y_{i-1}',y_i')}\cdot \cosh{d(y_i',y_{i+1}')}-\cosh{d(y_{i-1}',y_{i+1}')}}{\sinh{d(y_{i-1}',y_i')}\cdot \sinh{d(y_i',y_{i+1}')}}\\
\leq\ & \frac{\cosh{(d(y_{i-1},y_i)+\frac{\delta}{2})}\cdot \cosh{(d(y_i,y_{i+1})+\frac{\delta}{2})}-\cosh{(d(y_{i-1},y_{i+1})-\frac{\delta}{2})}}{\sinh{(d(y_{i-1},y_i)-\frac{\delta}{2})}\cdot \sinh{(d(y_i,y_{i+1})-\frac{\delta}{2})}}\\
\leq\ & \frac{e^{\delta}\cosh{d(y_{i-1},y_i)}\cdot \cosh{d(y_i,y_{i+1})}-e^{-\delta}\cosh{d(y_{i-1},y_{i+1})}}{e^{-2\delta}\sinh{d(y_{i-1},y_i)}\cdot \sinh{d(y_i,y_{i+1})}}\\
=\ & (e^{3\delta}-e^{\delta})\frac{\cosh{d(y_{i-1},y_i)}\cdot \cosh{d(y_i,y_{i+1})}}{\sinh{d(y_{i-1},y_i)}\cdot \sinh{d(y_i,y_{i+1})}}+e^{\delta}\frac{\cosh{d(y_{i-1},y_i)}\cdot \cosh{d(y_i,y_{i+1})}-\cosh{d(y_{i-1},y_{i+1})}}{\sinh{d(y_{i-1},y_i)}\cdot \sinh{d(y_i,y_{i+1})}}\\
\leq\ & (e^{3\delta}-e^{\delta})\coth^2{\frac{\delta^{-1}}{16}}+e^{\delta}\cos{\angle{y_{i-1}y_iy_{i+1}}}< 0.97.
\end{align*} 
Then we have $\angle{y_{i-1}'y_i'y_{i+1}'}> 0.24$.

Now we apply Theorem \ref{lengthphase} (1) to the concatenation of geodesic segments $\overline{y_0'y_1'},\cdots, \overline{y_{n-1}'y_n'}$. For $L=\frac{\delta^{-1}}{32}$, we have
\begin{align*}
|d(y_0',y_n')-\sum_{i=1}^nd(y_{i-1}',y_i')+\sum_{i=1}^{n-1}I(\pi-\angle{y_{i-1}'y_i'y_{i+1}'})|<\frac{(n-1)e^{(-L+10\log{2})/2}}{L-\log{2}}.
\end{align*}
Then we get 
\begin{align*}
& |d(\tilde{E}(x),\tilde{E}(y))-d(x,y)|=|d(y_0',y_n')-\sum_{i=1}^nd(x_{i-1},x_i)|\\
\leq\ & |\sum_{i=1}^nd(y_{i-1}',y_i')-\sum_{i=1}^nd(x_{i-1},x_i)|+\sum_{i=1}^{n-1}I(\pi-\angle{y_{i-1}'y_i'y_{i+1}'})+\frac{(n-1)e^{(-L+10\log{2})/2}}{L-\log{2}}\\
\leq\ & n(4+\frac{\delta}{2})+(n-1)I(\pi-0.24)+(n-1)\frac{e^{(-L+10\log{2})/2}}{L-\log{2}}\\
\leq\ & n(4+\frac{\delta}{2}+I(\pi-0.24)+\frac{e^{(-L+10\log{2})/2}}{L-\log{2}})<10n \leq 200\delta \sum_{i=1}^n\frac{\delta^{-1}}{20}\\
\leq\ & 200\delta\sum_{i=1}^nd(x_{i-1},x_i)=200\delta \cdot d(x,y).
\end{align*}
Here the second inequality follows from $$|d(y_{i-1}',y_i')-d(x_{i-1},x_i)|\leq |d(y_{i-1}',y_i')-d(y_{i-1},y_i)|+d(x_{i-1},y_{i-1})+d(x_i,y_i)<\frac{\delta}{2}+4.$$
Then equation (\ref{5.1}) holds. 

Now we work on the angle comparison. By Case I, we have 
\begin{align}\label{5.6}
d_{S^2}\big(	\Theta(y_0,y_1,\tilde{c}),\Theta(y_0',y_1',\tilde{E}(\tilde{c}))\big)<\frac{\delta}{10}.
\end{align} 
Since $y_0=x_0$, $d(x_0,x_1)>\frac{\delta^{-1}}{12}$ and $d(x_1,y_1)<2$, we have 
\begin{align*}
\angle {x_1x_0y_1}<2\sin{\angle {x_1x_0y_1}}\leq 2\frac{\sinh{d(x_1,y_1)}}{\sinh{d(x_0,x_1)}}\leq 2\frac{\sinh{2}}{\sinh{\frac{\delta^{-1}}{12}}}<\frac{\delta}{10}.
\end{align*}
Since $x_1$ lies in the geodesic $\overline{xy}$, we have 
\begin{align}\label{5.7}
d_{S^2}\big(\Theta(y_0,y_1,\tilde{c}),\Theta(x,y,\tilde{c})\big)<\frac{\delta}{10}.
\end{align}
Then we apply Lemma \ref{directiondifference} to estimate the angle $\angle{y_1'y_0'y_n'}$ and get
\begin{align*}
\angle y_1'y_0'y_n'<\frac{\delta}{2}.
\end{align*}
Since $y_0'=\tilde{E}(x)$ and $y_n'=\tilde{E}(y)$, we have 
\begin{align}\label{5.8}
d_{S^2}\big(\Theta(y_0',y_1',\tilde{E}(\tilde{c})),\Theta(\tilde{E}(x),\tilde{E}(y),\tilde{E}(\tilde{c}))\big)< \frac{\delta}{2}.
\end{align}

Equations (\ref{5.6}), (\ref{5.7}) and (\ref{5.8}) imply the desired inequality 
$$d\big(\Theta(x,y,\tilde{c}),\Theta(\tilde{E}(x),\tilde{E}(y),\tilde{E}(\tilde{c}))\big)<\delta.$$

The moreover part of this result follows from a similar argument as above, and we skip it here.

\end{proof}

\bigskip

\subsection{Estimation on the ideal model of $Z$}\label{estimate1}

In this section, we first define a family of representations $\{\rho_t:\pi_1(Z)\to \text{Isom}_+(\mathbb{H}^3)\ |\ t\in[0,1]\}$ and a family of $\rho_t$-equivariant maps $\{\tilde{j}_t:\tilde{Z}\to \mathbb{H}^3\ |\ t\in[0,1]\}$, such that the following hold.
\begin{itemize}
\item $\tilde{j}_0$ maps each component of $\tilde{Z}\setminus \tilde{Z}^{(1)}$ to a totally geodesic subsurface of $\mathbb{H}^3$.
\item $\rho_1:\pi_1(Z)\to \text{Isom}_+(\mathbb{H}^3)$ and $j_*:\pi_1(Z)\to \pi_1(\tilde{M_0})$ are the same representation (up to conjugation), and $\tilde{j}_1:\tilde{Z}\to \mathbb{H}^3$ is the lifting of $j:Z\looparrowright M_0$ to universal covers (up to equivariant homotopy).
\end{itemize}
We will equip $\tilde{Z}$ with a metric and prove that $\tilde{j}_0:\tilde{Z}\to \mathbb{H}^3$ is both an embedding and a quasi-isometric embedding. We will also prove that the convex core of $\mathbb{H}^3/\rho_0(\pi_1(Z))$ is homeomorphic to the compact oriented $3$-manifold $\mathcal{Z}$ in Theorem \ref{pi1injectivity}.

The construction of the two families $\{\rho_t\}_{t\in[0,1]}$ and $\{\tilde{j}_t\}_{t\in [0,1]}$ is similar to the construction of corresponding objects in Construction 3.6 of \cite{Sun2}.

Note that when we constructed $j:Z\looparrowright M$ in Section \ref{constructZ1}, we get the following parameters. 
\begin{parameter}\label{parameter}
\begin{enumerate}
\item For each vertex $v_i\in Z^{(0)}$ and each edge $e_{ij}^Z\subset Z^{(1)}$ adjacent to $v_i$, the initial frame of $\mathfrak{s}_{ij}$ equals $F_{ijk}^{M_0}\cdot A_{ij}$ for some  $A_{ij}\in \text{SO}(3)$ that is $\frac{\epsilon}{10}$-close to $id\in \text{SO}(3)$ (Construction \ref{constructionZ1} (2) (b)).

\item For each edge $e_{ij}^{Z}\in Z^{(1)}$, the complex length of its associated $\partial$-framed segment $\mathfrak{s}_{ij}$ equals $R+\lambda_{ij}$ for some complex number $\lambda_{ij}$ with $|\lambda_{ij}|<\frac{\epsilon}{5}$ (Construction \ref{constructionZ1} (2) (a)).

\item For each decomposition curve $C$ of the surface $S_{ijk}$ that is not an inner cuff of any hamster wheel, the corresponding good curve has complex length $2R'+\xi_C$ for some complex number $\xi_C$ with $|\xi_C|<2\epsilon$ (the condition of $(R',\epsilon)$-good curves).

\item For each hamster wheel $H$ in some $S_{ijk}$, it has $R'$ rungs $r_{H,1},\cdots, r_{H,R}$, and these rungs divide two outer cuffs $c,c'$ to $R'$ geodesic segments $s_{H,1},\cdots,s_{H,R'}$ and $s_{H,1}',\cdots,s_{H,R'}'$ respectively. Then for any $i=1,\cdots,R'$, the complex distance between $c$ and $c'$ along $r_{H,i}$ is $R'-2\log{\sinh{1}}+\mu_{H,i}$ with $|\mu_{H,i}|<\frac{\epsilon}{R'}$ (equation (2.9.1) of \cite{KW1}). For any $i=1,\cdots,R'-1$, the complex distance between $r_{H,i}$ and $r_{H,i+1}$ along $s_{H,i}$ and $s_{H,i}'$ are $2+\nu_{H,i}$ and $2+\nu_{H,i}'$ respectively, with $|\nu_{H,i}|,|\nu_{H,i}'|<\frac{\epsilon}{R'}$ (equation (2.9.3) of \cite{KW1}). 

\item For each curve $C$ in the decomposition of $S_{ijk}$ or $\partial S_{ijk}$, the feet of its two adjacent good components differ by $1+\pi i+\eta_C$, for some complex number $\eta_C$ with $|\eta_C|<200\epsilon$ (the $(R',\epsilon)$-well-matched condition), and $|\eta_C|<\frac{\epsilon}{R}$ if formal feet are defined on both sides of $C$. Here if $C=\partial S_{ijk}$, the foot from the three-cornered annulus $A_{ijk}$ is the foot of the shortest geodesic segment from $\gamma_{ijk}$ to a preferred vertex $v_i$, as in Construction \ref{constructionZ}.
\end{enumerate}
\end{parameter}

So we have parameters $A_{ij}\in \text{SO}(3),\lambda_{ij},\xi_C,\mu_{H,i}, \nu_{H,i},\nu_{H,i}', \eta_C\in \mathbb{C}$ associated to $j:Z\looparrowright M_0$, and these parameters are very small with respect to metrics of $\text{SO}(3)$ and $\mathbb{C}$. 

The data in Parameter \ref{parameter} (3) and (4) determine shapes of all $(R',\epsilon)$-components in $Z$. For an $(R',\epsilon)$-hamster wheel, if the complex lengths of both outer cuffs are given in (3), then the data in (4) determines complex distances between $r_{H,R}$ and $r_{H,1}$ along $s_{H,R}$ and $s_{H,R}'$. In this case, the rungs divide a hamster wheel to $R'$ right-angled quadrilaterals, and their side lengths give complex lengths of all inner cuffs. If the complex length of one outer cuff is not given in (3), then it is the inner cuff of another hamster wheel (in an umbrella), and its complex length is determined by the data of that hamster wheel. Note that the dual graph of an umbrella (as a union of hamster wheels) is a tree, so the data in (3) and (4) determines shapes of all $(R',\epsilon)$-hamster wheels. Once the length of all $(R',\epsilon)$-good curves are determined, the shape of all $(R',\epsilon)$-good components and $(R_{ijk},R',\epsilon)$-good pants are determined.

For each $t\in [0,1]$, we take $t A_{ij}\in \text{SO}(3)$ and $t\lambda_{ij},t\xi_C,t\mu_{H,i}, t\nu_{H,i}, t\nu_{H,i}', t\eta_C\in \mathbb{C}$. Here $t A_{ij}$ denotes the image of $t\in [0,1]$ under the shortest geodesic $[0,1]\to \text{SO}(3)$ from $id$ to $A_{ij}$. These parameters give rise to a map $\tilde{j}_t:\tilde{Z}\to \mathbb{H}^3$, which can be defined by a developing argument as following. 

We first set up some notations. We use $\pi:\tilde{Z}\to Z$ to denote the universal cover of $Z$. Then each component of $\tilde{Z}^{(1)}=\pi^{-1}(Z^{(1)})\subset \tilde{Z}$ is a tree, and each component of $\tilde{Z}\setminus \tilde{Z}^{(1)}$ is simply connected. We use $\pi_H:\mathbb{H}^3\to M_0$ to denote the universal cover of $M_0$. We also take a subset $Z'$ of $Z$ as following. Here $Z'$ contains $Z^{(1)}$, $\partial S_{ijk}$, all decomposition curves in $S_{ijk}$ and following edges. For each three-cornered annulus $A_{ijk}$ in $Z$, $Z'$ includes edges from vertices of $A_{ijk}$ to $\partial S_{ijk}$. For each pair of pants in $Z$, $Z'$ includes its three seams. For each hamster wheel in $Z$, $Z'$ includes all short seams between its adjacent inner cuffs, and $2R'$ seams from its two outer cuffs to all inner cuffs. Then each component of $Z\setminus Z'$ is a disc.

\begin{construction}\label{family}
We first define $\tilde{j}_t:\tilde{Z}\to \mathbb{H}^3$ on $\tilde{Z}'=\pi^{-1}(Z')\subset \tilde{Z}$, by the following steps.
\begin{enumerate}

\item We start with a fixed vertex $\tilde{v}_i\in \tilde{Z}$ such that $\pi(\tilde{v}_i)=v_i$, and a point $p\in \mathbb{H}^3$ such that $\pi_H(p)=j(v_i)\in M_0$. Then we define $\tilde{j}_t(\tilde{v}_i)=p$, and we have an isometry of tangent spaces $(d\pi_H)_p:T_p\mathbb{H}^3\to T_{j(v_i)}(M_0)$. Let $\tilde{e}_{ij}^Z\subset \tilde{Z}$ be an edge from $\tilde{v}_i$ to another vertex $\tilde{v}_j$, such that it projects to $e_{ij}^Z\subset Z$. Then we map $\tilde{e}_{ij}^Z$ to a geodesic segment in $\mathbb{H}^3$ of length $R+t\text{Re}(\lambda_{ij})$ from $p$ to some $q\in \mathbb{H}^3$, such that it is tangent to the tangent vector of 
$$\tilde{\bf F}_{ijk}^{M_0}(t)=(d\pi_H)_p^{-1}({\bf F}_{ijk}^{M_0}\cdot tA_{ij}).$$
Let $C_{ij}(t)=
\begin{pmatrix}
1 & 0 & 0 \\
0 & \cos{(t\text{Im}(\lambda_{ij}))} & - \sin{(t\text{Im}(\lambda_{ij}))} \\
0 &  \sin{(t\text{Im}(\lambda_{ij}))} &  \cos{(t\text{Im}(\lambda_{ij}))}
\end{pmatrix}.$ We parallel transport $-\tilde{\bf F}_{ijk}^{M_0}(t)\cdot C_{ij}(t) \cdot (tA_{ji})^{-1}$ along this geodesic segment to obtain a frame $\tilde{\bf F}_{jik}^{M_0}(t)\in \text{SO}_q(\mathbb{H}^3)$. Then we take an isometry $T_q\mathbb{H}^3\to T_{j(v_j)}M_0$
that identifies $\tilde{\bf F}_{jik}^{M_0}(t)$ with ${\bf F}_{jik}^{M_0}\in \text{SO}_{j(v_j)}(M_0)$. By using this isometry, we define the map $\tilde{j}_t$ on edges adjacent to $\tilde{v}_j$ similarly. By applying this process inductively, $\tilde{j}_t$ is defined on the component $\tilde{W}\subset \tilde{Z}^{(1)}$ containing $\tilde{v}_i$. 

Note that when $t=0$, for any triangle $\Delta_{ijk}$ in $N$, the $\tilde{j}_0$ image of any component of $\pi^{-1}(e_{ij}^Z\cup e_{jk}^Z\cup e_{ki}^Z)\cap \tilde{W}$ lies in a hyperbolic plane in $\mathbb{H}^3$. When $t=1$, $\tilde{j}_1$ is exactly the restriction of $\tilde{j}:\tilde{Z}\to \mathbb{H}^3$ on $\tilde{W}$. Each $\tilde{j}_t$ also induces a representation $\rho_t^{\tilde{W}}:\pi_1(Z^{(1)})\to \text{Isom}_+(\mathbb{H}^3)$, such that $\tilde{j}_t|_{\tilde{W}}$ is $\rho_t^{\tilde{W}}$-equivariant.


\item For any line component $l\subset \tilde{Z}'$ of the preimage of $C=\partial S_{ijk}$ that is adjacent to $\tilde{W}$, it corresponds to a bi-infinite concatenation of edges in $\tilde{W}$, and $\tilde{j}_t$ maps this concatenation to a bi-infinite quasi-geodesic in $\mathbb{H}^3$. Then $\tilde{j}_t$ maps $l$ to the bi-infinite geodesic with same end points as the above quasi-geodesic, equivariant under the $\pi_1(C)$-action. For each edge $e$ of $\tilde{Z'}$ adjacent to $\tilde{W}$ that is mapped into a three-cornered annulus in $Z$, we map it to the shortest geodesic segment between the corresponding vertex in $\tilde{W}$ and $\tilde{j}_t(l)\subset \mathbb{H}^3$.

\item For each line component $l'\subset \tilde{Z'}$ in the preimage of an $(R',\epsilon)$-good curve $C'$ in $S_{ijk}$ and adjacent to $l$ in $\tilde{Z'}$, we map the seam $s$ between $l$ and $l'$ to the geodesic segment whose foot is the $(1+\pi i +t\eta_C)$-shift of the closest feet on $\tilde{j}_t(l)$. The complex length of $\tilde{j}_t(s)$ is determined by the parameters of this good component. Then the complex length of $\tilde{j}_t(s)$ determines the image of $l'$, which is a bi-infinite geodesic in $\mathbb{H}^3$. We let $\pi_1(C')$ acts on $\mathbb{H}^3$ by translation along $\tilde{j}_t(l')$, with translation length $2R'+\xi_{C'}$, and we define $\tilde{j}_t$ on $l'$ to be $\pi_1(C')$-equivarient. 

Then we apply this process inductively to define $\tilde{j}_t$ on one component of $\tilde{Z}\setminus \tilde{Z}^{(1)}$. For a hamster wheel, the complex lengths of its seams are determined by parameters in Parameter \ref{parameter}  (4) and complex lengths of its outer cuffs.

\item Then we can define $\tilde{j}_t$ on $\tilde{Z'}$ inductively by the process (and inverse process) in (3) and (4). It is easy to check that $\tilde{j}_t:\tilde{Z}\to \mathbb{H}^3$ induces a representation $\rho_t:\pi_1(Z)\to \text{Isom}(\mathbb{H}^3)$ and $\tilde{j}_t$ is $\rho_t$-equivariant. When $t=0$, $\tilde{j}_0$ maps each component of $\tilde{Z}\setminus \tilde{Z}^{(1)}$ to a totally geodesic subsurface of $\mathbb{H}^3$. When $t=1$, by our choice of parameters in Parameter \ref{parameter}, $\tilde{j}_1:\tilde{Z}\to \mathbb{H}^3$ is the lifting of $j:Z\looparrowright M_0$ to universal covers, up to equivariant homotopy.
\end{enumerate}

Since each component of $Z\setminus Z'$ is topologically a disc, we can further triangulate $Z$ and map each triangle in $Z$ to a totally geodesic triangle in $\mathbb{H}^3$, to get a $\rho_t$-equivariant map $\tilde{j}_t:\tilde{Z}\to \mathbb{H}^3$.
\end{construction}

By construction, $\tilde{j}_0$ maps each component of $\tilde{Z}\setminus \tilde{Z}^{(1)}$ to a totally geodesic subsurface in $\mathbb{H}^3$. The metrics on these totally geodesic subsurfaces pull-back to metrics on components of $\tilde{Z}\setminus \tilde{Z}^{(1)}$, and further induces a path metric on $\tilde{Z}$. This is the desired metric on $\tilde{Z}$ we will work with.

Now we state our result on $\tilde{j}_0:\tilde{Z}\to \mathbb{H}^3$.

\begin{proposition}\label{model}
Given the above metric on $\tilde{Z}$, the map $\tilde{j}_0:\tilde{Z}\to \mathbb{H}^3$ is both an embedding and a quasi-isometric embedding. The representation $\rho_0:\pi_1(Z)\to \text{Isom}_+(\mathbb{H}^3)$ is injective and the image is a convex cocompact subgroup of $\text{Isom}_+(\mathbb{H}^3)$.\\
Moreover, the convex core of $\mathbb{H}^3/\rho_0(\pi_1(Z))$ is homeomorphic to the compact oriented $3$-manifold $\mathcal{Z}$ in Theorem \ref{pi1injectivity}.
\end{proposition}

The proof of Proposition \ref{model} is similar to the proof of Proposition 4.1 of \cite{Sun2}, and all works will be done on $\mathbb{H}^3$. So the fact that $M_0$ has cusps and $Z$ contains hamster wheels do not cause much complication, and most of the proof in \cite{Sun2} is still valid here. Some terminologies defined in the proof will also be used in the next section.

\begin{proof}
The bulk of this proof will be devoted to prove that $\tilde{j}_0$ is a quasi-isometric embedding.

Let $Z''$ be the subset of $Z$ consists of $Z^{(1)}$, $\partial S_{ijk}$ and all decomposition curves in $S_{ijk}$. Since $Z''$ is $2$-dense in $Z$, to prove $\tilde{j}_0$ is a quasi-isometric embedding, we only need to prove a quasi-isometric inequality for two arbitrary points $x,y\in \tilde{Z''}\subset \tilde{Z}$. 

Let $\gamma$ be the shortest oriented path in $\tilde{Z}$ from $x$ to $y$. If $\gamma\setminus \{x,y\}$ does not intersect with $\tilde{Z}^{(1)}\subset \tilde{Z}$, then the restriction of $\tilde{j}_0$ on $\gamma$ is an isometry, and the proof is done. So we can assume that $\gamma\setminus \{x,y\}$ intersects with $\tilde{Z}^{(1)}$, and we denote the intersection points of $(\gamma\setminus \{x,y\})\cap \tilde{Z}^{(1)}$ by $x_1,x_2,\cdots, x_n$, which follow the orientation of $\gamma$. Here if $\gamma$ contains an edge of $\tilde{Z}^{(1)}$, we only record the initial and terminal points of this edge. The sequence $x_1,x_2,\cdots,x_n$ is called the {\it intersection sequence} of $\gamma$.

Based on the intersection sequence $x_1,x_2,\cdots,x_n$, we construct a sequence of points $y_1,y_2,\cdots,y_m$ called the {\it modified sequence} of $\gamma$. For any maximal subsequence $x_i,\cdots,x_{i+j}$ of the intersection sequence such that any two adjacent points have distance at most $R/64$, the edges of $\tilde{Z}^{(1)}$ containing these points must share a vertex $y\in \tilde{Z}^{(0)}$. Then we replace the subsequence  $x_i,\cdots,x_{i+j}$ of the intersection sequence by $y$. By doing this process for all such maximal subsequences of the intersection sequence, we get the {\it modified sequence} $y_1,\cdots,y_m$ of $\gamma$. Let $y_0=x, y_{m+1}=y$, and let $\gamma_i$ be the oriented geodesic segment from $y_i$ to $y_{i+1}$, with $i=0,1,\cdots,m$. The {\it modified path} $\gamma'$ of $\gamma$ is defined to be the concatenation of all $\gamma_i$ for $i=0,1,\cdots,m$. 

By Construction \ref{constructionZ} (2) and the fact that edges of $Z^{(1)}$ have length $R$, any two edges in $\tilde{Z}^{(1)}$ that do not share a vertex have distance at least $R/2$ (note that $R'>R$ holds). By the proof of Proposition 4.1 of \cite{Sun2}, we have $l(\gamma_i)>R/128$ for $i=1,\cdots,m-1$, and the angle between $\gamma_i$ and $\gamma_{i+1}$ in $\mathbb{H}^3$ is at least $\theta_0$ ($\theta_0>0$ was defined in Notation \ref{notation} (5).) Although the proof of Proposition 4.1 of \cite{Sun2} used some explicit angles, the proof is still valid here, since we have a lower bound $\phi_0>0$ for all angles of triangles in $N$ and all dihedral angles between triangles in $N$ (Notation \ref{notation} (4)).

Since the metrics on components of $\tilde{Z}\setminus \tilde{Z}^{(1)}$ are pull-back metrics of $\mathbb{H}^3$ via $\tilde{j}_0$, we always have
$$d_{\mathbb{H}^3}(\tilde{j}_0(x),\tilde{j}_0(y))\leq d_{\tilde{Z}}(x,y).$$
If both $l(\gamma_0),l(\gamma_m)>R/128$, we have $(m+1)R/128\leq \sum_{i=0}^ml(\gamma_i).$
By Lemma \ref{lengthphase} (1), we have 
\begin{align*}
&d_{\mathbb{H}^3}(\tilde{j}_0(x),\tilde{j}_0(y))\geq \sum_{i=0}^ml(\gamma_i)-m(I(\pi-\theta_0)+1)\\
\geq\ & \sum_{i=0}^ml(\gamma_i)-\frac{128(I(\pi-\theta_0)+1)}{R}\sum_{i=0}^ml(\gamma_i)
\geq (1-\frac{128(I(\pi-\theta_0)+1)}{R})d_{\tilde{Z}}(x,y).
\end{align*}
If only one of $l(\gamma_0)$ and $l(\gamma_m)$ is greater than $R/128$, say it is $l(\gamma_0)$, then we have $mR/128\leq \sum_{i=0}^ml(\gamma_i).$ By Lemma \ref{lengthphase} (1), we have 
\begin{align*}
&d_{\mathbb{H}^3}(\tilde{j}_0(x),\tilde{j}_0(y))\geq d(\tilde{j}_0(x),\tilde{j}_0(y_m))-l(\gamma_m)\geq \sum_{i=0}^{m-1}l(\gamma_i)-(m-1)(I(\pi-\theta_0)+1)-R/128\\
\geq\ & \sum_{i=0}^{m-1}l(\gamma_i)-\frac{128(I(\pi-\theta_0)+1)}{R}\sum_{i=0}^{m}l(\gamma_i)-R/128\geq (1-\frac{128(I(\pi-\theta_0)+1)}{R})\sum_{i=0}^{m}l(\gamma_i)-R/64\\
\geq\ &  (1-\frac{128(I(\pi-\theta_0)+1)}{R})d_{\tilde{Z}}(x,y)-R/64.
\end{align*}
If $l(\gamma_0),l(\gamma_m)<R/128$, then we have $(m-1)R/128\leq \sum_{i=0}^ml(\gamma_i)$. By Lemma \ref{lengthphase} (1) again, we have 
\begin{align*}
&d_{\mathbb{H}^3}(\tilde{j}_0(x),\tilde{j}_0(y))\geq d(\tilde{j}_0(y_1),\tilde{j}_0(y_m))-l(\gamma_0)-l(\gamma_m)\geq \sum_{i=1}^{m-1}l(\gamma_i)-(m-2)(I(\pi-\theta_0)+1)-R/64\\
\geq\ & \sum_{i=1}^{m-1}l(\gamma_i)-\frac{128(I(\pi-\theta_0)+1)}{R}\sum_{i=0}^{m}l(\gamma_i)-R/64\geq (1-\frac{128(I(\pi-\theta_0)+1)}{R})\sum_{i=0}^{m}l(\gamma_i)-R/32\\
\geq\ & (1-\frac{128(I(\pi-\theta_0)+1)}{R})d_{\tilde{Z}}(x,y)-R/32.
\end{align*}
So we have proved that $\tilde{j}_0:\tilde{Z}\to \mathbb{H}^3$ is a quasi-isometric embedding, with quasi-isometric constants $((1-\frac{128(I(\pi-\theta_0)+1)}{R})^{-1},\frac{R}{32})$. Since $\tilde{j}_0:\tilde{Z}\to \mathbb{H}^3$ is equivariant with respect to the representation $\rho_0:\pi_1(Z)\to \text{Isom}_+(\mathbb{H}^3)$ and $\pi_1(Z)$ is torsion free, we know that $\rho_0$ is injective and $\rho_0(\pi_1(Z))$ is a convex cocompact subgroup of $\text{Isom}_+(\mathbb{H}^3)$.

In the above case-by-case argument, we also obtain the following inequality:
\begin{align}\label{5}
& \sum_{i=0}^ml(\gamma_i)\leq (1-\frac{128(I(\pi-\theta_0)+1)}{R})^{-1}(d_{\tilde{Z}}(x,y)+R/32).
\end{align}

For large enough $R$ (say $R>2^{11}(I(\pi-\theta_0)+1)$), the above estimates imply that $\tilde{j}_0(x)\ne \tilde{j}_0(y)$ if $d_{\tilde{Z}}(x,y)>R/30$. If $d_{\tilde{Z}}(x,y)\leq R/30$, then either $x$ and $y$ lie on the same component of $\tilde{Z}\setminus \tilde{Z}^{(1)}$, or they lie on two different components of $\tilde{Z}\setminus \tilde{Z}^{(1)}$ whose closures have nontrivial intersection. Then it is clear that $\tilde{j}_0(x)\ne \tilde{j}_0(y)$ in this case, and we finish the proof that $\tilde{j}_0:\tilde{Z}\to \mathbb{H}^3$ is injective.

So $\tilde{j}_0:\tilde{Z}\to \mathbb{H}^3$ is injective and maps each component of $\tilde{Z}\setminus \tilde{Z}^{(1)}$ to a totally geodesic subsurface in $\mathbb{H}^3$. By the construction of $Z$, $\tilde{j}_0(\tilde{Z})$ has a $\rho_0(\pi_1(Z))$-invariant neighborhood $\mathcal{N}(\tilde{Z})$ in $\mathbb{H}^3$ such that $\mathcal{N}(\tilde{Z})/\rho_0(\pi_1(Z))$ is homeomorphic to the compact oriented $3$-manifold $\mathcal{Z}$. Also note that $\mathcal{N}(\tilde{Z})/\rho_0(\pi_1(Z))$ is a compact submanifold of $\mathbb{H}^3/\rho_0(\pi_1(Z))$ such that the inclusion induces an isomorphism on $\pi_1$. Since all boundary components of $\mathcal{Z}$ are incompressible, each component of 
$$\big(\mathbb{H}^3/\rho_0(\pi_1(Z))\big)/\big(\mathcal{N}(\tilde{Z})/\rho_0(\pi_1(Z))\big)$$
is homeomorphic to the product of a surface and $(0,\infty)$, so $\mathbb{H}^3/\rho_0(\pi_1(Z))$ is homeomorphic to $\mathcal{Z}\setminus \partial \mathcal{Z}$. Since $\rho_0(\pi_1(Z))$ is convex cocompact, the convex core of $\mathbb{H}^3/\rho_0(\pi_1(Z))$ is homeomorphic to $\mathcal{Z}$. Note that all homeomorphisms in this paragraph preserve orientations on these oriented manifolds.
\end{proof}

\bigskip

\subsection{$\tilde{j}_t$ is a quasi-isometric embedding}\label{qi1}
In this section, we prove that each $\tilde{j}_t:\tilde{Z}\to \mathbb{H}^3$ is a quasi-isometric embedding, and finish the proof of Theorem \ref{pi1injectivity}.

We first state the following lemma that estimates the geometry of $\tilde{j}_t$ on components of $\tilde{Z}\setminus \tilde{Z}^{(1)}$. Note that Theorem \ref{estimateonsurface} estimates the geometry of maps (on universal covers) of $(R',\epsilon)$-nearly geodeisc subsurfaces with boundaries. However, each component of $Z\setminus Z^{(1)}$ also contains three-cornered annuli and $(R_{ijk},R',\epsilon)$-good pants (as in Construction \ref{constructionZ}). So we need a version of Theorem \ref{estimateonsurface} for components of $\tilde{Z}\setminus \tilde{Z}^{(1)}$.

\begin{lemma}\label{smallangle}
For any $\delta\in(0,10^{-6})$, there exists $\epsilon_0>0$ and $R_0>0$, such that for any positive numbers $\epsilon\in (0,\epsilon_0)$, $R>R_0$ and any positive integer $R'$ greater than all of the $R_{ijk}$, the following statement holds. 

If $\{\tilde{j}_t:\tilde{Z}\to \mathbb{H}^3\ |\ t\in [0,1]\}$ is constructed with respect to $\epsilon,R$ and $R'$, then for any $t\in[0,1]$ and any $x,y$ lying in the closure of a component $C\subset \tilde{Z}\setminus \tilde{Z}^{(1)}$ such that $x\in \partial C$, we have:
\begin{align}\label{a}
\frac{1}{2}d_{\mathbb{H}^3}(\tilde{j}_t(x),\tilde{j}_t(y))\leq d_{\tilde{Z}}(x,y)\leq 2d_{\mathbb{H}^3}(\tilde{j}_t(x),\tilde{j}_t(y)).
\end{align}
Moreover, let $e$ be the edge in $\tilde{Z}^{(1)}$ containing $x$ (with a preferred orientation), if $d(x,y)\geq 100$, then we have
\begin{align}\label{b}
d_{S^2}(\Theta(x,y,e),\Theta(\tilde{j}_t(x),\tilde{j}_t(y),\tilde{j}_t(e)))<10\delta.
\end{align}
Here $\Theta(x,y,e)$ denotes the point in $S^2$ determined by the tangent vector of $\overline{xy}$ in $\mathbb{H}^3$, with respect to a coordinate of $T^1_x(\mathbb{H}^3)$ given by a frame ${\bf p}=(x,\vec{v},\vec{n})$, such that $\vec{v}$ is tangent to $e$ and $\vec{n}$ is tangent to $C$ (points inward). Similarly, $\Theta(\tilde{j}_t(x),\tilde{j}_t(y),\tilde{j}_t(e))$ is defined with respect to a frame based at $\tilde{j}_t(x)$, with first vector tangent to $\tilde{j}_t(e)$, and the second vector is $\epsilon$-close to be tangent to $\tilde{j}_t(C)$ (points inward).
\end{lemma}

The proof of Lemma \ref{smallangle} is left to the readers, and we only describe the idea of proof here.

At first, Lemmas 4.6 and 4.7 of \cite{Sun2} estimate the geometry of nearly geodesic $4$-gons (with two right angles) in $\mathbb{H}^3$, and estimate the geometry of $\tilde{j}_t$ on the preimage of three-cornered annuli in $C$. 
For the $(R_{ijk},R',\epsilon)$-good pants $\Pi_i \subset Z\setminus Z^{(1)}$, a computation on hyperbolic geometry as in Appendix A.9 of \cite{KW1} estimates the geometry of $\tilde{j}_t$ on the preimage of $\Pi_i$ in $C$. Besides the above estimates on these exceptional pieces, the moreover part of Theorem \ref{estimateonsurface} estimates the geometry of $\tilde{j}_t$ on the preimage of $(R',\epsilon)$-nearly geodesic subsurfaces in $C$.  Adopting these estimates on pieces, we can argue as in Proposition 4.8 of \cite{Sun2} to prove Lemma \ref{smallangle}.

Alternatively, for the component $C\subset \tilde{Z}\setminus \tilde{Z}^{(1)}$, we can associate it with subsets of frames in $\text{SO}(\mathbb{H}^2)$ and $\text{SO}(\mathbb{H}^3)$ as in Section \ref{estimatesurface}. We can first prove a version of Theorem \ref{localestimateKW} for these frames as in \cite{KW1}. Then we can prove Lemma \ref{smallangle} by running the proof of Theorem \ref{estimateonsurface}.

The following proposition proves that each $\tilde{j}_t$ is a quasi-isometric embedding.

\begin{proposition}\label{quasiisometry}
If $\epsilon>0$ is small enough and $R>0$ is large enough, for any $t\in [0,1]$, the map $\tilde{j}_t:\tilde{Z}\to \mathbb{H}^3$ is a quasi-isometric embedding.
\end{proposition}


\begin{proof}
We choose $0<\delta<\min{\{10^{-6}, \frac{\theta_0}{100}\}}$, where $\theta_0$ is defined in Notation \ref{notation} (5). Then we take $\epsilon>0$ and $R>0$ so that Lemma \ref{smallangle} holds for $\delta$.

Since $\tilde{Z}''$ (defined in the proof of Proposition \ref{model}) is $2$-dense in $\tilde{Z}$, it suffices to prove that $\tilde{j}_t|:\tilde{Z}''\to \mathbb{H}^3$ is a quasi-isometric embedding. For any $x,y\in \tilde{Z}''$, we take the shortest path $\gamma$ in $\tilde{Z}$ from $x$ to $y$. If $x$ and $y$ lie in the same component of $\tilde{Z}\setminus \tilde{Z}^{(1)}$, the proof follows from Lemma \ref{smallangle}. So we assume that they lie in different components of $\tilde{Z}\setminus \tilde{Z}^{(1)}$.
 Let $y_1,\cdots,y_m$ be the modified sequence of $\gamma$, and let $\gamma'$ be the modified path of $\gamma$. 

Then we have 
\begin{align*}
& d_{\mathbb{H}^3}(\tilde{j}_t(x),\tilde{j}_t(y))\leq \sum_{i=0}^md_{\mathbb{H}^3}(\tilde{j}_t(y_i),\tilde{j}_t(y_{i+1}))\leq 2\sum_{i=0}^md_{\tilde{Z}}(y_i,y_{i+1}) \\
\leq\ & 2(1-\frac{128(I(\pi-\theta_0)+1)}{R})^{-1}(d_{\tilde{Z}}(x,y)+R/32).
\end{align*}
Here the second inequality follows from equation (\ref{a}), and the last inequality follows from equation (\ref{5}).

On the other hand, since $d_{\tilde{Z}}(y_i,y_{i+1})>R/128$ for each $i=1,\cdots,m-1$, we still have $(m-1)R/128\leq \sum_{i=0}^{m}d_{\tilde{Z}}(y_i,y_{i+1})$. Moreover, by equation (\ref{a}), we have $d_{\mathbb{H}^3}(\tilde{j}_t(y_i),\tilde{j}_t(y_{i+1}))>R/256$. Since the angle between $\overline{\tilde{j}_0(y_{i-1})\tilde{j}_0(y_i)}$ and $\overline{\tilde{j}_0(y_i)\tilde{j}_0(y_{i+1})}$ is at least $\theta_0$, by equation (\ref{b}), the angle between  $\overline{\tilde{j}_t(y_{i-1})\tilde{j}_t(y_i)}$ and $\overline{\tilde{j}_t(y_i)\tilde{j}_t(y_{i+1})}$ is at least $\theta_0-20\delta>\theta_0/2$. 

By the case-by-case argument in the proof of Proposition \ref{model}, we have
\begin{align*}
& d_{\mathbb{H}^3}(\tilde{j}_t(x),\tilde{j}_t(y))\geq \sum_{i=0}^md_{\mathbb{H}^3}(\tilde{j}_t(y_i),\tilde{j}_t(y_{i+1}))-m(I(\pi-\theta_0/2)+1)-R/32\\
\geq\ & \frac{1}{2}\sum_{i=0}^md_{\tilde{Z}}(y_i,y_{i+1})-m(I(\pi-\theta_0/2)+1)-R/32\\
\geq\ & \frac{1}{2}\sum_{i=0}^md_{\tilde{Z}}(y_i,y_{i+1})-(\frac{128}{R}\sum_{i=0}^md_{\tilde{Z}}(y_i,y_{i+1})+1)(I(\pi-\theta_0/2)+1)-R/32\\
=\ & \big(\frac{1}{2}-\frac{128(I(\pi-\theta_0/2)+1)}{R}\big)\sum_{i=0}^md_{\tilde{Z}}(y_i,y_{i+1})-(R/32+I(\pi-\theta_0/2)+1)\\
\geq\ & \big(\frac{1}{2}-\frac{128(I(\pi-\theta_0/2)+1)}{R}\big)d_{\tilde{Z}}(x,y)-(R/32+I(\pi-\theta_0/2)+1)
\end{align*}
Here the second inequality follows from equation (\ref{a}) and the third inequality follows from our upper bound of $m$.

We finish the proof that $\tilde{j}_t:\tilde{Z}\to\mathbb{H}^3$ is a quasi-isometric embedding.
\end{proof}

Now we are ready to prove Theorem \ref{pi1injectivity}, the main result of this section.

\begin{proof}[Proof of Theorem \ref{pi1injectivity}]
By Proposition \ref{quasiisometry}, each $\tilde{j}_t:\tilde{Z}\to \mathbb{H}^3$ is a quasi-isometric embedding. Since $\pi_1(Z)$ is torsion free, each representation $\rho_t:\pi_1(Z)\to \text{Isom}_+(\mathbb{H}^3)$ is injective, and $\rho_t(\pi_1(Z))<\text{Isom}_+(\mathbb{H}^3)$ is a convex cocompact subgroup. In particular, $\rho_1(\pi_1(Z))=j_*(\pi_1(Z))$ is a convex cocompact subgroup of $\pi_1(M_0)<\text{Isom}_+(\mathbb{H}^3)$.

Moreover, $\{\rho_t(\pi_1(Z))\ |\ t\in [0,1]\}$ forms a continuous family of convex cocompact subgroups of $\text{Isom}_+(\mathbb{H}^3)$. So the convex core of $\mathbb{H}^3/j_*(\pi_1(Z))=\mathbb{H}^3/\rho_1(\pi_1(Z))$ is homeomorphic to the convex core of $\mathbb{H}^3/\rho_0(\pi_1(Z))$, which is homeomorphic to $\mathcal{Z}$ (as oriented manifolds) by Proposition \ref{model}.
\end{proof}

\end{document}